\pdfoutput=1
\documentclass[twoside]{article}
\usepackage[letterpaper, left=3.5cm, right=3.5cm, top=4cm, bottom=2.5cm]{geometry}
\usepackage{graphicx}
\usepackage{amsmath,amssymb,amsthm}
\usepackage{authblk}
\usepackage[utf8]{inputenc}
\usepackage{microtype}
\usepackage{mathrsfs} 
\usepackage{enumitem}
\usepackage{calc}
\usepackage{esint}
\usepackage{fancyhdr}
\usepackage{xcolor}
\usepackage[labelfont = bf]{caption}
\usepackage{doi}
\usepackage{subfigure}
\usepackage[numbers]{natbib}
\usepackage{float}
\usepackage{stmaryrd}
\usepackage{mathtools}
\usepackage{booktabs}
\usepackage{colortbl}
\usepackage{tabulary}
\usepackage{diagbox}
\usepackage{caption}
\usepackage{cleveref}
\usepackage{commath}
\usepackage{multirow}
\usepackage{relsize}

\usepackage{physics}
\usepackage{amsmath}
\usepackage{tikz}
\usepackage{pgfplots}
\pgfplotsset{compat=1.18}
\usepackage{mathdots}
\usepackage{yhmath}
\usepackage{cancel}
\usepackage{color}
\usepackage{siunitx}
\usepackage{array}
\usepackage{multirow}
\usepackage{amssymb}
\usepackage{gensymb}
\usepackage{tabularx}
\usepackage{extarrows}
\usepackage{booktabs}
\usetikzlibrary{fadings}
\usetikzlibrary{patterns}
\usetikzlibrary{shadows.blur}
\usetikzlibrary{shapes}

\newtheorem{theorem}{Theorem}[section]
\numberwithin{equation}{section}
\newtheorem{proposition}[theorem]{Proposition}

\newtheorem{corollary}[theorem]{Corollary}
\newtheorem{remark}[theorem]{Remark}
\newtheorem{lemma}[theorem]{Lemma}

\newtheorem{algorithm}[theorem]{Algorithm}
\newtheorem{assumption}[theorem]{Assumption}

\usepackage{titlesec}
\titleformat{\section}{\normalfont\scshape\centering}{\thesection.}{0.5em}{}
\titleformat*{\subsection}{\itshape}
\titleformat*{\subsubsection}{\itshape}

 \providecommand{\keywords}[1]
 {
 	{\small\emph{Keywords:} #1}
 }
 \providecommand{\MSC}[1]
 {
 	{\small\emph{AMS MSC (2020):~~} #1}
 }

\definecolor{denim}{rgb}{0.08, 0.38, 0.74}
\definecolor{byzantium}{rgb}{0.44, 0.16, 0.39} 
\definecolor{shamrockgreen}{rgb}{0.0, 0.62, 0.38} 

\usepackage{hyperref}
\hypersetup{
	colorlinks=true,
	linkcolor=denim,
	citecolor = shamrockgreen,
	filecolor=magenta, 
	urlcolor=byzantium,
}

\begin{document}
	\setlength{\abovedisplayskip}{5.5pt}
	\setlength{\belowdisplayskip}{5.5pt}
	\setlength{\abovedisplayshortskip}{5.5pt}
	\setlength{\belowdisplayshortskip}{5.5pt}

	\title{\vspace{-17.5mm}A Finite Element Approximation of an Optimal Insulation Problem with Convective Heat Transfer\thanks{This work is partially supported by the Office of Naval Research (ONR) under Award NO: N00014-24-1-2147, NSF grant DMS-2408877, the Air Force Office of Scientific Research (AFOSR) under Award NO: FA9550-25-1-0231.}
    }
	\author[1]{Harbir Antil\thanks{Email: \url{hantil@gmu.edu}}}
	\author[2]{Alex Kaltenbach\thanks{Email: \url{kaltenbach@math.tu-berlin.de}}}
	\author[3]{Keegan L.~A. Kirk\thanks{Email: \url{kkirk6@gmu.edu}}}
	\date{\today\vspace{-3.5mm}} 
	\affil[1,2,3]{\small{Department of Mathematical Sciences and the Center for Mathematics and Artificial Intelligence (CMAI), George Mason University, Fairfax, VA 22030, USA.}}
	\affil[2]{\small{Institute of Mathematics, Technical University of Berlin, Stra\ss e des 17.\ Juni 135, 10623 Berlin}\vspace{-2mm}}
	\maketitle

	\pagestyle{fancy}
	\fancyhf{}
	\fancyheadoffset{0cm}
	\addtolength{\headheight}{-0.25cm}
	\renewcommand{\headrulewidth}{0pt} 
	\renewcommand{\footrulewidth}{0pt}
	\fancyhead[CO]{\textsc{FE Approximation of an Optimal Insulation Problem}}
	\fancyhead[CE]{\textsc{H. Antil, A. Kaltenbach, and K.  L.~A. Kirk}}
	\fancyhead[R]{\thepage}
	\fancyfoot[R]{}
	
	\begin{abstract} 
        A finite element discretization of an optimal insulation problem with convective heat transfer is considered. The model is formulated as a non-smooth, two-variable convex minimization problem. It accounts for the temperature distribution in a thermally conducting~body~${\Omega\hspace{-0.1em}\subseteq\hspace{-0.1em}\mathbb{R}^d}$, with~${d\hspace{-0.1em}\in\hspace{-0.1em} \{2,3\}}$,~and the distribution of a given amount of insulation material  on an insulated boundary~part~$\Gamma_I\subseteq \partial\Omega$.

        The surface integral over the insulated boundary $\Gamma_I$ is approximated~by~a~\mbox{mass-lumping} quadrature that preserves the structure of the continuous setting and,~in~\mbox{particular}, yields discrete optimality conditions mirroring their continuous counterparts. Well-posedness, stability, and weak convergence of discrete solutions to the continuous~ones~are~established.

        Furthermore, a block coordinate descent algorithm for the computation of the discrete solutions is formulated and its linear convergence is derived. Under suitable regularity assumptions, uniform $L^\infty(\Gamma_I)$-bounds and \textit{a priori} error estimates for both the temperature distribution and the distribution of a given amount of insulation material are obtained. 
Numerical experiments are carried out that confirm the predicted error decay rates and
demonstrate the method in a qualitative three-dimensional test on a
realistic spacecraft crew module capsule geometry with idealized reentry-heating Robin data.
    \end{abstract} 
	
	
	\MSC{65N30; 65N15; 65N12; 49M05; 35J25; 80A19\vspace{-1mm}}

    \medskip

	\section{Introduction}\label{sec:introduction}\thispagestyle{empty}\enlargethispage{17.5mm}\vspace{-0.5mm}

    \hspace{5mm}Thermal \hspace{-0.1mm}insulation \hspace{-0.1mm}is \hspace{-0.1mm}commonly \hspace{-0.1mm}used \hspace{-0.1mm}to \hspace{-0.1mm}regulate \hspace{-0.1mm}heat \hspace{-0.1mm}exchange \hspace{-0.1mm}between \hspace{-0.1mm}a \hspace{-0.1mm}thermally~\hspace{-0.1mm}\mbox{conducting}\linebreak body and its 
surrounding environment in applications ranging from buildings and electronic devices to air- and spacecraft.
 If only a limited amount of insulating material is available,~\textit{e.g.}, due to 
geometric, weight, or cost constraints, the question of its optimal distribution arises naturally 
and has attracted considerable attention in both theoretical and applied~research~(\textit{cf}.~\cite{Claesson1980,DellaPietraOliva2025}). 

While many classical models assume that heat transfer across the boundary of the thermally conducting body is governed by thermal 
conduction, numerous practical situations are dominated by \hspace{-0.15mm}convective \hspace{-0.15mm}heat \hspace{-0.15mm}transfer, \hspace{-0.15mm}which \hspace{-0.15mm}is 
\hspace{-0.15mm}commonly \hspace{-0.15mm}modeled \hspace{-0.15mm}by \hspace{-0.15mm}Robin-type~\hspace{-0.15mm}\mbox{boundary}~\hspace{-0.15mm}\mbox{conditions}.
In this setting,~we~seek~the~\textit{`best'} distribution of a given amount of insulating material 
attached to parts of a thermally conducting body. 

To this end, we study a convex minimization problem~proposed~by~Della~Pietra~\textit{et~al.}~(\textit{cf}.~\cite{PietraNitschScalaTrombetti2021}):

    Let $\Omega\subseteq \mathbb{R}^d$, $d\in \{2,3\}$, be a bounded polyhedral Lipschitz domain
    representing the \textit{thermally conducting body}, with (topological) boundary $\partial\Omega$ decomposed into an \textit{insulation part} (\textit{i.e.}, $\Gamma_I$)  (to which the insulating material is attached) and a Neumann part (\textit{i.e.}, $\Gamma_N$). Moreover, let
    $m>0$ be a given \textit{amount of insulating material}, $f\in L^2(\Omega)$ a given 
    \textit{heat source density}, $g\in  (H^{\smash{\frac{1}{2}}}(\Gamma_N))^*$ a~given~\textit{heat flux}, $u_\infty\in H^1(\mathbb{R}^d\setminus \overline{\Omega})$ a given \textit{ambient temperature}, $\kappa>0$ the \textit{thermal conductivity}, and $\beta>0$ a given \textit{heat transfer coefficient}.\enlargethispage{1mm}
    
    Then, we seek a \textit{temperature distribution} $u\hspace{-0.15em}\in \hspace{-0.15em}H^1(\Omega)$ and a \textit{distribution function}~(of~the~insula\-tion material in direction of the normal vector field) $\smash{\widetilde{\mathtt{d}}}\in \mathcal{H}_m$, where\vspace{-0.5mm}
    \begin{align*}
         \mathcal{H}_m\coloneqq \smash{\bigl\{ \eta\in L^2(\Gamma_{I})\mid \eta\ge 0\text{ a.e.\ on }\Gamma_{I}\,,\; (\eta,1)_{\Gamma_{I}}=m\bigr\}}\,,\\[-6mm]\notag
    \end{align*}
    that minimize the energy functional $E\colon\hspace{-0.1em} H^1(\Omega)\times \mathcal{H}_m\hspace{-0.1em}\to\hspace{-0.1em} \mathbb{R}$, for every $(v,\eta)\hspace{-0.1em}\in\hspace{-0.1em} H^1(\Omega)\times \mathcal{H}_m$~\mbox{defined}~by\vspace{-0.5mm}
    \begin{align}\label{def:E_intro}
        E(v,\eta)\coloneqq  
            \tfrac{\kappa}{2}\|\nabla v\|_{\Omega}^2 + \tfrac{\beta}{2} \|(1+\beta \eta)^{-\smash{\frac{1}{2}}}(v-u_{\infty})\|_{\Gamma_{I}}^2-(f,v)_{\Omega}-\langle g,v\rangle_{\Gamma_N}+I_{\mathcal{H}_m}(\eta)\,.\\[-6mm]\notag
    \end{align} 

    \subsection{Main contributions}\vspace{-1mm}
    
    \hspace{5mm}The main contributions of the present paper can be summarized as follows:

\begin{enumerate}[leftmargin=*, noitemsep,font=\itshape,topsep=2pt]

\item 
A finite element discretization of the non-smooth convex minimization problem, given via the minimization of \eqref{def:E_intro}, is introduced, in which the insulation boundary integral is approximated by means of a mass-lumping quadrature. 
This quadrature preserves the structure of the continuous problem  and yields discrete optimality conditions that closely mirror~the~\mbox{continuous}~\mbox{setting};

\item 
Weak convergence of discrete solutions to a solution of the continuous problem~is~established~by passing to the limit in the discrete optimality conditions;

\item 
A block coordinate descent algorithm for the computation of discrete solutions based on the discrete optimality system is formulated, and its well-posedness and linear convergence~are~proved;

\item 
Uniform $L^\infty(\Gamma_I)$-bounds for both the continuous and  discrete temperature distribution and the distribution function are derived using continuous and discrete Stampacchia truncation arguments, providing a key ingredient for the subsequent error analysis;

\item 
Under appropriate regularity assumptions, including fractional regularity, 
\textit{a~priori} error estimates for both the temperature distribution and the distribution function are established;

\item 
An $H^1(\Gamma_I)$-stable quasi-interpolation operator and a positivity- and mean-preserving projection are constructed as  technical tools required for the error analysis;

\item   Numerical experiments are presented that confirm the optimality of the predicted error decay
rates and illustrate the practical applicability of the proposed approach, including a 
three-dimensional capsule example with an idealized reentry-heating
Robin datum.\vspace{-0.5mm}

\end{enumerate}

\subsection{Related contributions}\vspace{-1mm}

\hspace{5mm}
The existing literature primarily addresses either the theoretical analysis of the minimization 
of~\eqref{def:E_intro} or the numerical analysis of related eigenvalue-type problems, typically 
in the setting $\partial\Omega\in C^{1,1}$ with $\Gamma_I=\partial\Omega$ and predominantly for 
heat transfer governed by conduction:\vspace*{-0.5mm}

\begin{itemize}
[noitemsep,topsep=2pt,leftmargin=!,labelwidth=\widthof{6.},font=\itshape]

\item \textit{Theoretical analysis:} There is a rich literature deriving energy functionals related to \eqref{def:E_intro} under 
conductive heat transfer as asymptotic limits as the thickness of the insulating~layer~tends~to~zero (see, \textit{e.g.}, \cite{AcerbiButtazzo1986} when $\partial\Omega\in C^{1,1}$, 
\cite{AKK2025_modelling}  when $\partial\Omega\in C^{0,1}$ is piece-wise flat,~and~\cite{CristoforoniNitschTrombetti2026} for a recent overview of the literature). 
The energy functional~\eqref{def:E_intro}~\mbox{itself}~was~derived~in~\cite{PietraNitschScalaTrombetti2021} 
for~${\partial\Omega\in C^{1,1}}$ and has recently been complemented by the case 
$\partial\Omega\in C^{0,1}$ with~\mbox{piece-wise}~flat~\mbox{boundary}~in~\cite{AKK2025ConvectionModelling2025};

\item \textit{Numerical analysis:}  Numerical analyses have been carried out for related eigenvalue problems in
\cite{BartelsButtazzo2019,BartelsKellerWachsmuth2022,BartelsButtazzoKeller2025}
and for a related optimal insulation problem under conductive heat transfer in
\cite{AKK2025_numerics}.\vspace{-0.5mm}

\end{itemize}
    
	\section{Preliminaries}\label{sec:preliminaries}\vspace{-1.5mm}\enlargethispage{6.5mm}

    \subsection{Classical function spaces}\vspace{-1mm}

    \hspace{5mm}Let $\omega\subseteq \mathbb{R}^d$, $d\in \mathbb{N}$, be a (Lebesgue) measurable set. Then, for (Lebesgue) measurable functions or vector fields $v,w\colon \omega\to \mathbb{R}^{\ell}$, $\ell\in\{1,d\}$, we employ~the~inner~product~${(v,w)_{\omega}\coloneqq \int_{\omega}{v\odot w\,\mathrm{d}x}}$, 
	whenever the right-hand side is well-defined, where $\odot\colon \mathbb{R}^{\ell}\times \mathbb{R}^{\ell}\to \mathbb{R}$ either denotes scalar~multiplication (\textit{i.e.}, $v\odot w\coloneqq v\,w$) or the Euclidean inner product (\textit{i.e.}, $v\odot w\coloneqq v\cdot w$).
    ~For~$p\in [1,+\infty]$, we employ the notation $\|\cdot\|_{p,\omega}\coloneqq (\int_\omega{\vert \cdot\vert^p\,\mathrm{d}x})^{\smash{\frac{1}{p}}}$~if~$p\in [1,+\infty)$ and $\|\cdot\|_{\infty,\omega}\coloneqq 
    \textup{ess\,sup}_{x\in \omega}{\vert (\cdot)(x)\vert}$~else. Moreover, in the particular case $p=2$, we employ the abbreviated notation $\|\cdot\|_{\omega}\coloneqq \|\cdot\|_{2,\omega}$.\newpage

	
	For $m\in \mathbb{N}$ and an open set $\omega\subseteq \mathbb{R}^d$, $d\in \mathbb{N}$,~we~define the Sobolev space
	\begin{align*}  H^m(\omega)\coloneqq \bigl\{v\in L^2(\omega)\mid \mathrm{D}^{\boldsymbol{\alpha}} v\in L^2(\omega)\textup{ for all }\boldsymbol{\alpha}\in \mathbb{N}_0^d\text{ with }\vert \boldsymbol{\alpha}\vert \leq m\bigr\}\,,
	\end{align*}
	where $\mathrm{D}^{\boldsymbol{\alpha}}\coloneqq \frac{\partial^{\smash{\vert \boldsymbol{\alpha}\vert }}}{\partial x_1^{\smash{\alpha_1}}\cdot\ldots\cdot \partial x_d^{\smash{\alpha_d}}}$ and $\vert \boldsymbol{\alpha}\vert \coloneqq\sum_{i=1}^d{\alpha_i}$ 
	for each multi-index 
    ${\boldsymbol{\alpha}\hspace{-0.1em}\coloneqq \hspace{-0.1em}(\alpha_1,\ldots,\alpha_d)\in \mathbb{N}_0^d}$, and the Sobolev~semi-norm
	\begin{align*} 
		\vert\cdot\vert _{m,\omega}\coloneqq\Bigg(\sum_{\boldsymbol{\alpha}\in \smash{\mathbb{N}_0^d}\,:\,\vert\boldsymbol{\alpha}\vert= m }{\|\mathrm{D}^{\boldsymbol{\alpha}}(\cdot)\|_{\omega}^2}\Bigg)^{\smash{\frac{1}{2}}}\,.
	\end{align*}  
	For $s \in (0,\infty)\setminus \mathbb{N}$ and an open set $\omega\subseteq \mathbb{R}^d$, $d\in \mathbb{N}$, the  Sobolev--Slobodeckij~semi-norm~is~defined~by
	\begin{align*}
		\vert \cdot \vert_{s,\omega}\coloneqq \Bigg(\sum_{\boldsymbol{\alpha}\in \smash{\mathbb{N}_0^d}\,:\,\vert \boldsymbol{\alpha}\vert= \lfloor s\rfloor}{\int_{\omega}{\int_{\omega}{\frac{\vert(\mathrm{D}^{\boldsymbol{\alpha}} (\cdot))(x)-(\mathrm{D}^{\boldsymbol{\alpha}}(\cdot))(y)\vert^2}{\vert x-y\vert^{2(s-\lfloor s\rfloor)+d}}\,\mathrm{d}x}\,\mathrm{d}y}}\Bigg)^{\smash{\frac{1}{2}}}\,.
	\end{align*} 
	Then, for $s \in (0,\infty)\setminus \mathbb{N}$ 
    and an open set $\omega\subseteq \mathbb{R}^n$, $n\in \mathbb{N}$,  the Sobolev--Slobodeckij space~is~\mbox{defined}~by
	\begin{align*}
	      H^s(\omega)\coloneqq \bigl\{v\in H^{ \lfloor s\rfloor}(\omega)\mid \vert v\vert _{s,\omega}<\infty\bigr\}\,.
	\end{align*}  
    If $\omega$ is replaced with a relatively open boundary part $\gamma\subseteq \partial\Omega$, we still employ the~above~notation, with  the Lebesgue measure $\mathrm{d}x$ replaced by the surface measure $\mathrm{d}s$. 

     \subsection{Assumptions on the thermally conducting body and insulated boundary}

     \hspace{5mm}Throughout the entire paper, we assume that
     the \emph{thermally conducting body} $\Omega\subseteq \mathbb{R}^d$, $d\in \{2,3\}$, is a bounded polyhedral Lipschitz domain with (topological) boundary $\partial\Omega$, which  is disjointly~split~into a relatively open \emph{insulated boundary part} $\Gamma_I\subseteq \partial\Omega$ and a relatively open  \emph{Neumann~part}~$\Gamma_N\subseteq \partial\Omega$; more precisely, we have that $\partial\Omega=\overline{\Gamma}_I\cup \overline{\Gamma}_N$.  In~this~connection, we always assume that $\Gamma_I\neq \emptyset$, which
    guarantees the validity of 
    Friedrich's inequality (\textit{cf}.\ \cite[Lem.~3.30]{EG21I}), which states that there exists a constant $c_{\mathrm{F}}>0$ such that for every ${v\in H^1(\Omega)}$,~there~holds
        \begin{align}\label{lem:poin_cont}
            \|v\|_{\Omega}^2\leq \smash{c_{\mathrm{F}}(\|\nabla v\|_{\Omega}^2+\|v\|_{1,\Gamma_I}^2)}\,.
        \end{align}
    The outward unit normal vector field to $\Omega$ is denoted by $n\colon \partial\Omega\to \mathbb{S}^{d-1}\coloneqq\{x\in \mathbb{R}^d\mid \vert x\vert=1\}$.

    \subsection{Triangulations and finite element space}

    \hspace{5mm}Throughout \hspace{-0.1mm}the \hspace{-0.1mm}entire \hspace{-0.1mm}paper, \hspace{-0.1mm}we \hspace{-0.1mm}denote \hspace{-0.1mm}by \hspace{-0.1mm}$\{\mathcal{T}_h\}_{h>0}$ \hspace{-0.1mm}a \hspace{-0.1mm}family \hspace{-0.1mm}of \hspace{-0.1mm}shape-regular~\hspace{-0.1mm}\mbox{triangulations}~\hspace{-0.1mm}of~\hspace{-0.1mm}$\Omega$ (\textit{cf}.\  \cite[Def.\ 22.21]{EG21I}). Here, the parameter
	$h\hspace{-0.1em}\coloneqq\hspace{-0.1em} \max_{T\in \mathcal{T}_h}{\{h_T\hspace{-0.1em}\coloneqq \hspace{-0.1em}\textup{diam}(T)\}}\hspace{-0.1em}>\hspace{-0.1em}0$ refers~to~the~\textit{maximal mesh-size}. The sets of nodes and sides of a triangulation $\mathcal{T}_h$ are denoted by $\mathcal{N}_h$~and~$\mathcal{S}_h$,~\mbox{respectively}. Moreover, we employ the following notation for sets of boundary nodes and facets, respectively:
    \begin{align*}
        \begin{aligned} 
        \mathcal{N}_h^{\partial}&\coloneqq \{\nu\in \mathcal{N}_h\mid \nu\in \partial\Omega\} \,,&&\mathcal{N}_h^{\mathrm{X}}\coloneqq \{\nu\in \mathcal{N}_h^{\partial}\mid \nu\in  \Gamma_{\mathrm{X}}\}\,,&&\mathrm{X}\in \{I,N\}\,,\\
        \mathcal{S}_h^{\partial}&\coloneqq\{S\in \mathcal{S}_h\mid S\subseteq \partial\Omega\}\,, &&\mathcal{S}_h^{X}\coloneqq\{S\in \mathcal{S}_h^{\partial}\mid S\subseteq \Gamma_{\mathrm{X}}\}\,,&&\mathrm{X}\in \{I,N\}\,.
        \end{aligned}
    \end{align*}
    Here, we assume that  $\{\mathcal{T}_h\}_{h>0}$, $\Gamma_I$, and $\Gamma_N$ are chosen in such a way that  ${\mathcal{S}_h^{\partial}=\mathcal{S}_h^I\dot{\cup} \mathcal{S}_h^N}$.

    For every $T\in \mathcal{T}_h$ (and $S\in \mathcal{S}_h$), denoting by 
    $\mathbb{P}^1(T)$ (and $\mathbb{P}^1(S)$) the space of affine functions in $T$ (on $S$), we define the spaces of globally continuous element-wise (and side-wise)~affine~functions  
    \begin{align}
      \mathcal{S}^1(\mathcal{T}_h) &\coloneqq \{ v_h \in C^0(\overline{\Omega})\;\, \mid v_h|_{T} \in \mathbb{P}^1(T)\text{ for all } T \in \mathcal{T}_h \}\,, \label{eq:def_P1space_dom} \\
      \mathcal{S}^1(\mathcal{S}_h^\mathrm{X}) & \coloneqq \{ v_h \in C^0(\overline{\Gamma}_\mathrm{X}) \mid  v_h|_{S} \in \mathbb{P}^1(S)\text{ for all } S \in \mathcal{S}_h^\mathrm{X}\}\,,\quad \mathrm{X}\in \{\partial,I,N\}\,. \label{eq:def_P1space_bndX}
    \end{align}
    Restricting the heat loss functional \eqref{def:E_intro}  to the product space
$ \mathcal{S}^1(\mathcal{T}_h)\times \mathcal{S}^1(\mathcal{S}_h^I)$,
the boundary term generally cannot be integrated exactly and, thus, must  be approximated by a~suitable~quadrature. A \hspace{-0.15mm}natural \hspace{-0.15mm}choice \hspace{-0.15mm}is \hspace{-0.15mm}a \hspace{-0.15mm}node-based \hspace{-0.15mm}quadrature, \hspace{-0.15mm}analogous \hspace{-0.15mm}to \hspace{-0.15mm}the \hspace{-0.15mm}standard~\hspace{-0.15mm}\mbox{mass-lumping}~\hspace{-0.15mm}\mbox{inner}~\hspace{-0.15mm}\mbox{product}, obtained by integrating the nodal interpolation of the original integrand over $\Gamma_I$.

    More precisely, denoting by $I_h\colon C^0(\overline{\Gamma}_I)\to \mathcal{S}^1(\mathcal{S}_h^I)$, the nodal interpolation operator associated with $\mathcal{S}_h^I$ (\textit{cf}.\ \cite[Sec.\ 4]{DziukElliott2013}), for every $v\in C^0(\overline{\Gamma}_I)$ defined by
    \begin{align*}
        I_h v\coloneqq \sum_{\nu \in \mathcal{N}_h^{I}}{v(\nu)\varphi_\nu\!\!\restriction_{\Gamma_I}}\quad\text{ in }\Gamma_I\,,
    \end{align*}
    where $\{\varphi_\nu\}_{\nu\in \mathcal{N}_h}$ is the nodal basis of $\mathcal{S}^1(\mathcal{T}_h)$, 
    the mass-lumping inner product $(\cdot,\cdot)_{\Gamma_I,h}\colon C^0(\overline{\Gamma}_I)\times C^0(\overline{\Gamma}_I)\to \mathbb{R} $ and the induced norm $\|\cdot\|_{\Gamma_I,h}\colon C^0(\overline{\Gamma}_I)\to \mathbb{R}$, for every $v,w\in C^0(\overline{\Gamma}_I)$, are defined~by\vspace{-0.5mm}
    \begin{align*}
        (v,w)_{\Gamma_I,h}\coloneqq (I_h\{v w\},1)_{\Gamma_I}\,,\qquad \|v\|_{\Gamma_I,h}\coloneqq \|I_h\{v\}\|_{\Gamma_I}\,.
    \end{align*}

    Important approximation properties of the node-based quadrature on $\Gamma_I$ (applied~to compositions of Lipschitz functions with polynomial functions) are summarized in~the~following~lemma.\enlargethispage{5mm}\vspace{-0.5mm}

    \begin{lemma}[Approximation properties of node-based quadrature on $\Gamma_I$]
\label{lem:mass_lumping_lipschitz}
    Let $\Phi\hspace{-0.1em}\in\hspace{-0.1em} C^{0,1}(\mathbb{R}\times \mathbb{R})$, $v\in H^1(\Omega)$, and $\mathtt{C}\in \mathbb{R}$. Moreover, let  $v_h\in \mathcal{S}^{\ell}(\mathcal{T}_h)$, $h>0$, where $\ell\in \mathbb{N}$, and $\{\mathtt{C}_h\}_{h>0}\subseteq \mathbb{R}$ be sequences such that
\begin{subequations} \label{lem:mass_lumping_lipschitz.0}
    \begin{alignat}{3}\label{lem:mass_lumping_lipschitz.0.1}
        v_h&\rightharpoonup v&&\quad\text{ in }H^1(\Omega)&&\quad (h\to 0^+)\,,\\
        \mathtt{C}_h&\to \mathtt{C}&&\quad \text{ in }\mathbb{R}&&\quad (h\to 0^+)\,.\label{lem:mass_lumping_lipschitz.0.2}
    \end{alignat} 
\end{subequations}
Then, there holds\vspace{-0.5mm}
\begin{subequations} \label{lem:mass_lumping_lipschitz.1}
\begin{align}\label{lem:mass_lumping_lipschitz.1.1}
    \|I_h\{\Phi(v_h,\mathtt{C}_h)\}-\Phi(v_h,\mathtt{C}_h)\|_{\Gamma_I}\lesssim \smash{h^{\frac{1}{2}}}\| \nabla v_h\|_{\Omega}\,,\\
    I_h\{\Phi(v_h,\mathtt{C}_h)\}\to \Phi(v,\mathtt{C})\quad \text{ in }L^2(\Gamma_I)\quad (h\to 0^+)\,.\label{lem:mass_lumping_lipschitz.1.2}
\end{align} 
\end{subequations}
\end{lemma}

\begin{proof}
Let $S \in  \mathcal{S}_h^I$ be fixed, but arbitrary with outward unit normal vector $n_S\coloneqq n\!\! \restriction_S\in \mathbb{S}^{d-1}$. Moreover, let ${\mathrm{P}_S \coloneqq(\mathrm{Id} - n_S \otimes n_S)\in \mathbb{R}^{d\times d}}$ denote the projection onto $(\mathbb{R}n_S)^\perp$ and denote by $\nabla_S \hspace{-0.1em}\coloneqq\hspace{-0.1em} \mathrm{P}_S \nabla$ the tangential gradient on $S$. Then, since $\Phi\hspace{-0.1em}\in\hspace{-0.1em} C^{0,1}(\mathbb{R}\times \mathbb{R})$,~by~the~chain~rule~for~\mbox{Sobolev}\linebreak  functions (\textit{cf}.\ \cite[Thm.\ 2.1.11]{Ziemer1989}) and $\nabla_S\mathtt{C}_h=0$ on $\Gamma_I$, we~have~that~${\Phi(v_h,\mathtt{C}_h)|_S\in W^{1,\infty}(S)}$~with\vspace{-0.5mm}
\begin{align}\label{lem:mass_lumping_lipschitz.2}
    |\nabla_S \Phi(v_h,\mathtt C_h)|
\le
\operatorname{Lip}(\Phi)|\nabla_S v_h|
\quad\text{ a.e.\ on }S\,.
\end{align}
By a local inverse estimate (\textit{cf}.\ \cite[Lem.\ 12.1]{EG21I}), the local interpolation estimate for~$I_h$~on~$S$ (\textit{cf}.\ \cite[ineq.\ (11.17)]{EG21I}), again, \eqref{lem:mass_lumping_lipschitz.2}, a local inverse estimate (\textit{cf}.\ \cite[Lem.\ 12.1]{EG21I}), and the discrete trace inequality (\textit{cf}.\ \cite[Lem.\ 12.8]{EG21I}), we have that
\begin{align}\label{lem:mass_lumping_lipschitz.3}
\begin{aligned}
    \|I_h\{\Phi(v_h,\mathtt C_h)\}-\Phi(v_h,\mathtt C_h)\|_{S}^2&\lesssim \vert S\vert\|I_h\{\Phi(v_h,\mathtt C_h)\}-\Phi(v_h,\mathtt C_h)\|_{\infty,S}^2
    \\&\lesssim
\vert S\vert h_S^2 \|\nabla_S \Phi(v_h,\mathtt C_h)\|_{\infty,S}^2
\\&\lesssim
\vert S\vert  h_S^2 \|\nabla_S v_h\|_{\infty,S}^2
\\&\lesssim
h_S^2 \|\nabla_S v_h\|_{S}^2
\\&\lesssim h_S \|\nabla v_h\|_{T_S}^2\,,
\end{aligned}
\end{align}
where $T_S\hspace{-0.15em}\in\hspace{-0.15em} \mathcal{T}_h$ is the unique element such that $S\hspace{-0.15em}\subseteq\hspace{-0.15em} T_S$. Eventually, summing with respect~to~${S\hspace{-0.15em}\in \hspace{-0.15em}\mathcal{S}_h^{I}}$, we conclude that the claimed error estimate \eqref{lem:mass_lumping_lipschitz.1.1} applies,
which, by the Lipschitz continuity of $\Phi\colon \mathbb{R}\times \mathbb{R}\to \mathbb{R}$, in turn, implies that the claimed convergence \eqref{lem:mass_lumping_lipschitz.1.2} applies as well.
\end{proof}

The following lemma summarizes standard stability and approximation properties of the node-based quadrature on $\Gamma_I$.\vspace{-0.5mm}

    \begin{lemma}[Stability and approximation properties of node-based quadrature on $\Gamma_I$]\label{lem:std_prop_lumping}
        For every $v_h,w_h\in \mathcal{S}^1(\mathcal{S}_h^{I})$, there holds\vspace{-0.5mm}
        \begin{subequations}
        \begin{align}\label{lem:std_prop_lumping.1}
            \|v_h\|_{\Gamma_I}\leq \|v_h\|_{\Gamma_I,h}\leq (d+1)\|v_h\|_{\Gamma_I}\,,
            \\
            \vert (v_h,w_h)_{\Gamma_I,h}- (v_h,w_h)_{\Gamma_I}\vert \lesssim h^2\|\nabla_{\Gamma} v_h\|_{\Gamma_I}\|\nabla_{\Gamma} w_h\|_{\Gamma_I}\,,\label{lem:std_prop_lumping.2}
        \end{align}
        where the  \emph{tangential gradient} is defined by 
        $\nabla_{\Gamma} v_h\!\!\restriction_S\coloneqq \nabla_S (v_h\!\!\restriction_S)$ for all $v_h\in \mathcal{S}^1(\mathcal{S}_h^{I})$ and $S\in \mathcal{S}_h^{I}$.
        \end{subequations}
    \end{lemma}

    \begin{proof}
        See \cite[Lem.\ 3.9]{Bartels15}.
    \end{proof}

     \pagebreak

      \section{Mathematical modelling of an optimal insulation problem}\label{sec:modelling}\vspace{-1mm}
      
\hspace{5mm}Let $f\in L^2(\Omega)$ be a given \emph{heat source density} (located in the thermally conducting~body~$\Omega$), $g\in (H^{\smash{\frac{1}{2}}}(\Gamma_N))^*$ a given \emph{heat flux} (across the Neumann boundary $\Gamma_N$), jointly satisfying the \textit{non-trivial net heat input condition} (see \cite[Sec.\ 3]{AKK2025ConvectionModelling2025}, for a short interpretation)\vspace{-0.5mm}
\begin{align}\label{eq:data_condition}
    (f,1)_{\Omega}+\langle g,1\rangle_{\Gamma_N}\neq 0\,,\\[-6mm]\notag
\end{align}
$u_{\infty}\in \smash{H^1(\mathbb{R}^d\setminus \overline{\Omega})}$ a given \emph{ambient temperature} (of the surrounding medium in $\smash{\mathbb{R}^d\setminus \Omega}$), $\kappa>0$ the material-specific \emph{thermal conductivity} of the conducting body, $\beta>0$ a given system-specific \emph{heat transfer coefficient},  $m>0$ a given \emph{amount of insulation material}, and $\smash{\widetilde{\mathtt{d}}}\in L^\infty(\Gamma_I)$ a non-negative \emph{distribution function (of the insulation material in direction of the normal vector field)}.~Moreover,\linebreak let 
 $\varepsilon>0$ be a fixed, but arbitrarily small number. Then, for a given Lipschitz continuous (globally) transversal vector field $k\in (C^{0,1}(\Gamma_I))^d$ (\textit{cf}.\ \cite[Sec.\ 2.4]{AKK2025ConvectionModelling2025}), denoting by $\mathtt{d}\coloneqq (k\cdot n)^{-1}\widetilde{\mathtt{d}}\in L^\infty(\Gamma_I)$~the\linebreak distribution function in direction of the transversal vector field, we define the~\textit{\mbox{insulation}~\mbox{boundary} layer} $\Sigma^{\varepsilon}_{I}\subseteq \mathbb{R}^d$, the \textit{interacting insulation boundary} $\Gamma^{\varepsilon}_{I}$, and the \textit{insulated body} $\Omega^{\varepsilon}_{I}$,~respectively,~via\vspace{-4.5mm}
 \begin{subequations}
     \begin{align}\label{def:some_notation}
        \Sigma^{\varepsilon}_{I}&\coloneqq\Sigma^{\varepsilon}_{I}(\mathtt{d})\coloneqq \bigl\{s+tk(s)\mid s\in \Gamma_I\,,\;t\in [0,\varepsilon\mathtt{d}(s))\bigr\}\,,\\
        \Gamma^{\varepsilon}_{I}&\coloneqq \Gamma^{\varepsilon}_{I}(\mathtt{d})\coloneqq \bigl\{s+\varepsilon\mathtt{d}(s)k(s)\mid s\in \Gamma_{I}\bigr\}\,,\\
        \Omega^{\varepsilon}_{I}&\coloneqq \Omega^{\varepsilon}_{I}(\mathtt{d})\coloneqq \overline{\Omega}\cup\Sigma^{\varepsilon}_{I}\,.\\[-6mm]\notag
    \end{align} 
 \end{subequations}
    Then, we consider the \emph{heat loss functional} $\smash{E}_\varepsilon^\mathtt{d}\colon 
    H^1(\Omega^{\varepsilon}_{I})\to \mathbb{R}$, for every $v_\varepsilon\in H^1(\Omega^{\varepsilon}_{I})$ defined by\vspace{-0.5mm}
    \begin{align}\label{eq:Evarh}
\smash{E}_\varepsilon^\mathtt{d}(v_\varepsilon)\coloneqq  
    \tfrac{\kappa}{2}\|\nabla v_\varepsilon\|_{\Omega}^2+\tfrac{\varepsilon}{2}\|\nabla v_\varepsilon\|_{\Sigma^{\varepsilon}_{I}}^2+\tfrac{\beta}{2}\|v_\varepsilon-u_{\infty}\|_{\Gamma^{\varepsilon}_{I}}^2-(f,v_\varepsilon)_{\Omega}-\langle g,v_{\varepsilon}\rangle_{\Gamma_N}\,. \\[-6mm]\notag
    \end{align} 
    Since \hspace{-0.1mm}the \hspace{-0.1mm}functional \hspace{-0.1mm}\eqref{eq:Evarh} \hspace{-0.1mm}is \hspace{-0.1mm}proper, \hspace{-0.1mm}strictly \hspace{-0.1mm}convex, \hspace{-0.1mm}weakly \hspace{-0.1mm}coercive, \hspace{-0.1mm}and \hspace{-0.1mm}lower \hspace{-0.1mm}semi-continuous,~\hspace{-0.1mm}the
    direct method in the calculus of variations yields the existence of a unique~minimizer~${u_\varepsilon^\mathtt{d}\in H^1(\Omega^{\varepsilon}_{I})}$, which formally satisfies the  Euler--Lagrange equations\vspace{-0.5mm} 
\begin{align}\label{eq:ELE_Eepsh}
        \left\{\quad\begin{aligned}
            -\kappa \Delta u_\varepsilon^\mathtt{d}&=f&&\quad \text{ a.e.\ in }\Omega\,,\\[-0.5mm]
             \kappa\nabla u_\varepsilon^\mathtt{d}\cdot n&=g&&\quad \text{ a.e.\ on }\Gamma_N\,,\\[-0.5mm]
             -\varepsilon \Delta u_\varepsilon^\mathtt{d}&=0&&\quad \text{ a.e.\ in }\Sigma^{\varepsilon}_{I}\,,\\[-0.5mm]
            \varepsilon \nabla u_\varepsilon^\mathtt{d}\cdot n +\beta(u_\varepsilon^\mathtt{d}-u_{\infty}) &= 0&&\quad \text{ a.e.\ on }\Gamma^{\varepsilon}_{I}\,,\\[-0.5mm]
           \kappa \nabla (u_\varepsilon^\mathtt{d})^+\cdot n&=\varepsilon \nabla (u_\varepsilon^\mathtt{d})^-\cdot n&&\quad \text{ a.e.\ on } \Gamma_I\,,
        \end{aligned}\right.\\[-6mm]\notag
    \end{align} 
    where $(u_\varepsilon^\mathtt{d})^-$ and $(u_\varepsilon^\mathtt{d})^+$ denote the
    traces of $u_\varepsilon^\mathtt{d}$ with respect to $\Omega$ and $\Sigma^{\varepsilon}_{I}$, respectively.\enlargethispage{4mm}

    In \hspace{-0.15mm}the \hspace{-0.15mm}case \hspace{-0.15mm}$\mathtt{d}\hspace{-0.15em}\in\hspace{-0.15em} C^{0,1}(\Gamma_I)$, \hspace{-0.15mm}if \hspace{-0.15mm}we \hspace{-0.15mm}pass \hspace{-0.15mm}to \hspace{-0.15mm}the \hspace{-0.15mm}limit \hspace{-0.15mm}(as \hspace{-0.15mm}$\varepsilon\hspace{-0.15em}\to\hspace{-0.15em} 0^{+}$)~\hspace{-0.15mm}with~\hspace{-0.15mm}a~\hspace{-0.15mm}family~\hspace{-0.15mm}of~\hspace{-0.15mm}\mbox{trivial}~\hspace{-0.15mm}\mbox{extensions}~to $L^2(\mathbb{R}^d)$ of the heat loss functionals $\smash{E}_\varepsilon^\mathtt{d}\colon H^1(\Omega^{\varepsilon}_{I})\to \mathbb{R}$, $\varepsilon>0$, 
    in the sense of $\Gamma(L^2(\mathbb{R}^d))$-convergence (\textit{cf}.\ \cite[Thm.\ 5.1]{AKK2025ConvectionModelling2025}), we arrive at the functional $E^\mathtt{d}(v)\colon H^1(\Omega)\to  \mathbb{R}$, for every $v\in H^1(\Omega)$~defined~by\vspace{-0.5mm}
    \begin{align}\label{eq:Eh}
        \smash{E}^\mathtt{d}(v)\coloneqq  \tfrac{\kappa}{2}\|\nabla v\|_{\Omega}^2  + \tfrac{\beta}{2}\|(1+\beta\smash{\widetilde{\mathtt{d}}})^{-\smash{\frac{1}{2}}}(v- u_{\infty})\|_{\Gamma_{I}}^2-(f,v)_{\Omega}-\langle g,v\rangle_{\Gamma_N}\,. \\[-6mm]\notag
    \end{align}
    Since \hspace{-0.1mm}the \hspace{-0.1mm}functional \hspace{-0.1mm}\eqref{eq:Eh} \hspace{-0.1mm}is \hspace{-0.1mm}proper, \hspace{-0.1mm}strictly \hspace{-0.1mm}convex, \hspace{-0.1mm}weakly \hspace{-0.1mm}coercive, \hspace{-0.1mm}and \hspace{-0.1mm}lower \hspace{-0.1mm}semi-continuous,~\hspace{-0.1mm}the
    direct method in the calculus of variations yields the existence of a~unique~minimizer~${u^\mathtt{d}\in H^1(\Omega)}$, which formally satisfies
     \begin{align}\label{eq:ELE_Eh}
         \left\{\quad\begin{aligned}
             -\kappa \Delta u^\mathtt{d}&=f&&\quad \text{ a.e.\ in }\Omega\,,\\[-0.5mm]
             (1+\beta\smash{\widetilde{\mathtt{d}}})\nabla u^\mathtt{d}\cdot n+\beta (u^\mathtt{d}-u_{\infty})&=0&&\quad \text{ a.e.\ on }\Gamma_{I}\,,\\[-0.5mm]
              \nabla u^\mathtt{d} \cdot n &=g&&\quad \text{ a.e.\ on }\Gamma_N\,.
         \end{aligned}\right.
     \end{align} 
    We are interested in determining the non-negative distribution function $\smash{\widetilde{\mathtt{d}}}\in L^2(\Gamma_{I})$  that provides the \textit{`best'} insulating performance, once
the total amount of  insulation material~$m>0$~is~fixed. We seek the distribution function $\smash{\widetilde{\mathtt{d}}}\in L^2(\Gamma_{I})$  in the class\vspace{-0.5mm}
    \begin{align*}
        \mathcal{H}_m\coloneqq \smash{\bigl\{ \eta\in L^2(\Gamma_{I})\mid \eta\ge 0\text{ a.e.\ on }\Gamma_{I}\,,\; (\eta,1)_{\Gamma_{I}}=m\bigr\}}\,.\\[-6mm]
    \end{align*}
    In summary, we are interested in the non-smooth double minimization problem\vspace{-0.5mm}
    \begin{align}\label{eq:double_min}
        \smash{\min_{(v,\eta)\in H^1(\Omega)\times \mathcal{H}_m}{\bigl\{\smash{E}^{\eta}(v)\bigr\}}\,.}
    \end{align}

\section{The continuous optimal insulation problem}

\hspace{5mm}Given a \emph{heat source density} $f\in L^2(\Omega)$, a \emph{heat flux} $g\in (H^{\smash{\frac{1}{2}}}(\Gamma_N))^*$, jointly satisfying \eqref{eq:data_condition}, 
an \emph{ambient temperature} $u_{\infty}\in H^1(\mathbb{R}^d\setminus \overline{\Omega})$, a \emph{thermal conductivity coefficient} $\kappa>0$,~a~\emph{heat~transfer coefficient} $\beta>0$, and an \emph{amount of insulation material} $m>0$, we~seek~a~pair~${(u,\smash{\widetilde{\mathtt{d}}})\hspace{-0.15em}\in\hspace{-0.15em} H^1(\Omega)\hspace{-0.15em}\times\hspace{-0.15em} L^2(\Gamma_{I})}$ consisting of a \emph{temperature distribution} and a \emph{distribution function (of the insulation~\mbox{material})} that minimize  the \emph{heat loss} functional $ \smash{E}\colon H^1(\Omega)\times L^2(\Gamma_{I})\to  \mathbb{R}\cup\{+\infty\}$, for every $(v,\eta) \in H^1(\Omega)\times L^2(\Gamma_{I})$ defined by
    \begin{align}\label{def:E}
        E(v,\eta)\coloneqq  
            \tfrac{\kappa}{2}\|\nabla v\|_{\Omega}^2 + \tfrac{\beta}{2} \|(1+\beta \eta)^{-\smash{\frac{1}{2}}}(v-u_{\infty})\|_{\Gamma_{I}}^2-(f,v)_{\Omega}-\langle g,v\rangle_{\Gamma_N}+I_{\mathcal{H}_m}(\eta)\,.
    \end{align}
Here, we have invoked the abbreviated notation $\langle g,\widehat{v}\rangle_{\Gamma_N}\coloneqq \langle g,\widehat{v}\rangle_{H^{\smash{\frac{1}{2}}}(\Gamma_N)}$ for all $\widehat{v}\in H^{\smash{\frac{1}{2}}}(\Gamma_N)$ and  the indicator functional $I_{\mathcal{H}_m}\colon L^2(\Gamma_{I})\to \mathbb{R}\cup\{+\infty\}$, for every $\widehat{\eta}\in L^2(\Gamma_{I})$, is defined by
\begin{align*}
    I_{\mathcal{H}_m}(\widehat{\eta})\coloneqq \begin{cases}
        0&\text{ if }\widehat{\eta}\in \mathcal{H}_m\,,\\
        +\infty &\text{ else}\,.
    \end{cases}
\end{align*} 

\subsection{Optimality conditions}

\hspace{5mm}The following proposition identifies optimality conditions that characterize a minimizing pair $(u,\smash{\widetilde{\mathtt{d}}})\in H^1(\Omega)\times L^2(\Gamma_{I})$ of the heat loss functional \eqref{def:E}.

\begin{proposition}[Optimality conditions]\label{prop:optimality}
    A pair $(u,\smash{\widetilde{\mathtt{d}}})\in  H^1(\Omega)\times \mathcal{H}_m$ is minimal for the heat loss functional \eqref{def:E} if and only if the following optimality conditions are satisfied:\enlargethispage{5mm}
    \begin{itemize}[noitemsep,topsep=2pt,leftmargin=!,labelwidth=\widthof{(iii)}]
        \item[(i)] There exists a constant $\mathtt{C}_u>0$ such that\footnote{Here, $(\cdot)_+\coloneqq \max\{\cdot,0\}\colon \mathbb{R}\to \mathbb{R}_{\ge 0}$.}
        \begin{align}\label{prop:optimality.1}
            \mathtt{C}_u=\tfrac{1}{\beta m}\|(|u-u_\infty|-\mathtt{C}_u)_+\|_{1,\Gamma_{I}}\,;
        \end{align}
        \item[(ii)] The distribution function $\smash{\widetilde{\mathtt{d}}} \in \mathcal{H}_m$ satisfies 
        \begin{align}\label{prop:optimality.2}
              \smash{\widetilde{\mathtt{d}}} = 
             \tfrac{1}{\beta \mathtt{C}_u}(\vert u-u_\infty\vert-\mathtt{C}_u)_+\quad \text{ a.e.\ on }\Gamma_{I}\,; 
        \end{align}
        \item[(iii)] The pair $(u,\smash{\widetilde{\mathtt{d}}})\in  H^1(\Omega)\times \mathcal{H}_m$, for every $v\in H^1(\Omega)$, satisfies
        \begin{align}\label{prop:optimality.3}
           \kappa (\nabla u,\nabla v)_{\Omega}+\beta ( (1+\beta \smash{\widetilde{\mathtt{d}}})^{-1}(u-u_{\infty}),v)_{\smash{\Gamma_{I}}}&=(f,v)_{\Omega}+\langle g,v\rangle_{\Gamma_N}\,.
        \end{align}
    \end{itemize}
\end{proposition}

\begin{remark}[Physical interpretation of the constant $\mathtt{C}_u$]
The constant $\mathtt{C}_u>0$ appearing in the optimality conditions \eqref{prop:optimality.1} may be interpreted 
as a \emph{critical temperature difference} on the insulated boundary. 
More precisely, the optimality condition \eqref{prop:optimality.2} 
shows that insulating material is allocated only at boundary points, where 
the local temperature difference exceeds~the~threshold~value~$\mathtt{C}_u$. 
In addition, $\mathtt{C}_u$ coincides with the Lagrange multiplier associated with the 
constraint on the total amount of insulation material. 
\end{remark}

The following lemma shows that the implicitly defined constant in \eqref{prop:optimality.1} exists and is unique.

\begin{lemma}[Existence and uniqueness of $\mathtt{C}_u$]\label{lem:uniqueness_Cu}
    For every $v\in  L^2(\Gamma_{I})$, there exists a unique constant $\mathtt{C}_v\ge  0$ such that
    \begin{align*}
         \mathtt{C}_v=\tfrac{1}{\beta m}\|(|v-u_{\infty}|-\mathtt{C}_v)_+\|_{1,\Gamma_{I}}\,.
    \end{align*}
    In addition, we have that $\mathtt{C}_v=0$ if and only if $v=u_\infty$ a.e.\ on $\Gamma_{I}$.
\end{lemma}

\begin{proof}
    This follows along the lines of the proof of \cite[Lem.\ 4.1]{PietraNitschScalaTrombetti2021} up to minor adjustments.
\end{proof}
 
The following lemma demonstrates that the implicitly defined constant in \eqref{prop:optimality.1} is positive, which is needed for the well-posedness of the optimality condition \eqref{prop:optimality.2}.

\begin{lemma}[Positivity of $\mathtt{C}_u$]\label{lem:positivity_Cu}
    Let $(u,\smash{\widetilde{\mathtt{d}}})\in  H^1(\Omega)\times L^2(\Gamma_{I})$  be a pair such that the optimality condition \eqref{prop:optimality.3} 
    is satisfied. Then, the constant 
    $\mathtt{C}_u\ge 0$ implicitly defined by \eqref{prop:optimality.1} satisfies $\mathtt{C}_u>0$.
\end{lemma}

\begin{proof}
    Suppose that $\mathtt{C}_u=0$. Then, \eqref{prop:optimality.1} implies that $u=u_{\infty}$ a.e.\ on $\Gamma_{I}$ and, consequently,~\eqref{prop:optimality.3}, for every $v\in H^1(\Omega) $, reads
    \begin{align}\label{lem:positivity_Cu.1}
        \kappa(\nabla u,\nabla v)_{\Omega}=(f,v)_{\Omega}+\langle g,v\rangle_{\Gamma_N}\,.
    \end{align}
    For $v=1\in H^1(\Omega)$ in \eqref{lem:positivity_Cu.1}, we find that $(f,1)_{\Omega}+\langle g,1\rangle_{\Gamma_N}=0$, which contradicts assumption~\eqref{eq:data_condition}. As a result, we conclude~that~$\mathtt{C}_u>0$.
\end{proof}

\begin{proof}[Proof (of Proposition \ref{prop:optimality}).]
    A pair $(u,\smash{\widetilde{\mathtt{d}}})\in  H^1(\Omega)\times L^2(\Gamma_{I})$ is minimal for \eqref{def:E} if and only if
    \begin{subequations}\label{prop:optimality.4}
    \begin{align}\label{prop:optimality.4.1}
        0_{L^2(\Gamma_I)}&\in \partial_2 E(u,\smash{\widetilde{\mathtt{d}}})\,,\\
        0_{(H^1(\Omega))^*}&=\mathrm{D}_1 E(u,\smash{\widetilde{\mathtt{d}}})\,,\label{prop:optimality.4.2}
    \end{align}
    \end{subequations}
    where the identity \eqref{prop:optimality.4.2} is equivalent to \eqref{prop:optimality.3}, while the inclusion \eqref{prop:optimality.4.1} is equivalent to 
    \begin{align}\label{prop:optimality.5}
        \smash{\widetilde{\mathtt{d}}} \in \underset{\eta\in \mathcal{H}_m}{\textup{argmin}\,}{\{E(u,\eta)\}}\,.
    \end{align}
    Due to \cite[Prop.\ 4.1 \& Lem.\ 4.1]{PietraNitschScalaTrombetti2021} (with $v=u-u_{\infty}$), the unique minimizer~is~given~via~\eqref{prop:optimality.2} 
    \if0\begin{align}\label{prop:optimality.5}
        \smash{\mathtt{C}_u=\tfrac{1}{\vert \{\vert u-u_{\infty}\vert \ge \mathtt{C}_u\}\vert+\beta m}\|u-u_{\infty}\|_{1,\{\vert u-u_{\infty}\vert \ge \mathtt{C}_u\}}\,,}
    \end{align}
    \fi
     with  $\mathtt{C}_u>0$
    given via \eqref{prop:optimality.1}. 
\end{proof}

\begin{remark}[on Proposition \ref{prop:optimality}]
    Since \eqref{prop:optimality.5} is equivalent to $-\beta(1+\beta \smash{\widetilde{\mathtt{d}}})^{-2}(u - u_\infty)^2\in \partial I_{\mathcal{H}_m}(\smash{\widetilde{\mathtt{d}}})$, 
    the optimality conditions \eqref{prop:optimality.1} and \eqref{prop:optimality.2} are jointly equivalent to the optimality condition 
    \begin{align}
        ((1+\beta \smash{\widetilde{\mathtt{d}}})^{-2}(u - u_\infty)^2,\smash{\widetilde{\mathtt{d}}} - \eta)_{\Gamma_I}\ge 0\quad\text{ for all }\eta\in \mathcal{H}_m\,.\label{prop:optimality.3b}
    \end{align} 
\end{remark}

\subsection{$L^\infty$-bounds on the insulated boundary}

\hspace{5mm}The \textit{a priori} error analysis in \Cref{sec:apriori} hinges on the assumption that both $u,\smash{\widetilde{\mathtt{d}}} \in L^\infty(\Gamma_I)$.
The following lemma affirms this assumption provided the heat source $f\in L^2(\Omega)$ and heat flux $g\in (H^{\frac{1}{2}}(\Gamma_N))^*$ 
are sufficiently integrable and the ambient temperature $u_\infty$ is essentially~bounded.
\begin{lemma}\label{lem:Linf}
Let $f \in L^{r}(\Omega)$ with $r>\frac{d}{2}$, $g \in L^s(\Gamma_N)$ with $s > d - 1$, and $u_{\infty} \in L^\infty(\Gamma_I)$. Then, there holds $u \in L^\infty(\Gamma_I)$ and, thus,  $\smash{\widetilde{\mathtt{d}}} \in L^\infty(\Gamma_I)$.
\end{lemma}

The proof of Lemma \ref{lem:Linf} proceeds using a classic argument by Stampacchia (\textit{cf}.\ \cite{Stampacchia1965}), which makes use of a truncation argument based on the truncation operator $\mathtt{T}_\lambda\colon W^{1,1}(\Omega)\to W^{1,1}(\Omega)\cap L^{\infty}(\Omega) $, $\lambda>0$, for every $v\in W^{1,1}(\Omega)$ defined by
\begin{align} \label{eq:def_trunc}
    \mathtt{T}_\lambda v \coloneqq \textup{sign}(v)(|v| - \lambda)_+\quad \text{ a.e.\ in }\Omega\,,
\end{align}
which is well-defined by the Lipschitz continuity of $\Phi\coloneqq ((a,c)\mapsto\textup{sign}(a)(|a| - c)_+)\colon  \mathbb{R}\times \mathbb{R}\to \mathbb{R}$ and the chain rule for Sobolev functions (\textit{cf}.\ \cite[Thm.\ 2.1.11]{Ziemer1989}), combined~with~the~\mbox{following}~lemma. 

 \begin{lemma}\label{lem:Stampacchia}
 Let $\lambda_0\in\mathbb{R}$ and let
$\varphi:[\lambda_0,\infty)\to[0,\infty)$ be non-increasing. Moreover, 
assume that there exist constants $K,a>0$ and $b>1$ such that
for every $\lambda'>\lambda\ge\lambda_0$, there holds
\begin{align*}
\varphi(\lambda')
\le
K(\lambda'-\lambda)^{-a}
\,\varphi(\lambda)^{b}\,.
\end{align*}
Then, for $\smash{\lambda^*\coloneqq
2^{\smash{\frac{b}{b-1}}}
K^{\smash{\frac{1}{a}}}
\varphi(\lambda_0)^{\smash{\frac{b-1}{a}}}+\lambda_0}$, there holds $\varphi(\lambda^*)=0$.
\end{lemma}

\begin{proof}
    See \cite[Lem.\ B.1]{KinderlehrerStampacchia}.
\end{proof}

\begin{proof}[Proof (of Lemma \ref{lem:Linf}).] To begin with, we introduce the sets
\begin{align*}
    \Gamma_I^> \coloneqq \{s \in \Gamma_I \mid |u(s) - u_\infty(s)| \ge \mathtt{C}_u\}\,,\qquad
    \Gamma_I^\le \coloneqq \Gamma_I \setminus \Gamma_I^>\,,
\end{align*} 
and fix $b \in (1,2)$. Then, depending on $d=3$ or $d=2$, we make the following definitions:
\begin{itemize}[noitemsep,topsep=2pt,leftmargin=!,labelwidth=\widthof{\quad$\bullet$ Case 1: ($d=3$).},font=\itshape]
    \item[$\bullet$ Case 1: ($d=3$).] We set $p = 6$, $q = 4$, $\tilde{r} = \frac{3}{3-b}$, $\tilde{s} = \frac{2}{2-b}$, implying $p = 2b \tilde{r}'$ and $q = 2b\tilde{s}'$;
    \item[$\bullet$ Case 2: ($d=2$).] We set $p = 2b \tilde{r}'$ and $q = 2b \tilde{s}'$ for fixed $\tilde{r} ,\tilde{s} > 1$.
\end{itemize}\newpage
\noindent Since $\mathtt{C}_u=0$ is excluded by Lemma \ref{lem:positivity_Cu} and 
since $|\Gamma_I^>| = 0$ is equivalent to $\vert u-u_\infty\vert \leq \mathtt{C}_u$~a.e.~in~$\Gamma_I$, which implies that $\vert u\vert\leq \mathtt{C}_u+\vert u_\infty\vert$ a.e.\ in $\Gamma_I$ and, due to \eqref{prop:optimality.2}, that $\smash{\widetilde{\mathtt{d}}}=0$ a.e.\ on $\Gamma_I$, given the assumption that $u_\infty\in L^\infty(\Gamma_I)$, it remains to consider the case that both $\mathtt{C}_u > 0$~and~$|\Gamma_I^>| > 0$.  Using the optimality conditions \eqref{prop:optimality.1}--\eqref{prop:optimality.3}, for every $v \in H^1(\Omega)$, we find that
\begin{align} \label{eq:Linf_bnd.1}
\smash{\kappa (\nabla u, \nabla v)_\Omega+ \mathtt{C}_u\beta ( \text{sgn}(u-u_{\infty}),v)_{\smash{\Gamma_I^>}} + \beta(u - u_\infty, v)_{\smash{\Gamma_I^\le}} = (f,v)_\Omega + (g, v)_{\Gamma_N}\,.}
\end{align}
Next, for every $\lambda > \lambda_0\coloneqq \|u_\infty \|_{\infty,\Gamma_I} + \mathtt{C}_u$, we define the sets 
\begin{align*}
    \Omega^\lambda   &\coloneqq \{ x \in \Omega \mid |u(x)| > \lambda\} \,,\\ \Gamma^{\lambda} &\coloneqq \{s \in \partial\Omega \mid |u(s)| > \lambda\}\,,\\\Gamma_{\mathrm{X}}^{\lambda}&\coloneqq \Gamma^{\lambda}\cap \Gamma_{\mathrm{X}}\,,\;\mathrm{X}\in \{I,N\}\,.
\end{align*} 
Then, by the definition of the truncation operator \eqref{eq:def_trunc} and the chain rule for Sobolev functions (\textit{cf}.\ \cite[Thm.\ 2.1.11]{Ziemer1989}), we have that
\begin{subequations}\label{eq:Linf_bnd.1.0}
\begin{alignat}{2}\label{eq:Linf_bnd.1.0.1}
    \mathtt{T}_\lambda u &= 0&&\quad\text{ a.e.\ in }\Omega\setminus \Omega^{\lambda}\text{ and }\text{a.e.\ on }\partial\Omega\setminus \Gamma^{\lambda}\,,\\\label{eq:Linf_bnd.1.0.2}
    \nabla \mathtt{T}_\lambda u &= \chi_{\Omega^{\lambda}} \nabla u&&\quad \text{ a.e.\ in }\Omega\,,\\
    \text{sgn}(\mathtt{T}_\lambda u) &=\text{sgn}(u)=\text{sgn}(u - u_\infty)&&\quad\text{ a.e.\ on }\smash{\Gamma_{I}^{\lambda}\subseteq \Gamma_{I}^{>}}\,,\label{eq:Linf_bnd.1.0.3}
\end{alignat}
\end{subequations}
where we used in \eqref{eq:Linf_bnd.1.0.3} that $\lambda > u_\infty + \mathtt{C}_u$ a.e.\ on $\Gamma_I$ and $-\lambda < u_\infty - \mathtt{C}_u$~a.e.~on~$\Gamma_I$ and, thus, $\Gamma^\lambda_I \subseteq \Gamma_I^>$ for $\lambda > \lambda_0$.
As a result, due to \eqref{eq:Linf_bnd.1.0}, by choosing  $v = \mathtt{T}_\lambda u \in H^1(\Omega)$ in \eqref{eq:Linf_bnd.1},~we~obtain
\begin{align} \label{eq:Linf_bnd.2}
\smash{\kappa \|\nabla \mathtt{T}_\lambda u \|_{\Omega}^2+ \mathtt{C}_u\beta\|\mathtt{T}_\lambda u\|_{1,\Gamma_I^\lambda}  = (f, \mathtt{T}_\lambda u )_{\Omega^\lambda} + (g , \mathtt{T}_\lambda u )_{\Gamma_N^\lambda}\,.}
\end{align}
Moreover,  due to \eqref{eq:Linf_bnd.1.0.1} and \eqref{eq:Linf_bnd.1.0.3}, we have that
\begin{align}\label{eq:Linf_bnd.2.0}
    \smash{\mathtt{C}_u = \tfrac{1}{\beta m} \||u-u_\infty| - \mathtt{C}_u\|_{1,\Gamma_I^>} \ge \tfrac{1}{\beta m} \|\mathtt{T}_\lambda u\|_{1,\Gamma_I^\lambda} = \tfrac{1}{\beta m} \|\mathtt{T}_\lambda u \|_{1,\Gamma_I}\,.}
\end{align}
Using Friedrich's inequality (\textit{cf}.\ \eqref{lem:poin_cont}), \eqref{eq:Linf_bnd.2.0} in \eqref{eq:Linf_bnd.2}, Hölder's inequality together with \eqref{eq:Linf_bnd.1.0.1}, the Sobolev~and~trace embedding $H^1(\Omega)\hookrightarrow L^{p}(\Omega)\cap L^{q}(\partial\Omega)\hookrightarrow L^{2\tilde{r}'}(\Omega)\cap L^{2\tilde{s}'}(\partial\Omega)$, and the $\varepsilon$-Young inequality, for every $\varepsilon>0$,~we~find~that 
\begin{align*} 
\begin{aligned}
\|\mathtt{T}_\lambda u\|_{\Omega}^2+\|\nabla\mathtt{T}_\lambda u\|_{\Omega}^2&\lesssim
\kappa \|\nabla \mathtt{T}_\lambda u\|_{\Omega}^2+ \tfrac{1}{m}\|\mathtt{T}_\lambda u\|_{1,\Gamma_I}^2\\&\leq  (f, \mathtt{T}_\lambda u )_{\Omega^\lambda} + (g, \mathtt{T}_\lambda u)_{\Gamma_N^\lambda} \\&\leq \|f\|_{\tilde{r},\Omega} |\Omega^\lambda|^{\smash{\frac{1}{2\tilde{r}'}}}\|\mathtt{T}_\lambda u\|_{2\tilde{r}',\Omega} +  \|g\|_{\tilde{s},\Gamma_N} |\Gamma_N^\lambda|^{\smash{\frac{1}{2\tilde{s}'}}}\|\mathtt{T}_\lambda u\|_{2\tilde{s}',\Gamma_N}
\\&\lesssim (|\Omega^\lambda|^{\smash{\frac{b}{p}}} +  |\Gamma^\lambda|^{\smash{\frac{b}{q}}})(\|\mathtt{T}_\lambda u\|_{\Omega}+\|\nabla\mathtt{T}_\lambda u\|_{\Omega})
\\&\lesssim \tfrac{1}{\varepsilon}(|\Omega^\lambda|^{\smash{\frac{2}{p}}} +  |\Gamma^\lambda|^{\smash{\frac{2}{q}}})^b+\varepsilon(\|\mathtt{T}_\lambda u\|_{\Omega}^2+\|\nabla\mathtt{T}_\lambda u\|_{\Omega}^2)\,,
\end{aligned}
\end{align*}
which, for $\varepsilon>0$ sufficiently small, implies that
\begin{align}\label{eq:Linf_bnd.3}
    \smash{\|\mathtt{T}_\lambda u\|_{\Omega}^2+\|\nabla\mathtt{T}_\lambda u\|_{\Omega}^2\lesssim (|\Omega^\lambda|^{\smash{\frac{2}{p}}} + |\Gamma^\lambda|^{\smash{\frac{2}{q}}})^b\,,}
\end{align}
By the Sobolev~and~trace embedding $H^1(\Omega) \hookrightarrow  L^{p}(\Omega)\cap L^{q}(\partial\Omega)$ as well as $\vert \mathtt{T}_\lambda u\vert=\vert u\vert-\lambda\ge \lambda'-\lambda $ a.e.\ in $\Omega^{\lambda'}\subseteq \Omega^\lambda$ and a.e.\ on $\Gamma^{\lambda'}\subseteq \Gamma^\lambda$, for every $\lambda' > \lambda$, we have that
\begin{align}\label{eq:Linf_bnd.4}
\begin{aligned}
     \smash{\|\mathtt{T}_\lambda u\|_{\Omega}^2+\|\nabla\mathtt{T}_\lambda u\|_{\Omega}^2
     \gtrsim  \| \mathtt{T}_\lambda u \|_{p,\Omega^\lambda}^2 + \| \mathtt{T}_\lambda u \|_{q,\Gamma^\lambda}^2 
     \gtrsim (\lambda'-\lambda)^2 (|\Omega^{\lambda'}|^{\smash{\frac{2}{p}}} + |\Gamma^{\lambda'}|^{\smash{\frac{2}{q}}})\,.}
     \end{aligned}
\end{align}
Putting \eqref{eq:Linf_bnd.3} and \eqref{eq:Linf_bnd.4} together, for every $\lambda' \hspace{-0.1em}>\hspace{-0.1em} \lambda$, we find that
\begin{align}\label{eq:Linf_bnd.5}
\smash{|\Omega^{\lambda'}|^{\smash{\frac{2}{p}}} + |\Gamma^{\lambda'}|^{\smash{\frac{2}{q}}}  \lesssim (\lambda'-\lambda)^{-2}(|\Omega^\lambda|^{\smash{\frac{2}{p}}} + |\Gamma^\lambda|^{\smash{\frac{2}{q}}})^b,}
\end{align}
and we conclude from Lemma \ref{lem:Stampacchia} (where $K>0$ is the hidden constant in $\lesssim$ from \eqref{eq:Linf_bnd.5}, $a=2$, $b\in (1,2)$ fixed as above, and $\varphi\colon [\lambda_0,+\infty)\to [0,\infty)$, defined by ${\varphi(\lambda)\coloneqq \vert\Omega^{\lambda}\vert^{\smash{\frac{2}{p}}} +\vert \Gamma^{\lambda}\vert^{\smash{\frac{2}{q}}}}$~for~all~$\lambda>\lambda_0$) the existence of a constant $\lambda^*\coloneqq 2^{\smash{\frac{b}{b-1}}}K^{\smash{\frac{1}{a}}}\varphi(\lambda_0)^{\smash{\frac{b-1}{a}}}+\lambda_0\in [\lambda_0,\infty)$ such that ${\vert\Omega^{\lambda^*}\vert^{\smash{\frac{2}{p}}} +\vert \Gamma^{\lambda^*}\vert^{\smash{\frac{2}{q}}}=0}$. In~other~words, we have that $\vert u\vert\leq \lambda^*$ a.e.\ in $\Omega$ and a.e.\ on $\partial\Omega$, \textit{i.e.},   $u \in L^\infty(\Omega) \cap L^\infty(\partial \Omega)$ and, due to \eqref{prop:optimality.2},  $\smash{\widetilde{\mathtt{d}}} \in L^\infty(\Gamma_I)$.
\end{proof}\newpage


 \section{The discrete insulation problem}\label{sec:discrete_problem}

\hspace{5mm}Given a \emph{discrete heat source} $f_h\in \mathcal{S}^1(\mathcal{T}_h)$, a \emph{discrete heat flux} $g_h\in \mathcal{S}^1(\mathcal{S}_h^N)$, jointly satisfying the \textit{discrete non-trivial net heat input condition}
\begin{align}\label{eq:discrete_data_condition}
    (f_h,1)_{\Omega}+(g_h,1)_{\Gamma_N}\neq 0\,,
\end{align}
and a \emph{discrete ambient temperature} $u_{\infty}^h\in \mathcal{S}^1(\mathcal{S}_h^{I})$, we seek a pair $(u_h,\smash{\widetilde{\mathtt{d}}_h})\in \mathcal{S}^1(\mathcal{T}_h)\times\mathcal{S}^1(\mathcal{S}_h^{I})$, consisting of a \emph{discrete temperature distribution} and a \emph{discrete distribution function (of the insulation material)} that minimize
 the \emph{discrete heat loss} functional ${\smash{E}_h\colon \mathcal{S}^1(\mathcal{T}_h)\times\mathcal{S}^1(\mathcal{S}_h^{I})  \to \mathbb{R}\cup\{+\infty\}}$, for every $(v_h,\eta_h)\in  \mathcal{S}^1(\mathcal{T}_h)\times\mathcal{S}^1(\mathcal{S}_h^{I})   $ defined by
\begin{equation} \label{def:Eh} 
 \smash{E}_h(v_h,\eta_h)\coloneqq \tfrac{\kappa}{2}\|\nabla v_h\|^2_{\Omega}  + \tfrac{\beta}{2}\|(1 + \beta \eta_h)^{-\smash{\frac{1}{2}}}(v_h-u_{\infty}^h)\|_{\Gamma_{I},h}^2 - (f_h, v_h)_{\Omega}-(g_h,v_h)_{\Gamma_N} + I_{\mathcal{H}_m^h}(\eta_h)\,,
\end{equation}
where\vspace{-0.5mm}
\begin{align*}
        \smash{\mathcal{H}_m^h\coloneqq \bigl\{ \eta_h\in \mathcal{S}^1(\mathcal{S}_h^{I}) \mid \eta_h\ge 0\text{ on }\Gamma_{I}\,,\; (\eta_h,1)_{\Gamma_{I},h}=m \bigr\}\,.}
\end{align*}
Note that 
the discrete non-trivial net heat input condition \eqref{eq:discrete_data_condition} follows from its continuous~counter\-part~\eqref{eq:data_condition} if, \textit{e.g.}, $f_h \in \mathcal{S}^1(\mathcal{T}_h)$ and  $g_h \in \mathcal{S}^1(\mathcal{S}^{N}_h)$~are~global $L^2$-projections of $f\in  L^2(\Omega)$ and $g\in (H^{\smash{\frac{1}{2}}}(\Gamma_N))^*$ onto $\mathcal{S}^1(\mathcal{T}_h)$~and~$\mathcal{S}^1(\mathcal{S}^{N}_h)$,~\mbox{respectively}, as  $(f_h,1)_{\Omega}=(f,1)_{\Omega}$~and~${( g_h,1)_{\Gamma_N}=\langle g,1\rangle_{\Gamma_N}}$.\vspace{-1.5mm}

\subsection{Optimality conditions}

\hspace{5mm}The following proposition identifies optimality conditions that characterize a minimizing pair $(u_h,\smash{\widetilde{\mathtt{d}}_h})\in \mathcal{S}^1(\mathcal{T}_h)\times\mathcal{S}^1(\mathcal{S}_h^{I}) $ of the discrete heat loss functional \eqref{def:Eh}.\enlargethispage{5mm}

\begin{proposition}[Optimality conditions]\label{prop:discrete_optimality}
    A pair $(u_h,\smash{\widetilde{\mathtt{d}}_h})\in  \mathcal{S}^1(\mathcal{T}_h)\times\mathcal{S}^1(\mathcal{S}_h^{I})   $ is minimal for the discrete heat loss functional \eqref{def:Eh} if and only if the following optimality conditions~are~satisfied:
    \begin{itemize}[noitemsep,topsep=2pt,leftmargin=!,labelwidth=\widthof{(iii)}]
        \item[(i)] There exists a constant $\mathtt{C}_{\smash{u_h}}>0$ such that 
        \begin{align}\label{prop:discrete_optimality.1}
            \smash{\mathtt{C}_{\smash{u_h}}=\tfrac{1}{\beta m}\|I_h\{(\vert u_h-u_{\infty}^h\vert-\mathtt{C}_{\smash{u_h}})_+\}\|_{1,\Gamma_{I}}\,;}
        \end{align}
        \item[(ii)] The function $\smash{\widetilde{\mathtt{d}}_h}\in \smash{\mathcal{S}^1(\mathcal{S}_h^{I})}$  satisfies
        \begin{align}\label{prop:discrete_optimality.2}
            \smash{\smash{\widetilde{\mathtt{d}}_h}=\tfrac{1}{\beta \mathtt{C}_{\smash{u_h}}}I_h\{(\vert u_h-u_{\infty}^h\vert-\mathtt{C}_{\smash{u_h}})_+\}\quad\text{ on }\Gamma_{I}\,;}
        \end{align}
        \item[(iii)] The pair $(u_h,\smash{\widetilde{\mathtt{d}}_h})\in  \mathcal{S}^1(\mathcal{T}_h)\times\mathcal{S}^1(\mathcal{S}_h^{I})   $, for every $v_h\in \mathcal{S}^1(\mathcal{T}_h) $, satisfies
        \begin{align}\label{prop:discrete_optimality.3}
           \kappa (\nabla u_h,\nabla  v_h)_{\Omega}+
          \beta ((1+\beta \smash{\widetilde{\mathtt{d}}_h})^{-1}(u_h-u_{\infty}^h),v_h)_{\Gamma_I,h}=(f_h,v_h)_{\Omega}+(g_h,v_h)_{\Gamma_N}\,.
        \end{align}
    \end{itemize}
\end{proposition} 

\begin{lemma}[Positivity of $\mathtt{C}_{\smash{u_h}}$]\label{lem:positivity_Cu_h}
    Let $(u_h,\smash{\widetilde{\mathtt{d}}_h})\hspace{-0.1em}\in  \hspace{-0.1em}\mathcal{S}^1(\mathcal{T}_h)\times\mathcal{S}^1(\mathcal{S}_h^{I})   $ be a pair such that the optimality condition \eqref{prop:discrete_optimality.3} is satisfied. Then, the constant $\mathtt{C}_{\smash{u_h}}\ge 0$ implicitly defined by \eqref{prop:discrete_optimality.1}~satisfies~$\mathtt{C}_{\smash{u_h}}>0$.
\end{lemma}
\begin{proof}
    Suppose that $\mathtt{C}_{\smash{u_h}}\hspace*{-0.1em}=\hspace*{-0.1em}0$. Then, \eqref{prop:discrete_optimality.1} implies that $u_h\hspace*{-0.1em}=\hspace*{-0.1em}u_\infty^h$ a.e.\ on $\Gamma_{I}$ and,~\mbox{consequently},~\eqref{prop:discrete_optimality.3}, for every $v_h\in \mathcal{S}^1(\mathcal{T}_h)$, reads
    \begin{align}\label{lem:positivity_Cu_h.1}
       \kappa (\nabla u_h,\nabla  v_h)_{\Omega}=(f_h,v_h)_{\Omega}+(g_h,v_h)_{\Gamma_N}\,.
    \end{align}
    For \hspace{-0.1mm}$v_h\hspace{-0.175em}=\hspace{-0.175em}1\hspace{-0.175em}\in\hspace{-0.175em} \mathcal{S}^1(\mathcal{T}_h)$ \hspace{-0.1mm}in \hspace{-0.1mm}\eqref{lem:positivity_Cu_h.1}, \hspace{-0.1mm}we \hspace{-0.1mm}find \hspace{-0.1mm}that \hspace{-0.1mm}$(f_h,1)_{\Omega}+(g_h,1)_{\Gamma_N}\hspace{-0.175em}=\hspace{-0.175em}0$, \hspace{-0.1mm}which \hspace{-0.1mm}contradicts~\hspace{-0.1mm}\mbox{assumption}~\hspace{-0.1mm}\eqref{eq:discrete_data_condition}. As a result, we conclude that $\mathtt{C}_{\smash{u_h}}>0$.
\end{proof}

\begin{proof}[Proof (of Proposition \ref{prop:discrete_optimality}).]
    A pair $(u_h,\smash{\widetilde{\mathtt{d}}_h})\in  \mathcal{S}^1(\mathcal{T}_h)\times \mathcal{S}^1(\mathcal{S}_h^I)$ is minimal for \eqref{def:Eh} if and only if
    \begin{subequations}\label{prop:discrete_optimality.4}
    \begin{align}\label{prop:discrete_optimality.4.1}
        0_{\smash{(\mathcal{S}^1(\mathcal{S}_h^{I}))^*}}&\in \partial_2 E_h(u_h,\smash{\widetilde{\mathtt{d}}_h})\,,\\
        0_{\smash{(\mathcal{S}^1(\mathcal{T}_h))^*}}&=\mathrm{D}_1 E_h(u_h,\smash{\widetilde{\mathtt{d}}_h})\,,\label{prop:discrete_optimality.4.2}
    \end{align}
    \end{subequations}
    where the identity \eqref{prop:discrete_optimality.4.2} is equivalent to \eqref{prop:discrete_optimality.3}, while the inclusion \eqref{prop:discrete_optimality.4.1} is equivalent to 
    \begin{align}\label{prop:discrete_optimality.5}
       \smash{\widetilde{\mathtt{d}}_h} \in \underset{\eta_h\in \mathcal{H}_m^h}{\textup{argmin}}{\{E_h(u_h,\eta_h)\}}\,.
    \end{align}
    Arguing as in the proof of \cite[Prop.\  4.1 \& Lem.\ 4.1]{PietraNitschScalaTrombetti2021}, with 
    modifications to fit~the~discrete~setting, we find that 
    the unique minimizer $\smash{\widetilde{\mathtt{d}}_h}\in \mathcal{H}_m^h$ is given~via \eqref{prop:discrete_optimality.2} 
    with  $\mathtt{C}_{\smash{u_h}}>0$~given~via~\eqref{prop:discrete_optimality.1}.
    \if0\begin{align}\label{prop:discrete_optimality.6}
            \mathtt{C}_{\smash{u_h}}=\tfrac{1}{\vert \{I_h\{\vert u_h-u_{\infty}^h\vert\} \ge \mathtt{C}_{\smash{u_h}}\}\vert+\beta m}\|I_h\{\vert u_h-u_{\infty}^h\vert\}\|_{1,\{I_h\{\vert u_h-u_{\infty}^h\vert\} \ge \mathtt{C}_{\smash{u_h}}\}}\,,
    \end{align}
    which is equivalent to \eqref{prop:discrete_optimality.1}.\fi 
\end{proof}\newpage

\begin{remark}[on Proposition \ref{prop:discrete_optimality}]
    Since \eqref{prop:discrete_optimality.5}  is equivalent to $-\beta I_h\{(1+\beta \smash{\widetilde{\mathtt{d}}_h})^{-2}(u_h- u_\infty^h)^2\}\in I_h\{\partial I_{\mathcal{H}_m^h}(\smash{\widetilde{\mathtt{d}}_h})\}$, 
    the optimality conditions \eqref{prop:discrete_optimality.1} and \eqref{prop:discrete_optimality.2} are jointly equivalent to the optimality condition\vspace{-2mm}
    \begin{align}
        ((1+\beta \smash{\widetilde{\mathtt{d}}_h})^{-2}(u_h - u_\infty^h)^2,\smash{\widetilde{\mathtt{d}}_h} - \eta_h)_{\Gamma_I,h}\ge 0\quad\text{ for all }\eta_h\in \mathcal{H}_m^h\,.\label{prop:discrete_optimality.3b}
    \end{align} 
\end{remark}

\subsection{$L^\infty$-bounds on the insulated boundary}\vspace{-0.5mm}\enlargethispage{2.5mm}

\hspace{5mm}The \textit{a priori} error analysis in Section \ref{sec:apriori}, in fact, does not only hinge on 
the assumption that both $u,\smash{\widetilde{\mathtt{d}}} \in L^\infty(\Gamma_I)$, but also the assumption of uniform $L^\infty(\Gamma_I)$-bounds for $u_h\in \mathcal{S}^1(\mathcal{T}_h)$~and~$\smash{\widetilde{\mathtt{d}}_h}\in \mathcal{H}_m^h$.
The following lemma affirms this assumption provided the heat source $f\in L^2(\Omega)$ and heat flux $g\in (H^{\frac{1}{2}}(\Gamma_N))^*$ are sufficiently integrable, the ambient temperature $u_\infty$ is essentially~bounded, and the following non-obtuseness condition is satisfied.

\begin{assumption}[Non-obtuse insulation triangulation condition] \label{assum:nonobtuse_mesh} Suppose $\mathcal{T}_h$ is \emph{non-obtuse}  in the sense of \cite{DieningScharleSuli:2021}, \textit{i.e.}, for every $\nu,\nu'\in \mathcal{N}_h$ with $\nu \ne \nu'$,~we~have~that\vspace{-0.5mm}
\begin{align}\label{assum:nonobtuse_mesh.0}
    \nabla \phi_{\nu}\cdot \nabla \phi_{\nu'} \le 0\quad\text{ a.e.\ in }\Omega\,.
\end{align}
     \end{assumption}


\begin{lemma} \label{lem:Linf_discrete}
 Let Assumption \ref{assum:nonobtuse_mesh} be satisfied. Then, if there exist  $r>\frac{d}{2}$ and $s>d-1$~such~that $\sup_{h>0}{\{\|f_h\|_{r,\Omega}+\|g_h\|_{s,\Gamma_N}+\|u_{\infty}^h\|_{\infty,\Gamma_I}\}}<+\infty$, 
then $\sup_{h>0}{\{\|u_h\|_{\infty,\Gamma_I}+\|\smash{\widetilde{\mathtt{d}}_h}\|_{\infty,\Gamma_I}\}}<+\infty$. 
\end{lemma}

The proof of Lemma \ref{lem:Linf_discrete} proceeds using a discrete version of a classic argument by Stampacchia (\textit{cf}.\ \cite{Stampacchia1965}), which makes use of a truncation argument based on the discrete truncation operator $\mathtt{T}_\lambda^h\colon W^{1,1}(\Omega)\to \mathcal{S}^1(\mathcal{T}_h)$, $\lambda>0$, for every $v\in W^{1,1}(\Omega)$ defined by
\begin{align} \label{eq:def_trunc_discrete}
    \mathtt{T}_\lambda^h v \coloneqq I_h\{\mathtt{T}_\lambda v\}=I_h\{\textup{sign}(v)(|v| - \lambda)_+\}\quad \text{ in }\Omega\,.
\end{align}
A key ingredient in the proof of the continuous counterpart (\textit{cf}.\ Lemma \ref{lem:Linf})~of~Lemma~\ref{lem:Linf_discrete}~is~the~gradient relation \eqref{eq:Linf_bnd.1.0.2}, which results from the chain rule for Sobolev functions (\textit{cf}.~\cite[Thm.~2.1.11]{Ziemer1989}). Owing to the application of the nodal interpolation operator to the truncation operator in \eqref{eq:def_trunc}, we cannot expect the latter to hold for the discrete truncation operator.~Instead, we resort to the following substitute that applies on non-obtuse triangulations.

\begin{lemma}\label{lem:discrete_chainrule}
    Let Assumption \ref{assum:nonobtuse_mesh} be satisfied. Then, 
    for every $v_h\in \mathcal{S}^1(\mathcal{T}_h)$, there holds
    \begin{align}\label{lem:discrete_chainrule.0}
        \vert \nabla \mathtt{T}_\lambda^hv_h\vert^2\leq  \nabla \mathtt{T}_\lambda^hv_h\cdot \nabla v_h\quad \text{ a.e.\ in }\Omega\,.
    \end{align}
\end{lemma}

\begin{proof}
    Let $T\in \mathcal{T}_h$ be fixed, but arbitrary and denote by $\mathcal{N}_h(T)\coloneqq \mathcal{N}_h\cap T$ the set of nodes in $T$. Then, there holds
    \begin{align} \label{lem:discrete_chainrule.1}
      \nabla \mathtt{T}_\lambda^h v_h\cdot \nabla v_h = - \sum_{\nu',\nu \in \mathcal{N}_h(T)}{\tfrac{1}{2}(\mathtt{T}_\lambda v_h(\nu') - \mathtt{T}_\lambda v_h(\nu))(v_h(\nu') - v_h(\nu)) \nabla \varphi_{\nu'} \cdot \nabla \varphi_\nu}\quad\text{ in }T\,.
    \end{align}
    By \cite[Lem.\ 4.8, eq. (4.28)]{Scharle:thesis}, for every $\nu',\nu \in \mathcal{N}_h(T)$ with $\nu'\ne \nu$, there holds
    \begin{align}\label{lem:discrete_chainrule.2}
        \smash{(\mathtt{T}_\lambda v_h(\nu') - \mathtt{T}_\lambda v_h(\nu))(v_h(\nu') - v_h(\nu)) \ge |\mathtt{T}_\lambda v_h(\nu') - \mathtt{T}_\lambda v_h(\nu)|^2\,,}
    \end{align}
    Therefore, since $ \nabla \varphi_{\nu'} \cdot \nabla \varphi_\nu\leq 0$ in $T$ (\textit{cf}.\ Assumption \ref{assum:nonobtuse_mesh}), from \eqref{lem:discrete_chainrule.2}  in \eqref{lem:discrete_chainrule.1}, it follows that
    \begin{align}\label{lem:discrete_chainrule.3} 
       \nabla \mathtt{T}_\lambda^hv_h\cdot \nabla v_h 
        \ge -\sum_{\nu',\nu \in \mathcal{N}_h(T)} {\tfrac{1}{2} |\mathtt{T}_\lambda v_h(\nu') - \mathtt{T}_\lambda v_h(\nu)|^2 \nabla \varphi_{\nu'} \cdot \nabla \varphi_\nu} = |\nabla \mathtt{T}_\lambda^hv_h|^2\quad \text{ in }T\,.
    \end{align}
    Eventually, since $\smash{T\in \mathcal{T}_h}$ was chosen arbitrarily, from \eqref{lem:discrete_chainrule.3}, we conclude the assertion.
\end{proof}
\begin{proof}[Proof (of Lemma \ref{lem:Linf_discrete}).] To begin with, we introduce the nodal sets
\begin{align*}
     \smash{\mathcal{N}_{h,I}^{>} \coloneqq \{\nu \in \mathcal{N}_h^{I} \mid |u_h(\nu)-u_\infty^h(\nu)| > \mathtt{C}_{\smash{u_h}}\}\,,\qquad \mathcal{N}_{h,I}^{\leq}\coloneqq \mathcal{N}_h^{I}\setminus \mathcal{N}_{h,I}^{>}\,,}
\end{align*} 
and fix $b \in (1,2)$. Then, depending on $d=3$ or $d=2$, we make the following definitions:
\begin{itemize}[noitemsep,topsep=2pt,leftmargin=!,labelwidth=\widthof{\quad$\bullet$ Case 1: ($d=3$).},font=\itshape]
    \item[$\bullet$ Case 1: ($d=3$).] We set $p = 6$, $q = 4$, $\tilde{r} = \frac{3}{3-b}$, $\tilde{s} = \frac{2}{2-b}$, implying $p = 2b \tilde{r}'$ and $q = 2b\tilde{s}'$;
    \item[$\bullet$ Case 2: ($d=2$).] We set $p = 2b \tilde{r}'$ and $q = 2b \tilde{s}'$ for fixed $\tilde{r} ,\tilde{s} > 1$.
\end{itemize}\newpage
\noindent Using the optimality conditions \eqref{prop:discrete_optimality.1}--\eqref{prop:discrete_optimality.3}, for every $v_h \in \mathcal{S}^1(\mathcal{T}_h)$, we find that
\begin{align} \label{eq:Linf_bnd_discrete.1}
\begin{aligned}
\kappa (\nabla u_h, \nabla v_h)_\Omega&+ \mathtt{C}_{\smash{u_h}}\beta ( \text{sgn}(u_h-\smash{u_{\infty}^h})\chi_{\smash{\mathcal{N}_{h,I}^{>}}},v_h)_{\Gamma_I,h}\\&\quad + \beta((u_h-\smash{u_\infty^h})\chi_{\smash{\mathcal{N}_{h,I}^{\le}}}, v_h)_{\Gamma_I,h} = (f_h,v_h)_\Omega + (g_h, v_h)_{\Gamma_N}\,.
\end{aligned}
\end{align}
 Next, for every $\lambda > \lambda_0\coloneqq \sup_{h>0}{\{\|u_\infty^h\|_{\infty,\Gamma_I} + \mathtt{C}_{\smash{u_h}}\}}$, we introduce the (nodal) sets 
\begin{align*}
    \mathcal{N}^\lambda_h   &\coloneqq \{ \nu \in \mathcal{N}_h \mid |u_h(\nu)| > \lambda\} \,,\\[-0.5mm] \mathcal{N}^{\lambda}_{h,\partial} &\coloneqq \{\nu \in \mathcal{N}_h^{\partial} \mid |u_h(\nu)| > \lambda\}\,,\\[-0.5mm]\mathcal{N}_{h,\mathrm{X}}^{\lambda}&\coloneqq \mathcal{N}_h^{\lambda}\cap \mathcal{N}^{\mathrm{X}}_h\,,\;\mathrm{X}\in \{I,N\}\,.
\end{align*} 
Then, by the definition of the discrete truncation operator \eqref{eq:def_trunc_discrete} and Lemma \ref{lem:discrete_chainrule}, we have that\vspace{-0.5mm}
\begin{subequations}\label{eq:Linf_bnd_discrete.1.0}
\begin{alignat}{2}\label{eq:Linf_bnd_discrete.1.0.1}
    \mathtt{T}_\lambda^h u_h&= 0&&\quad\text{ in }\smash{\mathcal{N}_h\setminus \mathcal{N}^{\lambda}_h}\text{ and }\text{on }\smash{\mathcal{N}_h^{\partial}\setminus \mathcal{N}_{h,\partial}^{\lambda}}\,,\\[-0.5mm]\label{eq:Linf_bnd_discrete.1.0.2}
    \vert \nabla \mathtt{T}_\lambda^h u_h\vert^2 &\leq  \nabla \mathtt{T}_\lambda^h u_h\cdot \nabla u_h&&\quad \text{ a.e.\ in }\Omega\,,\\[-0.5mm]
    \text{sgn}(\mathtt{T}_\lambda^h u_h) &=\text{sgn}(u_h)=\text{sgn}(u_h - u_\infty^h)&&\quad\text{ on }\smash{\mathcal{N}_{h,I}^{\lambda}\subseteq \mathcal{N}_{h,I}^{>}}\,,\label{eq:Linf_bnd_discrete.1.0.3}
\end{alignat}
\end{subequations}
where we used in \eqref{eq:Linf_bnd_discrete.1.0.3} that $\lambda > u_\infty^h + \mathtt{C}_{\smash{u_h}}$ a.e.\ on $\Gamma_I$ and $-\lambda < u_\infty^h - \mathtt{C}_{\smash{u_h}}$~a.e.~on~$\Gamma_I$~and,~thus, $\mathcal{N}_{h,I}^{\lambda}\hspace{-0.175em} \subseteq \hspace{-0.175em} \mathcal{N}_{h,I}^{>}$ \hspace{-0.15mm}for \hspace{-0.15mm}$\lambda \hspace{-0.175em}>\hspace{-0.175em} \lambda_0$.
\hspace{-0.15mm}As \hspace{-0.15mm}a \hspace{-0.15mm}result, \hspace{-0.15mm}due \hspace{-0.15mm}to \hspace{-0.15mm}\eqref{eq:Linf_bnd_discrete.1.0}, \hspace{-0.15mm}by \hspace{-0.15mm}choosing \hspace{-0.15mm}$v_h\hspace{-0.175em} =\hspace{-0.175em} \mathtt{T}_\lambda^h u_h \hspace{-0.175em}\in \hspace{-0.175em}\mathcal{S}^1(\mathcal{T}_h)$~\hspace{-0.15mm}in~\hspace{-0.15mm}\eqref{eq:Linf_bnd_discrete.1},~\hspace{-0.15mm}we~\hspace{-0.15mm}\mbox{obtain}\vspace{-0.5mm}
\begin{align} \label{eq:Linf_bnd_discrete.2}
\smash{\kappa \|\nabla \mathtt{T}_\lambda^h u_h \|_{\Omega}^2+ \mathtt{C}_{\smash{u_h}}\beta\|\mathtt{T}_\lambda^h u_h\|_{1,\Gamma_I}  \leq (f_h, \mathtt{T}_\lambda^h u _h)_{\Omega} + (g_h, \mathtt{T}_\lambda^h u_h )_{\Gamma_N}\,.}
\end{align}
Moreover,  due to \eqref{eq:Linf_bnd_discrete.1.0.1} and \eqref{eq:Linf_bnd_discrete.1.0.3}, we have that\vspace{-0.5mm}
\begin{align}\label{eq:Linf_bnd_discrete.2.0}
    \smash{\mathtt{C}_{\smash{u_h}} = \tfrac{1}{\beta m} \|I_h\{(|u_h-u_\infty^h| - \mathtt{C}_{\smash{u_h}})_+\}\|_{1,\Gamma_I} \ge  \tfrac{1}{\beta m} \|\mathtt{T}_\lambda^h u_h \|_{1,\Gamma_I}\,.}
\end{align}
Using Friedrich's inequality (\textit{cf}.\ \eqref{lem:poin_cont}), \eqref{eq:Linf_bnd_discrete.2.0} in \eqref{eq:Linf_bnd_discrete.2}, Hölder's inequality,  the 
notation
\begin{align}\label{eq:Linf_bnd_discrete.2.1}
    \vert \mathcal{N}_{h}^{\lambda}\vert \coloneqq \sum_{\nu\in \smash{\mathcal{N}_{h}^{\lambda}}}{h_\nu^{d}}\sim \bigl\|\chi_{\smash{\operatorname{supp}\mathtt{T}_\lambda^h u_h}}\bigr\|_{s,\Omega}^{s}\,,\qquad\vert \mathcal{N}_{h,\partial}^{\lambda}\vert \coloneqq \sum_{\nu\in \smash{\mathcal{N}_{h,\partial}^{\lambda}}}{h_\nu^{d-1}}\sim \bigl\|\chi_{\smash{\operatorname{supp}\mathtt{T}_\lambda^h u_h}}\bigr\|_{s,\partial\Omega}^{s}\,,
\end{align}
where the implicit constants in the  node-based norm equivalence $\sim $ depend only~on~the~shape~regularity of $\{\mathcal{T}_h\}_{h>0}$ and $s\in [1,\infty)$ (\textit{cf}.\ \cite[Lem.\ 3.4]{Bartels15}),
the Sobolev~and~trace embedding $H^1(\Omega)\hookrightarrow L^{p}(\Omega)\cap L^{q}(\partial\Omega)\hookrightarrow L^{2\tilde{r}'}(\Omega)\cap L^{2\tilde{s}'}(\partial\Omega)$, and the $\varepsilon$-Young inequality, for~every~$\varepsilon>0$,~we~find~that 
\begin{align*} 
\begin{aligned}
\|\mathtt{T}_\lambda^h u_h\|_{\Omega}^2+\|\nabla\mathtt{T}_\lambda ^hu_h\|_{\Omega}^2&\lesssim
\kappa \|\nabla \mathtt{T}_\lambda^h u_h\|_{\Omega}^2+ \tfrac{1}{m}\|\mathtt{T}_\lambda^h u_h\|_{1,\Gamma_I}^2\\&\leq  (f_h, \mathtt{T}_\lambda^h u_h )_{\Omega} + (g_h, \mathtt{T}_\lambda ^h u_h)_{\Gamma_N} \\&\lesssim \|f_h\|_{\tilde{r},\Omega} |\mathcal{N}^\lambda_h|^{\smash{\frac{1}{2\tilde{r}'}}}\|\mathtt{T}_\lambda^h u_h\|_{2\tilde{r}',\Omega} +  \|g_h\|_{\tilde{s},\Gamma_N} 
\vert \mathcal{N}_{h,\partial}^{\lambda}\vert^{\smash{\frac{1}{2\tilde{s}'}}}\|\mathtt{T}_\lambda ^hu_h\|_{2\tilde{s}',\Gamma_N}
\\&\lesssim (|\mathcal{N}^\lambda_h|^{\smash{\frac{b}{p}}} +  \vert \mathcal{N}_{h,\partial}^{\lambda}\vert^{\smash{\frac{b}{q}}})(\|\mathtt{T}_\lambda ^hu_h\|_{\Omega}+\|\nabla\mathtt{T}_\lambda^h u_h\|_{\Omega})
\\&\lesssim \tfrac{1}{\varepsilon}(|\mathcal{N}^\lambda_h|^{\smash{\frac{2}{p}}} +  \vert \mathcal{N}_{h,\partial}^{\lambda}\vert^{\smash{\frac{2}{q}}})^b+\varepsilon(\|\mathtt{T}_\lambda u_h^h\|_{\Omega}^2+\|\nabla\mathtt{T}_\lambda^h u_h\|_{\Omega}^2)\,,
\end{aligned}
\end{align*}
which, for $\varepsilon>0$ sufficiently small, implies that
\begin{align}\label{eq:Linf_bnd_discrete.3}
    \smash{\|\mathtt{T}_\lambda^h u_h\|_{\Omega}^2+\|\nabla\mathtt{T}_\lambda^h u_h\|_{\Omega}^2\lesssim (|\mathcal{N}^\lambda_h|^{\smash{\frac{2}{p}}} + |\mathcal{N}_{h,\partial}^{\lambda}|^{\smash{\frac{2}{q}}})^b\,,}
\end{align}
By the Sobolev~and~trace embedding $H^1(\Omega)\hookrightarrow L^{p}(\Omega)\cap L^{q}(\partial\Omega)$ as well as $\mathtt{T}_\lambda^h u_h\ge |u_h| - \lambda\ge \lambda'-\lambda$ in $\mathcal{N}_h^{\lambda'}\subseteq \mathcal{N}_h^{\lambda}$  and on $\mathcal{N}_{h,\partial}^{\lambda'}\subseteq \mathcal{N}_{h,\partial}^{\lambda}$ together with \eqref{eq:Linf_bnd_discrete.2.1}, for every $\lambda' > \lambda$,~we have that\enlargethispage{6.5mm}
\begin{align}\label{eq:Linf_bnd_discrete.4} 
     \smash{\|\mathtt{T}_\lambda^h u_h\|_{\Omega}^2+\|\nabla\mathtt{T}_\lambda ^hu_h\|_{\Omega}^2\gtrsim  \| \mathtt{T}_\lambda^h u_h \|_{p,\Omega}^2 + \| \mathtt{T}_\lambda^h u_h \|_{q,\partial\Omega}^2 
     \gtrsim (\lambda'-\lambda)^2 (|\mathcal{N}^{\lambda'}_h|^{\smash{\frac{2}{p}}} + |\mathcal{N}^{\lambda'}_{h,\partial}|^{\smash{\frac{2}{q}}})\,.}
\end{align}
Putting \eqref{eq:Linf_bnd_discrete.3} and \eqref{eq:Linf_bnd_discrete.4} together, for every $\lambda' \hspace{-0.1em}>\hspace{-0.1em} \lambda$, we find that\vspace{-0.5mm}
\begin{align}\label{eq:Linf_bnd_discrete.5}
\smash{|\mathcal{N}^{\lambda'}_h|^{\smash{\frac{2}{p}}} + |\mathcal{N}^{\lambda'}_{h,\partial}|^{\smash{\frac{2}{q}}}  \lesssim (\lambda'-\lambda)^{-2}(|\mathcal{N}^{\lambda}_h|^{\smash{\frac{2}{p}}} + |\mathcal{N}^{\lambda}_{h,\partial}|^{\smash{\frac{2}{q}}})^b\,,}
\end{align}
and we conclude from Lemma \ref{lem:Stampacchia} (where $K>0$ is the hidden constant in $\lesssim$ from \eqref{eq:Linf_bnd_discrete.5}, $a=2$, $b\in (1,2)$ fixed as above, and $\varphi_h\colon \hspace{-0.1em}[\lambda_0,+\infty)\hspace{-0.15em}\to\hspace{-0.15em} [0,\infty)$, defined by ${\varphi_h(\lambda)\hspace{-0.15em}\coloneqq\hspace{-0.15em} \vert\mathcal{N}^{\lambda}_h\vert^{\smash{\frac{2}{p}}} \hspace{-0.15em}+\hspace{-0.15em}\vert \mathcal{N}^{\lambda}_{h,\partial}\vert^{\smash{\frac{2}{q}}}}$~for~all~${\lambda\hspace{-0.15em}>\hspace{-0.15em}\lambda_0}$) the existence of a constant $\lambda^*_h\hspace{-0.15em}\coloneqq \hspace{-0.15em}2^{\smash{\frac{b}{b-1}}}K^{\smash{\frac{1}{a}}}\varphi_h(\lambda_0)^{\smash{\frac{b-1}{a}}}\hspace{-0.15em}+\hspace{-0.15em}\lambda_0\in [\lambda_0,\infty)$ such that ${\vert\mathcal{N}^{\lambda^*_h}_h\vert^{\smash{\frac{2}{p}}}\hspace{-0.15em} +\hspace{-0.15em}\vert \mathcal{N}^{\lambda_h^*}_{h,\partial}\vert^{\smash{\frac{2}{q}}}\hspace{-0.15em}=\hspace{-0.15em}0}$. In~other~words, we have that $\vert u_h\vert\leq \lambda^*_h$
 a.e.\ in $\Omega$ and a.e.\ on $\partial\Omega$. Since, due to \eqref{eq:Linf_bnd_discrete.2.1}, we have that $\vert \smash{\mathcal{N}^{\lambda_0}_h}\vert+\vert \smash{\mathcal{N}^{\lambda_0}_{h,\partial}}\vert\lesssim \vert \Omega\vert+\vert \partial\Omega\vert$, there holds 
 $\lambda^*_h\lesssim 2^{\smash{\frac{b}{b-1}}}K^{\smash{\frac{1}{a}}}(\vert \Omega\vert+\vert \partial\Omega\vert)^{\smash{\frac{b-1}{a}}}+\lambda_0\eqqcolon \lambda^*$,~so~that,~eventually, $\sup_{h>0}{\{\|u_h\|_{\infty,\Omega}+\|u_h\|_{\infty,\partial\Omega}\}}<+\infty$ and, due to \eqref{prop:discrete_optimality.2},~${\sup_{h>0}{\{\|\smash{\widetilde{\mathtt{d}}_h}\|_{\infty,\Gamma_I}\}}<+\infty}$. 
\end{proof}

\subsection{Well-posedness, stability, and convergence}

\hspace{5mm}On the basis of the discrete optimality conditions
\eqref{prop:discrete_optimality.1}--\eqref{prop:discrete_optimality.3}
(and the continuous optimality conditions
\eqref{prop:optimality.1}--\eqref{prop:optimality.3}),
we establish well-posedness, stability, and convergence of the
discretization introduced in Section~\ref{sec:discrete_problem}.\enlargethispage{2.5mm}

\begin{theorem}[Well-posedness, stability, and convergence of \eqref{prop:discrete_optimality.1}--\eqref{prop:discrete_optimality.3}]\label{thm:convergence_p1}
    Assume that 
    \begin{subequations}\label{thm:convergence_p1.0}
    \begin{alignat}{3}
    	f_h&\rightharpoonup f&&\quad \text{ in }L^2(\Omega)&&\quad (h\to 0^+)\,,\label{thm:convergence_p1.0.1}\\ 
        g_h&\rightharpoonup g&&\quad \text{ in }H^{\smash{-\frac{1}{2}}}(\Gamma_{N})&&\quad (h\to 0^+)\,,\label{thm:convergence_p1.0.2}\\ 
    	u_{\infty}^h&\to u_{\infty}&&\quad \text{ in }L^2(\Gamma_{I})&&\quad (h\to 0^+)\,.\label{thm:convergence_p1.0.3}
    \end{alignat}
	\end{subequations}
     Then, the following statements apply:
    \begin{itemize}[noitemsep,topsep=2pt,leftmargin=!,labelwidth=\widthof{(iii)}]
        \item[(i)]\hypertarget{thm:convergence_p1.i}{} \emph{Well-posedness:} There exists a triple $(u_h,\mathtt{C}_{\smash{u_h}},\smash{\widetilde{\mathtt{d}}_h})\hspace{-0.1em}\in\hspace{-0.1em}\mathcal{S}^{1}(\mathcal{T}_h)\times \mathbb{R}_{>0}\times \mathcal{H}_m^h$~\mbox{satisfying}~\mbox{\eqref{prop:discrete_optimality.1}--\eqref{prop:discrete_optimality.3}};
        \item[(ii)]\hypertarget{thm:convergence_p1.ii}{} \label{thm:convergence.ii} \emph{Stability:}      There exists a constant $c_{\mathrm{stab}}>0$, which does not depend on $h>0$, such that
    \begin{align}\label{thm:convergence_p1.1}
        \mathtt{C}_{\smash{u_h}}+\|\smash{\widetilde{\mathtt{d}}_h}\|_{\Gamma_{I}}^2+\|u_h\|_{\Omega}^2+\|\nabla u_h\|_{\Omega}^2\leq c_{\mathrm{stab}}\,.
    \end{align}
        \item[(iii)]\hypertarget{thm:convergence_p1.iii}{}  \emph{Convergence:} There holds\vspace{-2.5mm}
        \begin{subequations}\label{thm:convergence_p1.2}
            \begin{alignat}{3}
                u_h&\rightharpoonup u&&\quad\text{ in }H^1(\Omega)&&\quad (h\to 0^+)\,,\label{thm:convergence_p1.2.1}\\[-0.5mm]
                \smash{\widetilde{\mathtt{d}}_h}&\to \smash{\widetilde{\mathtt{d}}}&&\quad \text{ in }L^2(\Gamma_{I})&&\quad (h\to 0^+)\,,\label{thm:convergence_p1.2.2}\\[-0.5mm]
                \mathtt{C}_{\smash{u_h}}&\to \mathtt{C}_u&&\quad \text{ in }\mathbb{R}&&\quad (h\to 0^+)\,,\label{thm:convergence_p1.2.3}
        \end{alignat}
    \end{subequations}
    where $(u,\mathtt{C}_u,\smash{\widetilde{\mathtt{d}}})\in H^1(\Omega)\times \mathbb{R}_{> 0}\times \mathcal{H}_m$ satisfies the optimality conditions \eqref{prop:optimality.1}--\eqref{prop:optimality.3} and, thus, $(u,\smash{\widetilde{\mathtt{d}}})\in H^1(\Omega)\times \mathcal{H}_m$ is a minimizer of \eqref{def:E}.
    \end{itemize}
   
\end{theorem}

\begin{proof} \emph{1.\ Well-posedness:} 
Since $\smash{E}_h\colon \mathcal{S}^1(\mathcal{T}_h)\times\mathcal{S}^1(\mathcal{S}_h^{I})  \to \mathbb{R}\cup\{+\infty\}$ is proper, continuous,~convex, and weakly coercive,  since, \textit{e.g.}, for every $(v_h,\eta_h)\in \mathcal{S}^1(\mathcal{T}_h)\times \mathcal{H}_m^h$, due to \eqref{lem:std_prop_lumping.1} and $\|\cdot\|_{\infty,\Gamma_I}\lesssim h^{-d+1}\|\cdot\|_{1,\Gamma_I}=h^{-d+1}m$ on $\mathcal{H}_m^h$ (\textit{cf}.\ \cite[Lem.~3.5]{Bartels15}),~there~holds
\begin{align*}
    \smash{\tfrac{\beta}{2}\|(1 + \beta \eta_h)^{-\smash{\frac{1}{2}}}(v_h-u_{\infty}^h)\|_{\Gamma_{I},h}^2\ge \tfrac{\beta(d+1)}{2(1+\beta\|\eta_h\|_{\infty,\Gamma_I})}\|v_h-u_\infty^h\|_{\Gamma_I}^2\gtrsim \|v_h-u_\infty^h\|_{\Gamma_I}^2\,,}
\end{align*} 
the existence of a triple $(u_h,\mathtt{C}_{\smash{u_h}},\smash{\widetilde{\mathtt{d}}_h})\in\mathcal{S}^{1}(\mathcal{T}_h)\times \mathbb{R}_{>0}\times \mathcal{H}_m^h$~\mbox{satisfying}~\mbox{\eqref{prop:discrete_optimality.1}--\eqref{prop:discrete_optimality.3}} follows from
the direct method in the calculus of variations together with Proposition \ref{prop:discrete_optimality}.

\emph{2.\ Stability:} If we choose $v_h=u_h\in \mathcal{S}^{1}(\mathcal{T}_h)$ in the discrete optimality condition \eqref{prop:discrete_optimality.3}, using the $\varepsilon$-Young inequality, for every $\varepsilon>0$, we find that
\begin{align}\label{thm:convergence.3} 
	\begin{aligned} 
    \kappa\|\nabla u_h\|_{\Omega}^2+\beta \|(1+\beta \smash{\widetilde{\mathtt{d}}_h})^{-\smash{\frac{1}{2}}}(u_h-u_{\infty}^h)\|_{\Gamma_{I},h}^2&=(f_h,u_h)_{\Omega}+(g_h,u_h)_{\Gamma_{N}}\\&\quad-\beta((1+\beta \smash{\widetilde{\mathtt{d}}_h})^{-1}u_{\infty}^h,u_h)_{\Gamma_{I},h}
    \\& \leq \tfrac{1}{2\varepsilon} \|f_h\|_{\Omega}^2+\tfrac{\varepsilon}{2}\|u_h\|_{\Omega}^2
    \\&\quad+\tfrac{1}{2\varepsilon} \|g_h\|_{\smash{H^{\smash{-\frac{1}{2}}}(\Gamma_N)}}^2+\tfrac{\varepsilon}{2}\|u_h\|_{\smash{H^{\smash{\frac{1}{2}}}(\Gamma_N)}}^2
    \\&\quad+\beta\bigl(\tfrac{1}{2\varepsilon} \|(1+\beta \smash{\widetilde{\mathtt{d}}_h})^{-\smash{\frac{1}{2}}} u_{\infty}^h\|_{\Gamma_{I},h}^2\\&\quad\quad\quad+\tfrac{\varepsilon}{2} \|(1+\beta \smash{\widetilde{\mathtt{d}}_h})^{-\smash{\frac{1}{2}}}u_h\|_{\Gamma_{I},h}^2\bigr)\,.
	\end{aligned}\hspace*{-5mm}
\end{align}
Due to the optimality condition \eqref{prop:discrete_optimality.2}, setting $\smash{\Gamma_{h,I}^{>}\coloneqq \{I_h\{\vert u_h-u_{\infty}^h\vert\} \ge \mathtt{C}_{\smash{u_h}}\}}$ and $\smash{\Gamma_{h,I}^{\leq}\coloneqq \Gamma_I\setminus \Gamma_{h,I}^{>}}$ as well as using Hölder's inequality,
we observe that
\begin{align}\label{thm:convergence.4} 
	\begin{aligned} 
   	 \|(1+\beta \smash{\widetilde{\mathtt{d}}_h})^{-\smash{\frac{1}{2}}}(u_h-u_{\infty}^h)\|_{\Gamma_{I},h}^2
    &=\|I_h\{\vert u_h-u_{\infty}^h\vert\}\|_{\smash{\Gamma_{h,I}^{\leq}}}^2+(I_h\{\vert u_h-u_{\infty}^h\vert\},\mathtt{C}_{\smash{u_h}})_{\smash{\Gamma_{h,I}^{\geq}}}
   	 \\
     &= \|I_h\{\vert u_h-u_{\infty}^h\vert\}\|_{\smash{\Gamma_{h,I}^{\leq}}}^2+\tfrac{1}{\vert \Gamma_{h,I}^{\geq}\vert + \beta m}\|I_h\{\vert u_h-u_{\infty}^h\vert\}\|_{1,\smash{\Gamma_{h,I}^{\geq}}}^2
    \\[-1mm]&\ge \tfrac{1}{\vert \Gamma_{h,I}^{\leq}\vert}\|I_h\{\vert u_h-u_{\infty}^h\vert\}\|_{1,\smash{\Gamma_{h,I}^{\leq}}}^2\\[-1mm]&\quad 
    +\tfrac{1}{\vert \Gamma_{h,I}^{\geq}\vert + \beta m}\|I_h\{\vert u_h-u_{\infty}^h\vert\}\|_{1,\smash{\Gamma_{h,I}^{\geq}}}^2
    \\[-1mm]&\ge \smash{\tfrac{1}{\vert \Gamma_I\vert + \beta m}}\|I_h\{\vert u_h-u_{\infty}^h\vert\}\|_{1,\Gamma_{I}}^2\,.
	\end{aligned}
\end{align}
Using \eqref{thm:convergence.4}  in \eqref{thm:convergence.3}, the trace embedding $H^1(\Omega)\hookrightarrow H^{\frac{1}{2}}(\partial\Omega)$ with embedding constant $c_{\mathrm{Tr}}>0$, and  Friedrich's inequality (\textit{cf}.\ \eqref{lem:poin_cont}),~we~arrive~at 
\begin{align}\label{thm:convergence.5} 
	\begin{aligned} 
   \kappa \|\nabla u_h\|_{\Omega}^2&+\tfrac{\beta}{2}\bigl(\tfrac{1}{\vert \Gamma_I\vert + \beta m}\|I_h\{\vert u_h-u_{\infty}^h\vert\}\|_{1,\Gamma_{I}}^2+ \|(1+\beta \smash{\widetilde{\mathtt{d}}_h})^{-\smash{\frac{1}{2}}}(u_h-u_{\infty}^h)\|_{\Gamma_{I},h}^2\bigr) \\&\leq \tfrac{1}{2\varepsilon}\bigl(\|f_h\|_{\Omega}^2+\|g_h\|_{\smash{H^{\smash{-\frac{1}{2}}}(\Gamma_N)}}^2+\beta \|(1+\beta \smash{\widetilde{\mathtt{d}}_h})^{-\smash{\frac{1}{2}}} u_{\infty}^h\|_{\Gamma_{I},h}^2\bigr)\\&\quad+\tfrac{\varepsilon}{2}\bigl(
  (1+c_{\textrm{Tr}}^2) \|u_h\|_{\Omega}^2+ c_{\textrm{Tr}}^2\|\nabla u_h\|_{\Omega}^2+ \beta \|(1+\beta \smash{\widetilde{\mathtt{d}}_h})^{-\smash{\frac{1}{2}}}u_h\|_{\Gamma_{I},h}^2\bigr)
    \\&\leq \tfrac{1}{2\varepsilon}\bigl(\|f_h\|_{\Omega}^2+\|g_h\|_{\smash{H^{\smash{-\frac{1}{2}}}(\Gamma_N)}}^2+\beta \|u_{\infty}^h\|_{\Gamma_{I}}^2\bigr)
    \\&\quad +\tfrac{\varepsilon}{2}\bigl( (1+c_{\textrm{Tr}}^2)c_{\textrm{F}}^2\bigl(\|\nabla u_h\|_{\Omega}^2+2\|u_h-u_{\infty}^h\|_{1,\Gamma_{I}}^2+2\|u_{\infty}^h\|_{1,\Gamma_{I}}^2\bigr)
   \\&\quad+c_{\textrm{Tr}}^2\|\nabla u_h\|_{\Omega}^2+ \beta \|(1+\beta \smash{\widetilde{\mathtt{d}}_h})^{-\smash{\frac{1}{2}}}u_h\|_{\Gamma_{I},h}^2\\&\quad+2\beta \bigl(\|(1+\beta \smash{\widetilde{\mathtt{d}}_h})^{-\smash{\frac{1}{2}}}( u_h-u_{\infty}^h)\|_{\Gamma_{I},h}^2+ \|u_{\infty}^h\|_{\Gamma_{I},h}^2\bigr)\bigr)\,.
\end{aligned}
\end{align} 
Therefore, due to \eqref{thm:convergence_p1.0} and  \eqref{lem:std_prop_lumping.1}, and  $\|\cdot\|_{1,\Gamma_I}\sim \|I_h\{\vert \cdot\vert\}\|_{1,\Gamma_I}$ on $\mathcal{S}^1(\mathcal{S}_h^{I})$ (\textit{cf}.\ \cite[Lem.~3.4]{Bartels15}), choosing $\varepsilon>0$ small~enough~in~\eqref{thm:convergence.5}, we conclude that the claimed stability~estimate~\eqref{thm:convergence_p1.1}~\mbox{applies}.

\emph{3. Convergence:} Due to the stability estimate in \eqref{thm:convergence_p1.1} and the weak closedness of $\mathbb{R}_{\ge 0}$ and $\mathcal{H}_m$, there exists a triple $(u,\mathtt{C}_u,\smash{\widetilde{\mathtt{d}}})\in H^1(\Omega)\times \mathbb{R}_{\ge 0}\times \mathcal{H}_m$ such that (up to a subsequence) the convergences \eqref{thm:convergence_p1.2.1}--\eqref{thm:convergence_p1.2.3} apply. 
    It remains to establish that the triple $(u,\mathtt{C}_u,\smash{\widetilde{\mathtt{d}}})\in H^1(\Omega)\times \mathbb{R}_{\ge 0}\times \mathcal{H}_m$  satisfies the optimality conditions \eqref{prop:optimality.1}--\eqref{prop:optimality.3}:

    \textit{Step 1:} Since 
    $\Phi\coloneqq((a,c)\mapsto(|a|-c)_+)\in C^{0,1}(\mathbb{R}\times \mathbb{R})$, from Lemma \ref{lem:mass_lumping_lipschitz} together~with~\eqref{thm:convergence_p1.0.3}, \eqref{thm:convergence_p1.2.1}, and \eqref{thm:convergence_p1.2.3}, it follows  that
    \begin{align}\label{thm:convergence.6.0} 
    I_h\{(\vert u_h-u_\infty^h\vert-\mathtt{C}_{\smash{u_h}})_+\}\to (\vert u-u_\infty\vert-\mathtt{C}_u)_+\quad\text{ in }L^2(\Gamma_I)\quad (h\to 0^+)\,.
    \end{align}
    Therefore, if  we pass to the limit as $h\to 0^+$ in \eqref{prop:discrete_optimality.1}, using \eqref{thm:convergence.6.0}  and \eqref{thm:convergence_p1.2.3}, 
    we find that
    \begin{align}\label{thm:convergence.6} 
        \mathtt{C}_u=\tfrac{1}{\beta m}\|(|u-u_{\infty}|-\mathtt{C}_u)_+\|_{1,\Gamma_{I}}\,;
    \end{align}

    \textit{Step 2:} Let $v\hspace{-0.15em}\in\hspace{-0.15em} H^1(\Omega)$ be fixed, but arbitrary. Then, there exists a sequence ${v_h\hspace{-0.15em}\in\hspace{-0.15em} \mathcal{S}^1(\mathcal{T}_h)}$,~${h\hspace{-0.15em}>\hspace{-0.15em}0}$, such that $v_h\hspace{-0.1em}\to\hspace{-0.1em} v$ in $H^1(\Omega)$ $(h\hspace{-0.1em}\to\hspace{-0.1em} 0^+)$. Since $I_h\{\beta(1+\beta \smash{\widetilde{\mathtt{d}}_h})^{-1}(u_h-u_\infty^h)v_h\}\hspace{-0.1em}=\hspace{-0.1em}I_h\{\mathtt{T}_{\mathtt{C}_{\smash{u_h}}}(u_h-u_\infty^h)v_h\}$ on $\Gamma_I$ 
and 
$\Phi\coloneqq ((a,c)\mapsto\mathtt{T}_ca=\textup{sign}(a)\min\{\vert a\vert,c\})\in C^{0,1}(\mathbb{R}\times\mathbb{R})$, from
Lemma \ref{lem:mass_lumping_lipschitz} 
 together with \eqref{thm:convergence_p1.0.3}, \eqref{thm:convergence_p1.2.1}, and \eqref{thm:convergence_p1.2.3}, it follows that
\begin{align}\label{thm:convergence.7} 
    \beta((1+\beta \smash{\widetilde{\mathtt{d}}_h})^{-1}(u_h-u_\infty^h),v_h)_{\smash{\Gamma_{I}},h}\to ( \mathtt{T}_{\mathtt{C}_u}\{u-u_\infty\},v)_{\smash{\Gamma_{I}}}\quad (h\to 0^+)
\end{align}
Moreover, if we multiply \eqref{prop:discrete_optimality.2} with $\beta\mathtt{C}_{\smash{u_h}}>0$ and, subsequently, 
we pass to the limit as $h\to 0^+$, using \eqref{thm:convergence_p1.2.1},  \eqref{thm:convergence_p1.2.2}, and \eqref{thm:convergence.6.0} in doing so, we find that
\begin{align}\label{thm:convergence.8} 
   \beta\mathtt{C}_u \smash{\widetilde{\mathtt{d}}}=(\vert u-u_\infty\vert-\mathtt{C}_u)_+\quad\text{ a.e.\ on }\Gamma_I\,.
\end{align}
Next, we need to distinguish the cases $\mathtt{C}_u>0$ and  $\mathtt{C}_u=0$:
\begin{itemize}[noitemsep,topsep=2pt]
    \item[$\bullet$] \emph{Case ($\mathtt{C}_u>0$).} In this case, \eqref{thm:convergence.8} is equivalent to \eqref{prop:optimality.2}, so that  
    \begin{align}\label{thm:convergence.9}
    \beta(1+\beta \smash{\widetilde{\mathtt{d}}})^{-1}(u-u_\infty)=\mathtt{T}_{\mathtt{C}_u}\{u-u_\infty\}\quad \text{ a.e.\ on }\Gamma_I\,.
    \end{align}

    \item[$\bullet$] \emph{Case ($\mathtt{C}_u=0$).} In this case, due to \eqref{thm:convergence.8}, we have that $u=u_\infty$ a.e.\ on $\Gamma_I$,~so~that~\eqref{thm:convergence.9}~holds.

\end{itemize}

In summary, if we pass to the limit as $h\to 0^+$ in \eqref{prop:discrete_optimality.3}, using \eqref{thm:convergence_p1.0.1}--\eqref{thm:convergence_p1.0.3} and \eqref{thm:convergence_p1.2.1}--\eqref{thm:convergence_p1.2.3},  
    for every $v\in H^1(\Omega)$, we find that\vspace{-0.5mm}
    \begin{align}\label{thm:convergence.10} 
        \kappa(\nabla u ,\nabla v)_{\Omega}+\beta((1+\beta \smash{\widetilde{\mathtt{d}}})^{-1}(u-u_{\infty}),v)_{\smash{\Gamma_{I}}}=(f,v)_{\Omega}+\langle g,v\rangle_{\Gamma_N}\,.
    \end{align}

    \textit{Step 3:}  Due to Lemma \ref{lem:positivity_Cu}, from \eqref{thm:convergence.6} and \eqref{thm:convergence.10}, it follows that  $\mathtt{C}_u>0$, so that \eqref{thm:convergence.8} is equivalent to \eqref{prop:optimality.2}.

    In summary, the triple $(u,\mathtt{C}_u,\smash{\widetilde{\mathtt{d}}})\in H^1(\Omega)\times \mathbb{R}_{> 0}\times \mathcal{H}_m$  satisfies the optimality conditions \eqref{prop:optimality.1}--\eqref{prop:optimality.3} and, thus, by Proposition \ref{prop:optimality}, the pair $(u,\smash{\widetilde{\mathtt{d}}})\in H^1(\Omega)\times \mathcal{H}_m$ is a minimizer of \eqref{def:E}.
\end{proof}

\newpage
\section{Error analysis} \label{sec:apriori}

\hspace*{5mm}In this section, we derive \textit{a priori} error estimates for the discretization introduced in Section~\ref{sec:discrete_problem}.\linebreak
Throughout this section, we assume that $f_h \in \mathcal{S}^1(\mathcal{T}_h)$ and  $g_h \in \mathcal{S}^1(\mathcal{S}^{N}_h)$ are global $L^2$-projections of $f\hspace{-0.1em}\in\hspace{-0.1em}  L^2(\Omega)$ and $g\hspace{-0.1em}\in\hspace{-0.1em} (H^{\smash{\frac{1}{2}}}(\Gamma_N))^*$ onto $\mathcal{S}^1(\mathcal{T}_h)$ and $\mathcal{S}^1(\mathcal{S}^{N}_h)$,  respectively.~\mbox{Moreover},~we~\mbox{assume}~that $u_\infty\in  H^1(\Gamma_I)$ and set $u_\infty^h \coloneqq\mathcal{I}_h^{I} u_\infty\in  \mathcal{S}^1(\mathcal{T}_h)$, where $\mathcal{I}_h^{I}\colon H^1(\Gamma_I)\to\mathcal{S}^1(\mathcal{S}_h^{I})$ is the quasi-interpolation operator from Appendix \ref{sec:appendix_A}. 

\subsection{Error bounds for the temperature distribution}

\hspace{5mm}In this subsection, we derive an \textit{a priori} error estimate for the temperature distribution. As a by-product, we additionally obtain an \textit{a priori} error estimate for the distribution of the insulation material over $\operatorname{supp}(u-u_\infty)$.\enlargethispage{5mm}

\begin{theorem}\label{thm:H1_error}
    Let the assumptions of Lemma \ref{lem:Linf} and Lemma \ref{lem:Linf_discrete} be satisfied. Moreover, assume that $\{\mathcal{T}_h\}_{h>0}$ is quasi-uniform, that $u_\infty \in H^{1}(\Gamma_I) \cap L^\infty(\Gamma_I)$, and that $u \in H^{1+s}(\Omega)$, $s\in [\frac{1}{2},1]$, so that $u \in E_{\Delta}(\Omega)\hookrightarrow H^{1}(\Gamma_I)$ (since $\Delta u=-f \in L^2(\Omega)$ by \eqref{prop:optimality.3} for all $v\in H^1_0(\Omega)$) and, thus,  $\smash{\widetilde{\mathtt{d}}} \in H^{1}(\Gamma_I)$ (by \eqref{prop:optimality.2} and the chain rule for Sobolev functions).  Then, there holds 
    \begin{align}\label{thm:H1_error.0.1}
    \begin{aligned} 
        \tfrac{\kappa}{2}
\|\nabla(u_h-u)\|_\Omega^2
&+
\tfrac{\beta}{2}
\bigl\|
(1+\beta\smash{\widetilde{\mathtt{d}}_h})^{-\frac{1}{2}}
\bigl(
u_h-u
-
\tfrac{\beta(u-u_\infty)}
{1+\beta\smash{\widetilde{\mathtt{d}}}}
(\smash{\widetilde{\mathtt{d}}_h}-\smash{\widetilde{\mathtt{d}}})\bigr)\bigr\|_{\Gamma_I}^2
\\
&\quad+
\tfrac{\beta^2}{2}
(
(1+\beta\smash{\widetilde{\mathtt{d}}})^{-2}
(u-u_\infty)^2,
\smash{\widetilde{\mathtt{d}}}-\smash{\widetilde{\mathtt{d}}_h})_{\Gamma_I}\lesssim h^{2s} \,. 
\end{aligned}
    \end{align} 
    If, in addition, the non-degeneracy condition \begin{align}\label{cor:quadratic_growth_new.-1}
    \delta_{\star}
    \coloneqq
    \operatorname{dist}_{L^2(\Gamma_I)}\bigl(1, \tfrac{\beta (u-u_\infty)}{1+\beta\smash{\widetilde{\mathtt{d}}}}\smash{\widetilde{L}}^2(\Gamma_I)\bigr)>0\,,
    \end{align} 
    where $\smash{\widetilde{L}}^2(\Gamma_I)\coloneqq \{\eta\in L^2(\Gamma_I)\mid (\eta,1)_{\Gamma_I}=0\}$
    is satisfied, there holds 
    \begin{align}\label{thm:H1_error.0.2}
       \|u - u_h\|_{H^1(\Omega)}^2\lesssim h^{2s} \,. 
    \end{align} 
\end{theorem} 

\begin{remark}[Sufficient conditions for the non-degeneracy condition \eqref{cor:quadratic_growth_new.-1}]
\label{rem:nondegeneracy_conditions}
The non-degener\-acy condition \eqref{cor:quadratic_growth_new.-1}
excludes the possibility that constant modes of $v-u$ can be compensated
by admissible variations of $\smash{\widetilde{\mathtt{d}}}$. It is satisfied, for instance,
under the following assumptions:

\begin{itemize}[noitemsep,topsep=2pt,leftmargin=!,labelwidth=\widthof{(iii)}]

\item[(i)]\hypertarget{rem:nondegeneracy_conditions.i}{} If $\vert \{u=u_\infty\}\cap\Gamma_I\vert>0$, then $\frac{\beta(u-u_\infty)}{1+\beta\smash{\widetilde{\mathtt{d}}}}
    \xi
    =
    0$ 
on a set of positive measure for all 
$\xi\in \smash{\widetilde{L}}^2(\Gamma_I)$, so that 
\begin{align}\label{rem:nondegeneracy_conditions.1}
    \smash{1
    \notin
    \operatorname{cl}_{\smash{L^2(\Gamma_I)}}\bigl(
        \tfrac{\beta(u-u_\infty)}
        {1+\beta\smash{\widetilde{\mathtt{d}}}}
        \smash{\widetilde{L}}^2(\Gamma_I)
   \bigr)}\,,
\end{align}
and, therefore, $\delta_{\star}>0$;

\item[(ii)]\hypertarget{rem:nondegeneracy_conditions.ii}{} If $u-u_\infty\neq0 $ a.e.\ on $\Gamma_I$ and $\frac{1}{u-u_\infty}
    \in
    L^2(\Gamma_I)$,  
then $ \tfrac{\beta(u-u_\infty)}
    {1+\beta\smash{\widetilde{\mathtt{d}}}}
    \smash{\widetilde{L}}^2(\Gamma_I)$ is a 
closed hyperplane in $L^2(\Gamma_I)$ with normal vector $\frac{1+\beta\smash{\widetilde{\mathtt{d}}}}
    {\beta(u-u_\infty)}\in L^2(\Gamma_I)$. If, in addition, we have that $(\frac{1+\beta\smash{\widetilde{\mathtt{d}}}}
    {u-u_\infty},1)_{\Gamma_I}
    \neq0$, then \eqref{rem:nondegeneracy_conditions.1} applies and, therefore, $\delta_{\star}>0$;
    
    \hspace{5mm}A sufficient condition for
(\hyperlink{rem:nondegeneracy_conditions.ii}{ii}) is that $u-u_\infty\hspace{-0.15em}\ge\hspace{-0.15em} c_0$ or $u-u_\infty\hspace{-0.15em}\le\hspace{-0.15em} -c_0$ a.e.\ on $\Gamma_I$~for~some~${c_0\hspace{-0.15em}>\hspace{-0.15em}0}$; 

\item[(iii)]\hypertarget{rem:nondegeneracy_conditions.iii}{}
If $u_\infty=\mathrm{const}$ a.e.\ on $\Gamma_I$, $f\ge0 $ a.e.\ in $\Omega$, and  $g\ge0 $ a.e.\ on $\Gamma_N$, then the weak maximum principle
implies $u-u_\infty\ge0$ a.e.\ in $\Omega$, which follows by testing the optimality condition
\eqref{prop:optimality.3} with $v=-(u-u_\infty)_-\in H^1(\Omega)$ resulting in 
\begin{align}
    \kappa\|\nabla (u-u_\infty)_-\|_\Omega^2
    +
    \beta
    \bigl\|
        (1+\beta\smash{\widetilde{\mathtt{d}}})^{-\smash{\frac12}}
        (u-u_\infty)_-
    \bigr\|_{\Gamma_I}^2
    \le0 \,.
\end{align}
Consequently, if, in addition, $ u-u_\infty\hspace{-0.1em}>\hspace{-0.1em}0$
    a.e.\ on $\Gamma_I$ and $\frac{1}{u-u_\infty}\hspace{-0.1em}\in\hspace{-0.1em} L^2(\Gamma_I)$,  
then ${(\frac{1+\beta\smash{\widetilde{\mathtt{d}}}}
    {u-u_\infty},1)_{\Gamma_I}
    \hspace{-0.1em}\neq\hspace{-0.1em}0}$, so that (\hyperlink{rem:nondegeneracy_conditions.ii}{ii}) applies and, therefore, $\delta_\star>0$.

    \hspace{5mm}If, in addition,  $\Gamma_I=\partial\Omega$ and the corresponding Robin problem is regular enough for the strong maximum principle and Hopf's boundary point lemma (\textit{cf.} \cite[Lemma 3.4]{GilbargTrudinger2001Elliptic}) to apply up to $\Gamma_I$, then   $u-u_\infty\ge c_0$ a.e.\ on $\Gamma_I$ for some $c_0>0$.   Indeed, for $w\coloneqq u-u_\infty\in H^1(\Omega)$, the Robin condition is homogeneous, and the claim follows from the strong maximum principle, Hopf's boundary point lemma, and compactness of $\Gamma_I$.
\end{itemize}
\end{remark}\pagebreak

The proof of Theorem \ref{thm:H1_error} follows a standard procedure: 

In the first step, we establish a quadratic growth estimate, which bounds the natural error of the discretization in terms of an energy difference.

\begin{lemma} \label{lem:quadratic_growth_new}
    For every $(v,\eta)\in H^1(\Omega)\times\mathcal{H}_m$, there holds 
    \begin{align}\label{lem:quadratic_growth_new.0}
    \begin{aligned} 
      \tfrac{\kappa}{2}
\|\nabla(v-u)\|_\Omega^2
&+
\tfrac{\beta}{2}
\|
(1+\beta\eta)^{-\frac{1}{2}}
(
v-u
-
\tfrac{\beta(u-u_\infty)}
{1+\beta\smash{\widetilde{\mathtt{d}}}}
(\eta-\smash{\widetilde{\mathtt{d}}}))\|_{\Gamma_I}^2
\\
&\quad+
\tfrac{\beta^2}{2}
(
(1+\beta\smash{\widetilde{\mathtt{d}}})^{-2}
(u-u_\infty)^2,
\smash{\widetilde{\mathtt{d}}}-\eta)_{\Gamma_I}=E(v,\eta) - E(u,\smash{\widetilde{\mathtt{d}}})\,,
\end{aligned}
    \end{align} 
    where $\tfrac{\beta^2}{2}
(
(1+\beta\smash{\widetilde{\mathtt{d}}})^{-2}
(u-u_\infty)^2,
\smash{\widetilde{\mathtt{d}}}-\eta)_{\Gamma_I}\ge 0$ due to \eqref{prop:optimality.3b}.
\end{lemma}

\begin{proof}
Let $(v,\eta)\in H^1(\Omega)\times\mathcal{H}_m$ be fixed, but arbitrary. Then, by the binomial formula and the optimality condition \eqref{prop:optimality.3}, we find that
\begin{align}\label{lem:quadratic_growth_new.1}
    \begin{aligned}
E(v,\eta)-E(u,\smash{\widetilde{\mathtt{d}}})
&=
\tfrac{\kappa}{2}\|\nabla(v-u)\|_\Omega^2
+
\kappa(\nabla u,\nabla(v-u))_\Omega
-(f,v-u)_\Omega
-\langle g,v-u\rangle_{\Gamma_N}
\\
&\quad+
\tfrac{\beta}{2}
\smash{\bigl(
\|
(1+\beta\eta)^{-\smash{\frac{1}{2}}}
(v-u_\infty)
\|_{\Gamma_I}^2
-
\|
(1+\beta\smash{\widetilde{\mathtt{d}}})^{-\smash{\frac{1}{2}}}
(u-u_\infty)
\|_{\Gamma_I}^2
\bigr)}
\\&=
\tfrac{\kappa}{2}\|\nabla(v-u)\|_\Omega^2
-\beta
(
(1+\beta\smash{\widetilde{\mathtt{d}}})^{-1}
(u-u_\infty),
v-u
)_{\Gamma_I}
\\
&\quad+
\tfrac{\beta}{2}
\smash{\bigl(
\|
(1+\beta\eta)^{-\smash{\frac{1}{2}}}
(v-u_\infty)
\|_{\Gamma_I}^2
-
\|
(1+\beta\smash{\widetilde{\mathtt{d}}})^{-\smash{\frac{1}{2}}}
(u-u_\infty)
\|_{\Gamma_I}^2
\bigr)}\,.
\end{aligned}
\end{align} 
Since the boundary integrands may be rewritten as 
\begin{align*}
    \left.\begin{aligned}
&(1+\beta\eta)^{-1}
(v-u_\infty)^2
-
(1+\beta\smash{\widetilde{\mathtt{d}}})^{-1}
(u-u_\infty)^2
-
2(1+\beta\smash{\widetilde{\mathtt{d}}})^{-1}
(u-u_\infty)(v-u)
\\
&=
(1+\beta\eta)^{-1}
\bigl(
v-u
-
\tfrac{\beta(u-u_\infty)}
{1+\beta\smash{\widetilde{\mathtt{d}}}}
(\eta-\smash{\widetilde{\mathtt{d}}})
\bigr)^2
+
\beta
(1+\beta\smash{\widetilde{\mathtt{d}}})^{-2}
(u-u_\infty)^2
(\smash{\widetilde{\mathtt{d}}}-\eta)
\end{aligned}\quad\right\}\quad\text{ a.e.\ on }\Gamma_I\,,
\end{align*}  
from \eqref{lem:quadratic_growth_new.1}, we conclude the claimed quadratic growth identity \eqref{lem:quadratic_growth_new.0}. 
\end{proof} 

\begin{corollary}
    \label{cor:quadratic_growth_new}
Assume that $u,u_\infty,\smash{\widetilde{\mathtt{d}}}\in L^\infty(\Gamma_I)$ and the non-degeneracy condition \eqref{cor:quadratic_growth_new.-1}. 
Then, for every $(v,\eta)\in H^1(\Omega)\times H_m$ with
$\eta\in L^\infty(\Gamma_I)$, there holds
\begin{align}\label{cor:quadratic_growth_new.0}
    \|v-u\|_{H^1(\Omega)}^2
    \lesssim E(v,\eta)-E(u,\smash{\widetilde{\mathtt{d}}})\,,
\end{align} 
where the implicit constant in $\lesssim$ depends only on
$\Omega$, $\Gamma_I$, $\kappa$, $\beta$,
$\|\smash{\widetilde{\mathtt{d}}}\|_{\infty,\Gamma_I}$,
$\|\eta\|_{\infty,\Gamma_I}$, and $\delta_{\star}^{-1}$.
\end{corollary} 

\begin{proof}[Proof (of Corollary \ref{cor:quadratic_growth_new}).]
Let $(v,\eta)\in H^1(\Omega)\times H_m$ with
$\eta\in L^\infty(\Gamma_I)$ be fixed, but arbitrary. Then, due to Lemma \ref{lem:quadratic_growth_new} and $(1+\beta\eta)^{-1}\ge (1+\beta\|\eta\|_{\infty,\Gamma_I})^{-1}$ a.e.\ on $\Gamma_I$, we have that 
\begin{align}\label{cor:quadratic_growth_new.1} 
\begin{aligned}
    E(v,\eta)-E(u,\smash{\widetilde{\mathtt{d}}})
    &\gtrsim 
   \|\nabla(v-u)\|_\Omega^2
        +
        \smash{\bigl\|
            v-u
            -
            \tfrac{\beta(u-u_\infty)}{1+\beta\smash{\widetilde{\mathtt{d}}}}
            (\eta-\smash{\widetilde{\mathtt{d}}})
        \bigr\|_{\Gamma_I}^2}\,.
\end{aligned}
\end{align}
Next, let $\langle v-u\rangle_{\Gamma_I}\coloneqq \frac{1}{|\Gamma_I|}
    \int_{\Gamma_I}(v-u)\,\mathrm{d}s$.  
Then, by the boundary Poincar\'e inequality, there holds
\begin{align}\label{cor:quadratic_growth_new.2} 
    \|v-u-\langle v-u\rangle_{\Gamma_I}\|_{\Gamma_I}
    \lesssim \|\nabla(v-u)\|_\Omega\,.
\end{align}
Since $\eta-\smash{\widetilde{\mathtt{d}}}\in \smash{\widetilde{L}}^2(\Gamma_I)$, we have that $\frac{\beta(u-u_\infty)}{1+\beta\smash{\widetilde{\mathtt{d}}}}(\eta-\smash{\widetilde{\mathtt{d}}})\in\frac{\beta(u-u_\infty)}{1+\beta\smash{\widetilde{\mathtt{d}}}}\smash{\widetilde{L}}^2(\Gamma_I)$, so that, by the non-degeneracy condition \eqref{cor:quadratic_growth_new.-1}, we observe that
\begin{align} \label{cor:quadratic_growth_new.3} 
\begin{aligned}
    \vert\langle v-u\rangle_{\Gamma_I}\vert\delta_{\star}
    &=
    \operatorname{dist}_{L^2(\Gamma_I)}
    \bigl(
        \langle v-u\rangle_{\Gamma_I},
        \tfrac{\beta(u-u_\infty)}{1+\beta\smash{\widetilde{\mathtt{d}}}}\smash{\widetilde{L}}^2(\Gamma_I)
    \bigr)
    \\
    &\le
    \smash{\bigl\|
        \langle v-u\rangle_{\Gamma_I}
        -
        \tfrac{\beta(u-u_\infty)}{1+\beta\smash{\widetilde{\mathtt{d}}}}
        (\eta-\smash{\widetilde{\mathtt{d}}})
    \bigr\|_{\Gamma_I}}\,.
\end{aligned}
\end{align}
Moreover, using \eqref{cor:quadratic_growth_new.2}, from  \eqref{cor:quadratic_growth_new.3}, we infer that 
\begin{align*} 
\begin{aligned}
    \|v-u\|_{\Gamma_I}&\lesssim  
    \|v-u-\langle v-u\rangle_{\Gamma_I}\|_{\Gamma_I}+ \vert\langle v-u\rangle_{\Gamma_I}\vert
    \\&\lesssim  \smash{\|\nabla(v-u)\|_\Omega+\bigl\|
            v-u
            -
            \tfrac{\beta(u-u_\infty)}{1+\beta\smash{\widetilde{\mathtt{d}}}}
            (\eta-\smash{\widetilde{\mathtt{d}}})
        \bigr\|_{\Gamma_I}} \,,
\end{aligned}
\end{align*}  
which, by Friedrich's inequality \eqref{lem:poin_cont}, yields the claimed quadratic growth estimate \eqref{cor:quadratic_growth_new.0}.
\end{proof}

In the second step, we estimate the energy difference. To this end, in order to be in the position to exploit the minimality of $(u_h,\smash{\widetilde{\mathtt{d}}_h})\in \mathcal{S}^1(\mathcal{T}_h)\times \mathcal{H}_m^h $ for the discrete energy functional \eqref{def:Eh}, we first estimate the difference between the minimal discrete and the minimal continuous~energy.\enlargethispage{1.5mm}

\begin{lemma}\label{lem:energy_difference}
Under the assumptions of Theorem \ref{thm:H1_error}, there holds
\begin{align}\label{lem:energy_difference.0}
    E_h(u_h,\smash{\widetilde{\mathtt{d}}_h}) - E(u,\smash{\widetilde{\mathtt{d}}}) \lesssim \smash{h^{2s} }\,.
\end{align}
\end{lemma}
\begin{proof} 
    Let $\mathcal{J}_h\colon E_{\Delta}(\Omega)\coloneqq{\{ v \in  H^{\frac{3}{2}}(\Omega) \mid \Delta v \in L^2(\Omega)\}}\to\mathcal{S}^1(\mathcal{T}_h)$ be the quasi-interpolation operator defined in \eqref{eq:quasi_interpolant_def} and $\Pi_h^{\mathrm{m}}\colon L^1(\Gamma_I)\to\mathcal{S}^1(\mathcal{S}_h^{I})$  the mass-lumped $L^2$-projection~defined~in~\eqref{eq:lumped_l2_projection}. Then, define\enlargethispage{1.5mm}
    \begin{align*}
        \smash{(\smash{u_h^{\#}},\smash{\widetilde{\mathtt{d}}_h^{\#}})\coloneqq (\mathcal{J}_h u,\Pi_h^{\mathrm{m}} \smash{\widetilde{\mathtt{d}}})\in \mathcal{S}^1(\mathcal{T}_h) \times \mathcal{H}_m^h\,,\quad
        \Psi_h \coloneqq  (1 + \beta \smash{\widetilde{\mathtt{d}}_h^{\#}})^{-1}(\smash{u_h^{\#}} - u_\infty^h)^2\in W^{1,\infty}(\Gamma_I)\,.}
    \end{align*}
    To begin with, we bound the side-wise tangential Hessian $\mathrm{D}_{\Gamma,h}^2\Psi_h\hspace{-0.1em}\in \hspace{-0.1em}(L^\infty(\Gamma_I))^{d\times d}$,~for~every~${S\hspace{-0.1em}\in\hspace{-0.1em} \mathcal{S}_h^{\partial}}$, defined by $\mathrm{D}_{\Gamma,h}^2\Psi_h\!\!\restriction_S\coloneqq \mathrm{D}^2_S(\Psi_h\!\!\restriction_S)$, where $\mathrm{D}^2_{S} \coloneqq \mathrm{P}_S \mathrm{D}^2 \mathrm{P}_S$ denotes the tangential~Hessian~on~$S$: 
\begin{align}\label{eq:quotient_estimate.1}
    \begin{aligned} 
    \mathrm{D}_{\Gamma,h}^2\Psi_h &= \tfrac{2}{(1+\beta \widetilde{\mathtt{d}}_h^{\#})^3}\smash{\bigl(}(1+\beta \smash{\widetilde{\mathtt{d}}_h^{\#}})^2\nabla_{\Gamma}(\smash{u_h^{\#}} - u_\infty^h) \otimes \nabla_{\Gamma}(\smash{u_h^{\#}} - u_\infty^h)
    \\&\quad- (1 + \beta \smash{\widetilde{\mathtt{d}}_h^{\#}})\beta(\smash{u_h^{\#}} - u_\infty^h)(\nabla_{\Gamma} \smash{\widetilde{\mathtt{d}}_h^{\#}} \otimes \nabla_{\Gamma}(\smash{u_h^{\#}} - u_\infty^h)+\nabla_{\Gamma}(\smash{u_h^{\#}} - u_\infty^h) \otimes \nabla_{\Gamma} \smash{\widetilde{\mathtt{d}}_h^{\#}})
    \\&\quad+ \beta^2(\smash{u_h^{\#}} - u_\infty^h)^2\nabla_{\Gamma}\smash{\widetilde{\mathtt{d}}_h^{\#}} \otimes \nabla_{\Gamma}\smash{\widetilde{\mathtt{d}}_h^{\#}}\smash{\bigr)}\quad\text{ a.e.\ on }\Gamma_I\,,
    \end{aligned}
\end{align}
where we used that $\mathrm{D}_{\Gamma,h}^2\smash{u_h^{\#}} =\mathrm{D}_{\Gamma,h}^2 u_\infty^h =\mathrm{D}_{\Gamma,h}^2 \smash{\widetilde{\mathtt{d}}_h^{\#}}= 0$ a.e.~on~$\Gamma_I$. 
Using Lemma \ref{lem:Linf}, the $L^\infty(\Gamma_I)$-stability of $\mathcal{J}_h$ (\textit{cf}.\ \cite[Cor.\ 4.8.15]{brenner2008mathematical}), Young's inequality, \eqref{eq:quasi_interpolant_volume_error_S}, and \eqref{eq:lumped_fractional_bound},~from~\eqref{eq:quotient_estimate.1},~we~\mbox{infer}~that 
%
%
\begin{align} \label{eq:quotient_estimate.3}
    \smash{\|\mathrm{D}_{\Gamma,h}^2\Psi_h \|_{1,\Gamma_I} \lesssim \|\nabla_{\Gamma}(\smash{u_h^{\#}} - u_\infty^h)\|_{\Gamma_I}^2 + \|\nabla_{\Gamma}\smash{\widetilde{\mathtt{d}}_h^{\#}}\|_{\Gamma_I}^2 \lesssim \|\nabla_{\Gamma}(u -  u_\infty)\|_{\Gamma_I}^2\,.}
\end{align}
Since  $(u_h,\smash{\widetilde{\mathtt{d}}_h}) \in \mathcal{S}^1(\mathcal{T}_h) \times \mathcal{H}_m^h$ is minimal for the discrete energy functional \eqref{def:Eh}, we may estimate 
    \begin{align}\label{eq:quotient_estimate.3.0}
         \begin{aligned} 
        E_h(u_h,\smash{\widetilde{\mathtt{d}}_h}) - E(u,\smash{\widetilde{\mathtt{d}}})  
       &\leq \left\{\begin{aligned}
           &(E_h(\smash{u_h^{\#}},\smash{\widetilde{\mathtt{d}}_h^{\#}} ) - E(\smash{u_h^{\#}},\smash{\widetilde{\mathtt{d}}_h^{\#}}))\\& + (E(\smash{u_h^{\#}},\smash{\widetilde{\mathtt{d}}_h^{\#}}) - E(u,\smash{\widetilde{\mathtt{d}}_h^{\#}}))
            \\& + (E(u,\smash{\widetilde{\mathtt{d}}_h^{\#}}) - E(u,\smash{\widetilde{\mathtt{d}}}))
       \end{aligned}\right\} \eqqcolon T_1^h + T_2^h + T_3^h\,.
        \end{aligned}
    \end{align}
    Thus, it remains to estimate the terms $T_i^h$, $i=1,2,3$:
    
    \textit{ad $T_1^h$.} Using the local interpolation estimate for~$I_h$~on~$\Gamma_I$ (\textit{cf}.\ \cite[ineq.\ (11.17)]{EG21I}) and \eqref{eq:quotient_estimate.3}, we find that\vspace{-0.5mm}
    \begin{align}\label{eq:quotient_estimate.4}
        \smash{T_1^h  \lesssim \|I_h\{\Psi_h\} -\Psi_h\|_{1,\Gamma_I} 
         \lesssim   h^2 \| \mathrm{D}_{\Gamma,h}^2 \Psi_h\|_{1,\Gamma_I}
     \lesssim h^2 \|\nabla_{\Gamma}(u - u_\infty)\|_{\Gamma_I}^2\,.}
    \end{align} 

    \textit{ad $T_2^h$.} Using a binomial formula, the optimality condition \eqref{prop:optimality.3}, Lemma \ref{lem:Linf}, Lemma \ref{lem:Linf_discrete}, and Young's inequality, we find that
    \begin{align}\label{eq:quotient_estimate.5}
        \begin{aligned} 
       T_2^h 
       &= \tfrac{1}{2}\|\nabla (u - \smash{u_h^{\#}})\|_{\Omega}^2 + (\nabla u, \nabla(\smash{u_h^{\#}} - u))_\Omega - (f,\smash{u_h^{\#}} -u )_{\Omega} -\langle g,\smash{u_h^{\#}} - u \rangle_{\Gamma_N} 
       \\ & \quad +\tfrac{\beta}{2} \|(1+\beta \smash{\widetilde{\mathtt{d}}_h^{\#}})^{-\smash{\frac{1}{2}}}(\smash{u_h^{\#}}-u_{\infty})\|_{\Gamma_{I}}^2  - \tfrac{\beta}{2} \|(1+\beta \smash{\widetilde{\mathtt{d}}_h^{\#}})^{-\smash{\frac{1}{2}}}(u-u_{\infty})\|_{\Gamma_{I}}^2 \\
       &= \tfrac{1}{2}\|\nabla (u - \smash{u_h^{\#}})\|_{\Omega}^2 
+ \tfrac{\beta}{2} \|(1+\beta \smash{\widetilde{\mathtt{d}}_h^{\#}})^{-\smash{\frac{1}{2}}}(\smash{u_h^{\#}} - u)\|_{\Gamma_{I}}^2 
\\&\quad+ \beta^2 
((1+\beta \smash{\widetilde{\mathtt{d}}_h^{\#}})^{-1}(1+\beta \smash{\widetilde{\mathtt{d}}})^{-1}(u-u_\infty)(\smash{u_h^{\#}}-u),\smash{\widetilde{\mathtt{d}}} - \smash{\widetilde{\mathtt{d}}_h^{\#}})_{\Gamma_I}\\&\lesssim \|\nabla( u - \smash{u_h^{\#}}) \|_{\Omega}^2 + \|u - \smash{u_h^{\#}}\|_{2,\Gamma_I}^2 + \|u - \smash{u_h^{\#}}\|_{2,\Gamma_I}\|\smash{\widetilde{\mathtt{d}}} - \smash{\widetilde{\mathtt{d}}_h^{\#}}\|_{2,\Gamma_I} \\
    &\lesssim h^{2s} |u|_{1+s,\Omega}^2 + h^2 \|\nabla_{\Gamma}u\|_{\Gamma_I}^2 + h^2 \|\nabla_{\Gamma}u_\infty\|_{\Gamma_I}^2\,.
    \end{aligned}
\end{align}

    \textit{ad $T_3^h$.} Using the differentiability and convexity of $(a\mapsto (1+\beta a)^{-1})\colon\mathbb{R}_{\ge 0}\to \mathbb{R}_{\ge 0}$, that $\{(1+\beta \smash{\widetilde{\mathtt{d}}_h^{\#}})^{-2}(u - u_\infty)^2\}_{h>0} \subseteq H^1(\Gamma_I)$ is bounded (which follows from Lemma \ref{lem:Linf}, Lemma \ref{lem:Linf_discrete}, and \ref{eq:lumped_fractional_bound}), and \eqref{eq:lumped_negative_norm}, we find that\vspace{-0.5mm}
\begin{align}\label{eq:quotient_estimate.6}
    \begin{aligned} 
    T_3^h &= \tfrac{\beta}{2} \|(1+\beta\smash{\widetilde{\mathtt{d}}_h^{\#}})^{-\frac{1}{2}}(u-u_\infty)\|_{\Gamma_I}^2-\tfrac{\beta}{2} \|(1+\beta\smash{\widetilde{\mathtt{d}}})^{-\frac{1}{2}}(u-u_\infty)\|_{\Gamma_I}^2\\
    &\le \tfrac{\beta}{2}  \smash{((1+\beta \smash{\widetilde{\mathtt{d}}_h^{\#}})^{-2}(u - u_\infty)^2,\smash{\widetilde{\mathtt{d}}} - \smash{\widetilde{\mathtt{d}}_h^{\#}})_{\Gamma_I}} \\
    &\le \tfrac{\beta}{2} \|(1+\beta \smash{\widetilde{\mathtt{d}}_h^{\#}})^{-2}(u - u_\infty)^2\|_{H^1(\Gamma_I)}\|\smash{\widetilde{\mathtt{d}}} - \smash{\widetilde{\mathtt{d}}_h^{\#}}\|_{H^{-1}(\Gamma_I)} \\
    & \lesssim h^{2} \|\nabla_{\Gamma}\smash{\widetilde{\mathtt{d}}}\|_{\Gamma_I}\,.
    \end{aligned}
\end{align}
Eventually, using \eqref{eq:quotient_estimate.4}--\eqref{eq:quotient_estimate.6} in \eqref{eq:quotient_estimate.3.0}, we conclude that \eqref{lem:energy_difference.0} applies.
\end{proof}

Third, we estimate the consistency error.

\begin{lemma} \label{prop:energy_consistent_vs_lumped}
    For every $(v_h,\eta_h) \in \mathcal{S}^1(\mathcal{T}_h) \times \mathcal{H}_m^h$, there holds
    \begin{align} \label{eq:energy_consistent_vs_lumped_inconsistentdata.1}
        E(v_h, \eta_h) &\leq E_h(v_h,\eta_h)+\beta\|v_h-u_\infty^h\|_{\Gamma_I,h}
\|u_\infty-u_\infty^h\|_{\Gamma_I}
+
\tfrac{\beta}{2}\|u_\infty-u_\infty^h\|_{\Gamma_I}^2 \,.
    \end{align}
\end{lemma}
\begin{proof} 
Let $(v_h,\eta_h) \hspace{-0.15em}\in \hspace{-0.15em}\mathcal{S}^1(\mathcal{T}_h)\times \mathcal{H}_m^h$ be fixed, but arbitrary. The claimed estimate~\eqref{eq:energy_consistent_vs_lumped_inconsistentdata.1}~\mbox{follows}~from $(v_h - u_\infty)^2= (v_h - u_\infty^h)^2+2(v_h - u_\infty^h)(u_\infty - u_\infty^h)+(u_\infty - u_\infty^h)^2$ on $\Gamma_I$
and
\begin{align} \label{eq:robin_bnd_by_Ih}
    (1 + \beta \eta_h)^{-1}(v_h - u_\infty^h)^2 \le I_h\{(1 + \beta \eta_h)^{-1}(v_h - u_\infty^h)^2\}\quad\text{ on }\Gamma_I\,,
\end{align}
which is a consequence of convexity of the mapping $(s,a)\mapsto (1+\beta s)^{-1}a^2\colon \mathbb{R}_{\ge 0}\times \mathbb{R}\to \mathbb{R}$.
\end{proof}

\begin{proof}[Proof (of Theorem \ref{thm:H1_error}).]
    We combine Lemma \ref{lem:quadratic_growth_new}, Corollary \ref{cor:quadratic_growth_new}, Lemma \ref{lem:energy_difference}, and Lemma  \ref{prop:energy_consistent_vs_lumped}.
\end{proof}

\subsection{Error bounds for the distribution function of the insulation material and critical temperature difference}\vspace{-0.5mm}

\hspace{5mm}In this subsection, we derive \textit{a priori} error estimates for both the distribution function of the insulation material --also outside the region of non-trivial temperature excess $\operatorname{supp}(u-u_\infty)$-- and critical temperature difference.  In doing so, we assume that $m\beta > |\Gamma_I|$, 
meaning that the available insulation material is sufficiently large relative to the size of the insulated boundary and the heat transfer across it. Under this assumption, the \textit{a priori} error estimate derived in Theorem \ref{thm:H1_error} carries over by means of the following lemma.\vspace{-0.5mm}\enlargethispage{2.5mm}

\begin{lemma}\label{lem:C_u_d_u_estimate}
If $m\beta > |\Gamma_I|$, then
\begin{subequations}
\begin{align}\label{lem:C_u_estimate.0}
    |\mathtt C_u - \mathtt{C}_{\smash{u_h}}|&\leq \tfrac{1}{m\beta - |\Gamma_I|} ( \|u-u_h\|_{1,\Gamma_I} + \|u_\infty-u_\infty^h\|_{1,\Gamma_I})\,;\\
    \|\smash{\widetilde{\mathtt{d}}} - \smash{\widetilde{\mathtt{d}}_h}\|_{1,\Gamma_I} &\lesssim \|u-u_h\|_{1,\Gamma_I} + \|u_\infty-u_\infty^h\|_{1,\Gamma_I} + h^{\frac{1}{2}}\,.\label{lem:d_u_estimate.0}
\end{align}
\end{subequations}
\end{lemma}  

\begin{proof}

\textit{ad \eqref{lem:C_u_estimate.0}.}
Since  $\Phi\coloneqq((a,c)\mapsto(|a|-c)_+)\in C^{0,1}(\mathbb{R}\times \mathbb{R})$ and $\textup{Lip}(\Phi)\leq 1$, we find that 
\begin{align*}
     |\mathtt{C}_u - \mathtt{C}_{\smash{u_h}}|  \le \tfrac{1}{m\beta}(\|u-u_h\|_{1,\Gamma_I} +  \|u_\infty-u_\infty^h\|_{1,\Gamma_I} + \vert \Gamma_I\vert |\mathtt{C}_{\smash{u_h}} - \mathtt C_u|)\,,
\end{align*}
and, thus, due to $m\beta > |\Gamma_{I}|$, that the claimed estimate \eqref{lem:C_u_estimate.0} applies. 

\textit{ad \eqref{lem:d_u_estimate.0}.}
Using \eqref{prop:optimality.2} and \eqref{prop:discrete_optimality.2}, we find that 
\begin{align}\label{lem:d_u_estimate.1}
\begin{aligned}
    \smash{\widetilde{\mathtt{d}}}  - \smash{\widetilde{\mathtt{d}}_h} &= \tfrac{1}{\beta \mathtt C_{\smash{u}}}(\vert u-\smash{u}_{\infty}\vert-\mathtt C_{\smash{u}})_+ - \tfrac{1}{\beta \mathtt{C}_{\smash{u_h}}}I_h\{(\vert u_h-u_{\infty}^h\vert-\mathtt{C}_{\smash{u_h}})_+\}\\
&= \smash{\bigl(\tfrac{1}{\beta \mathtt C_{\smash{u}}} - \tfrac{1}{\beta \mathtt{C}_{\smash{u_h}}}\bigr)}(\vert u-\smash{u}_{\infty}\vert-\mathtt C_{\smash{u}})_+\\&\quad + \tfrac{1}{\beta \mathtt{C}_{\smash{u_h}}}((\vert u-\smash{u}_{\infty}\vert-\mathtt C_{\smash{u}})_+ -  (\vert u_h-u_{\infty}^h\vert-\mathtt{C}_{\smash{u_h}})_+)  
    \\ & \quad +\tfrac{1}{\beta \mathtt{C}_{\smash{u_h}}} ((\vert u_h-u_{\infty}^h\vert-\mathtt{C}_{\smash{u_h}})_+ -I_h\{(\vert u_h-u_{\infty}^h\vert-\mathtt{C}_{\smash{u_h}})_+\} ) \\
    &\eqqcolon T_1^h + T_2^h + T_3^h 
    \quad \text{ a.e.\ on }\Gamma_I\,.
\end{aligned}
\end{align}
Using \eqref{prop:optimality.1} and \eqref{lem:C_u_estimate.0}, we have that
\begin{align}\label{lem:d_u_estimate.2}
\begin{aligned} 
    \|T_1^h\|_{1,\Gamma_I} &\le \smash{\tfrac{1}{\beta \mathtt{C}_u \mathtt{C}_{\smash{u_h}}}}|\mathtt{C}_u - \mathtt{C}_{\smash{u_h}}|\|(\vert u-\smash{u}_{\infty}\vert-\mathtt C_{\smash{u}})_+\|_{1,\Gamma_I}
    \\&=\smash{\tfrac{m}{\mathtt{C}_{\smash{u_h}}}}|\mathtt{C}_u - \mathtt{C}_{\smash{u_h}}|
    \\&
    \le \smash{\tfrac{m}{\mathtt{C}_{\smash{u_h}}(m\beta - |\Gamma_I|)}} ( \|u-u_h\|_{1,\Gamma_I} + \|u_\infty-u_\infty^h\|_{1,\Gamma_I})\,.
    \end{aligned}
\end{align}
Since $\Phi\coloneqq ((a,c)\mapsto (\vert a\vert-c)_+)\in C^{0,1}(\mathbb{R}\times \mathbb{R})$, using \eqref{lem:C_u_estimate.0} and Lemma \ref{lem:mass_lumping_lipschitz}, 
we have that
\begin{align}\label{lem:d_u_estimate.3}
\begin{aligned} 
 \|T_2^h\|_{1,\Gamma_I} & \le \smash{\tfrac{1}{\beta \mathtt{C}_{\smash{u_h}}}}(\| \vert u-\smash{u}_{\infty}\vert - \vert u_h-u_{\infty}^h\vert\|_{1,\Gamma_I} +\vert \Gamma_I\vert \vert\mathtt{C}_u-\mathtt{C}_{\smash{u_h}}\vert)\\&\le \smash{\tfrac{1}{\mathtt{C}_{\smash{u_h}}}}\tfrac{m}{m\beta-\vert\Gamma_I\vert}(\|u - u_h\|_{1,\Gamma_I} + \|u_\infty - u_\infty^h\|_{1,\Gamma_I})\,,\\
 \|T_3^h\|_{1,\Gamma_I}&\lesssim  h^{\frac{1}{2}} \| \nabla u_h- \nabla u_{\infty}^h\|_{ \Omega}\,.
 \end{aligned}
\end{align}
Eventually, using \eqref{lem:d_u_estimate.2} and \eqref{lem:d_u_estimate.3} in \eqref{lem:d_u_estimate.1}, we conclude that the claimed estimate~\eqref{lem:d_u_estimate.0}~applies.
\end{proof}\newpage

\section{Numerical experiments}\vspace{-0.75mm}
\label{sec:num-insulation}

\hspace{5mm}All \hspace{-0.1mm}computations \hspace{-0.1mm}use \hspace{-0.1mm}the \hspace{-0.1mm}block \hspace{-0.1mm}coordinate \hspace{-0.1mm}descent \hspace{-0.1mm}solver \hspace{-0.1mm}of \hspace{-0.1mm}Algorithm~\hspace{-0.1mm}\ref{algorithm}, \hspace{-0.1mm}implemented~\hspace{-0.1mm}in~\hspace{-0.1mm}the~\hspace{-0.1mm}fi\-nite \hspace{-0.1mm}element \hspace{-0.1mm}library \hspace{-0.1mm}\texttt{NETGEN}/\texttt{NGSolve} \hspace{-0.1mm}(version \hspace{-0.1mm}v6.2.2602, \hspace{-0.1mm}\textit{cf}.\ \hspace{-0.1mm}\cite{netgen}/\cite{ngsolve}); \hspace{-0.1mm}all \hspace{-0.1mm}graphics \hspace{-0.1mm}use \hspace{-0.1mm}\texttt{Matplotlib} (version 3.10.8, \textit{cf}.\ \cite{Hunter07}) or \texttt{PyVista} (version 0.48.4, \textit{cf}.\ \cite{pyvista}). We report two examples.~The~first~is~a manufactured-solution study on the unit square (\textit{cf}.\ \Cref{sec:num-apriori}) confirming the error decay rates of Theorem~\ref{thm:H1_error}. The second is a qualitative three-dimensional example on the Orion~crew-module geometry (\textit{cf}.\ \Cref{sec:num-orion}), with a Robin datum~idealizing~windward~biased~reentry~heating.\enlargethispage{10mm}\vspace{-1mm}

\subsection{Manufactured test case}
\label{sec:num-apriori}\vspace{-0.75mm}

\hspace{5mm}In this example, let $\Omega=(0,1)^2$, $\Gamma_I=\partial\Omega$, 
$\kappa=1$, $\beta=\tfrac32$, $m=4$, and $u_\infty\in C^\infty(\mathbb{R}^2)$, for every $x=(x_1,x_2)\in \mathbb{R}^2$ defined by $ u_\infty(x_1,x_2) \coloneqq \tfrac32(\gamma(x_1)+\gamma(x_2)) - 1
            + \tfrac1{10}(\cos 2\pi x_1+\cos 2\pi x_2)$,~where $\gamma(t)\hspace{-0.1em}\coloneqq\hspace{-0.1em} t(1-t)$ for all $t\hspace{-0.1em}\in\hspace{-0.1em} [0,1]$. 
Moreover, let $\omega\hspace{-0.1em}\coloneqq \hspace{-0.1em}m\beta\|\psi+\gamma^2\|_{1,\Gamma_I}^{-1}$~and~let~$\psi\hspace{-0.1em}\in\hspace{-0.1em} W^{2,\infty}(0,1)\cap C^1[0,1]$, for every $t\in [0,1]$, be defined by\vspace{-0.5mm}
\begin{align*}
    \psi(t)\coloneqq \begin{cases} 
        \smash{0.12(1-3\xi^2+2\xi^3)}&\text{ if }t\eqqcolon a\xi\in [0,a]\,,\\[-0.5mm]
        0&\text{ if }t\in[a,b]\,,\\[-0.5mm]
        \smash{0.12(3\zeta^2-2\zeta^3)}&\text{ if }t\eqqcolon (1-b)\zeta
        +b\in [b,1]\,,
    \end{cases}\text{ where }(a,b)=(0.3765,0.6184)\,. 
\end{align*} 
If  
$f\hspace{-0.1em}\coloneqq\hspace{-0.1em}-\Delta u\hspace{-0.1em}\in\hspace{-0.1em}  L^2(\Omega)$, a pair $(u,\smash{\widetilde{\mathtt{d}}})\hspace{-0.1em}\in\hspace{-0.1em} H^1(\Omega)\times\mathcal{H}_m$ minimizing \eqref{def:E} (solving~\mbox{\eqref{prop:optimality.1}--\eqref{prop:optimality.3}}~with~${\mathtt{C}_u\hspace{-0.1em}=\hspace{-0.1em}1}$) 
is given via $(u(x),\smash{\widetilde{\mathtt{d}}}(s))\hspace{-0.1em}=\hspace{-0.1em}(u_\infty(x)+1+ \omega(\psi(x_1)+\gamma^2(x_2)),\frac{\omega}{\beta}(\psi(s_1)+\gamma^2(s_2)))$ for all ${x\hspace{-0.1em}=\hspace{-0.1em}(x_1,x_2)\hspace{-0.1em}\in\hspace{-0.1em} \mathbb{R}^2}$\linebreak and $s =(s_1,s_2)\in \Gamma_I$, so that 
$\{|u-u_\infty|=\mathtt{C}_u\}=[a,b]\times \{0,1\}$ and  
$\smash{\widetilde{\mathtt{d}}}>0$ on  $\Gamma_I\setminus \{|u-u_\infty|=\mathtt{C}_u\}$.

Since $u\in H^2(\Omega)$, Theorem~\ref{thm:H1_error} predicts first-order convergence in the energy norm~and second-order convergence of the energy error.  The plotted $L^2$-errors in $u$ and
$\smash{\widetilde{\mathtt{d}}}$,~as~well~as~the~error~in~$\mathtt{C}_u$, also exhibit approximately second-order convergence under~\mbox{uniform}~\mbox{mesh-refinement},~see~\mbox{\Cref{fig:ins-test1-conv}}. 
\Cref{fig:ins-test1-sol} visualizes the corresponding discrete temperature and optimal insulation distribution.

\begin{figure}[H]
  \centering
  \includegraphics[width=\linewidth]{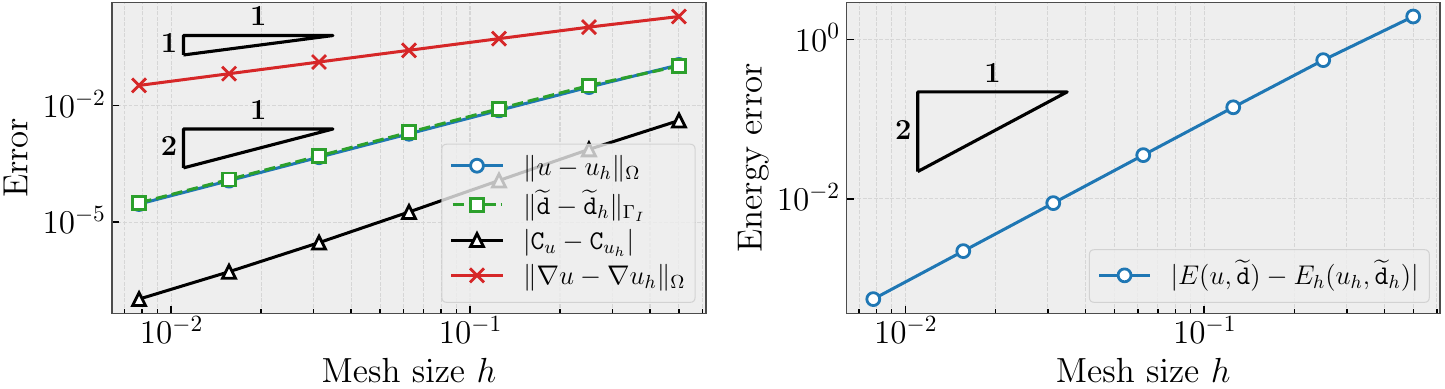}\vspace{-1mm}
  \caption{\textit{Left:} $L^2(\Omega)$- and $H^1(\Omega)$-semi-norm errors for $u$, $L^2(\Gamma_I)$-error for $\smash{\widetilde{\mathtt{d}}}$, and~\mbox{error}~in~$\mathtt C_u$; 
  \textit{right:} energy error with reference value $E(u,\smash{\widetilde{\mathtt{d}}})\approx -15.5424$.}
  \label{fig:ins-test1-conv}
\end{figure}\vspace{-5mm}

\begin{figure}[H]
  \centering
  \includegraphics[width=\linewidth]{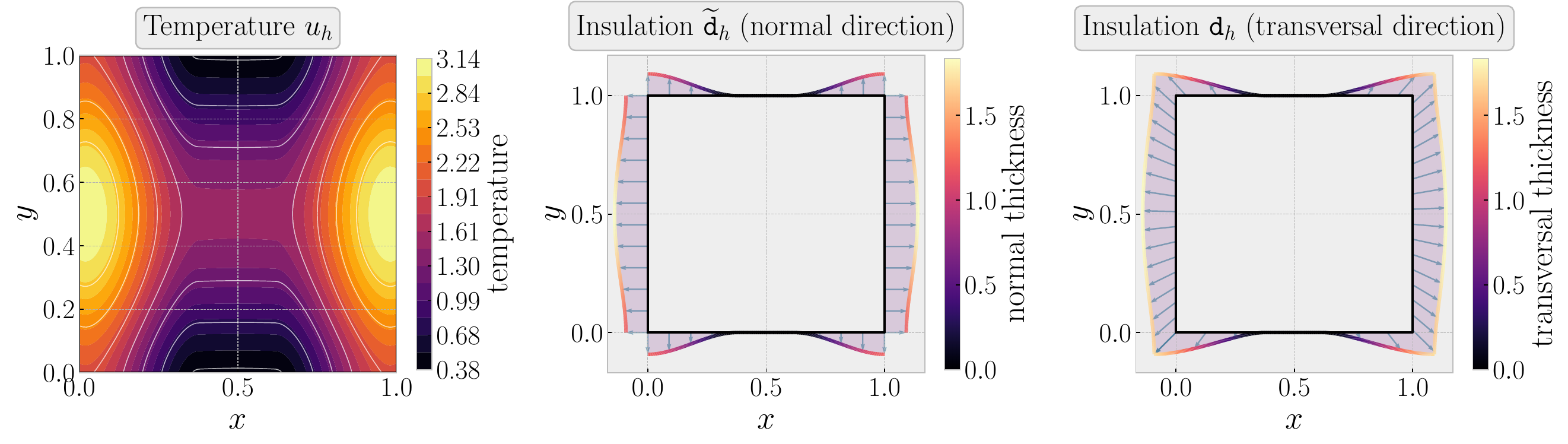}\vspace{-1.5mm}
  \caption{\textit{Left:} temperature $u_h\in \mathcal{S}^1(\mathcal{T}_h)$ computed on a
structured mesh with $65{,}536$ simplices. \textit{Center:} distribution function $\smash{\widetilde{\mathtt{d}}_h}\in \mathcal{S}^1(\mathcal{S}_h^{I})$ in the normal direction. \textit{Right:} distribution function $\smash{\widetilde{\mathtt{d}}_h}$ in the transversal direction. The support of $\smash{\widetilde{\mathtt{d}}_h}$ intersects all four sides, while the contact set has positive measure and leaves the central parts of the top and bottom edges uninsulated. 
  }
  \label{fig:ins-test1-sol}
\end{figure}\newpage
 
\subsection{A qualitative three-dimensional test inspired by reentry heating}
\label{sec:num-orion}

\hspace{5mm}In this example, we consider a qualitative three-dimensional test case on an idealized reentry capsule geometry
inspired by NASA's Orion crew module. The computational domain $\Omega$ represents a blunt
capsule with a broad heat shield, a rounded shoulder, and a tapered rear~section. The geometry
is based on the capsule dimensions reported in~\cite{HollisBergerHorvathCoblishNorrisLillardKirk2009}.
The entire exterior hull is treated as the insulated boundary (\textit{i.e.}, $\Gamma_I=\partial\Omega$).
The purpose of this example is to illustrate the behavior of the proposed finite element approximation on a
non-trivial three-dimensional geometry. It is not intended as a high-fidelity aero-thermal or
thermal-protection-system simulation.

In order to mimic the strongly windward-biased thermal loading during reentry, we prescribe a spatially
varying ambient temperature on $\Gamma_I$. Let $\widehat z\coloneqq e_3=(0,0,1)\in \mathbb{S}^2$ denote the flight axis, and assume
that the geometry is oriented such that the heat shield faces the positive~\mbox{$e_3$-direction}. Moreover,
let $x_0\in \Omega$ denote the centroid of $\Omega$ and $k\in (C^\infty(\Gamma_I))^3$~a~smooth~(globally)~transversal vector field, for every $s\in \Gamma_I$ defined by 
\begin{align*}
    k(s)\coloneqq\tfrac{s-x_0}{|s-x_0|}\,,
\end{align*}
which satisfies $k\cdot n\ge \kappa$ on $\Gamma_I$ for some $\kappa>0$, depending on the geometry of $\Omega$. The ambient temperature $u_\infty\in C^{1,1}(\Gamma_I)$, for every $x\in \mathbb{R}^3\setminus \Omega$, is defined by 
\begin{align*}
  u_\infty
  \coloneqq
  \Theta_{\mathrm{wake}}
  +
  (\Theta_{\mathrm{stag}}-\Theta_{\mathrm{wake}})
  (\widehat z\cdot k)_+^2\,,
  \qquad
  \Theta_{\mathrm{stag}}\coloneqq300\,,\quad
  \Theta_{\mathrm{wake}}\coloneqq0\,.
\end{align*}
In particular, the maximum ambient temperature is attained in the windward stagnation region, while
$u_\infty=0$ on the leeward part $\{\widehat z\cdot k\le 0\}$. The remaining data are $\kappa=40$, $\beta=1$,~and~$f\equiv\tfrac{1}{10}$.
The domain $\Omega$ is discretized by $1{,}758{,}964$ tetrahedral elements. In each step of the block
coordinate descent algorithm, the resulting linear system is solved by a direct sparse~Cholesky~factorization.

\Cref{fig:orion} displays the discrete temperature $u_h\in \mathcal{S}^1(\mathcal{T}_h)$ and the discrete optimal insulation
distribution (in transversal direction) $\mathtt{d}_h\coloneqq (k\cdot n)^{-1}\widetilde{\mathtt{d}}_h\in H^1(\Gamma_I)$ for the normalized insulation budgets
\begin{align*}
    \tfrac{m}{|\Gamma_I|}\in\{0.1,0.5,1,1.5,2,2.5\}\,.
\end{align*}

$\bullet$ For small budgets, the insulation is concentrated in a windward cap, where the temperature
mismatch $|u_h-u_\infty^h|$ is largest. As the budget increases, the support of the insulation
expands along the hull. More precisely, when $m/|\Gamma_I|$ increases from $0.1$ to $2.5$, the
threshold $\mathtt{C}_{u_h}$ decreases from $51$ to $8.4$, while the insulated area fraction increases from
$9\%$ to $71\%$. For $m/|\Gamma_I|\le 1$, the support consists of a single windward component,
with an insulated area fraction of at most $22\%$. 

$\bullet$ For larger budgets, an additional disconnected
component appears on the leeward flank.

\begin{figure}[H]
  \centering
  \includegraphics[width=\linewidth]{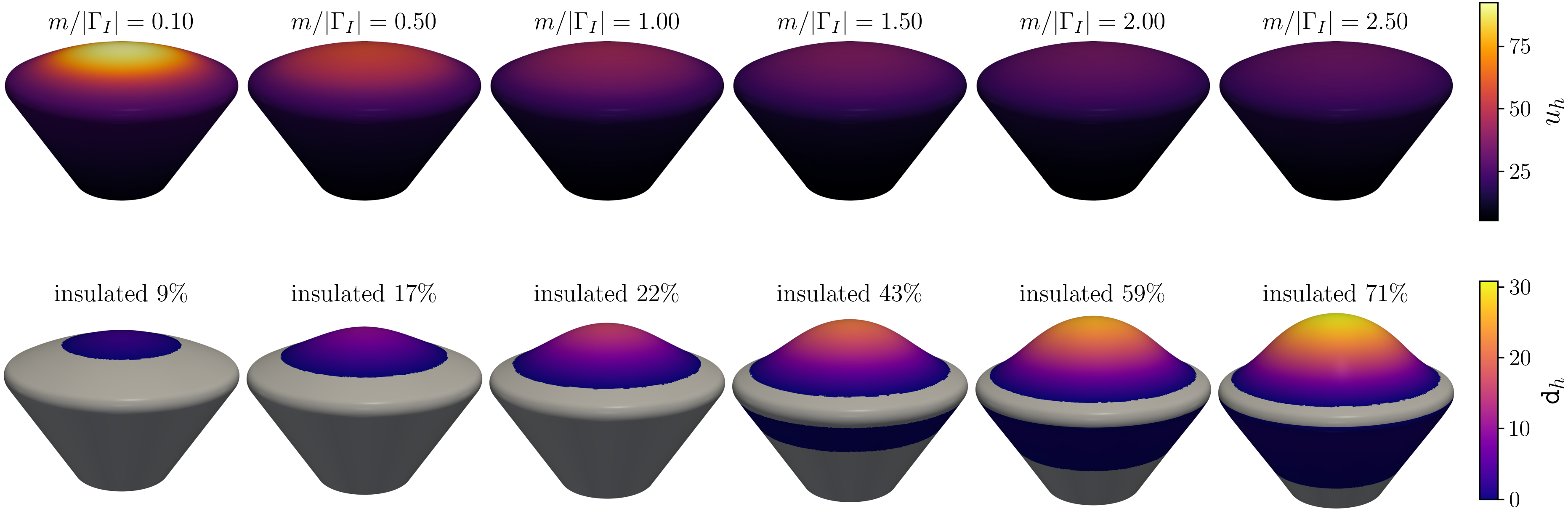}
  \caption{Qualitative three-dimensional reentry capsule test for the normalized insulation budgets
  $m/|\Gamma_I|\in\{0.1,0.5,1,1.5,2,2.5\}$.
  \textit{Top:} discrete temperature $u_h\in \mathcal{S}^1(\mathcal{T}_h)$.
  \textit{Bottom:} discrete optimal insulation distribution (in the transversal direction)
  $\mathtt{d}_h\coloneqq(k\cdot n)^{-1}\widetilde{\mathtt{d}}_h\in H^1(\Gamma_I)$.}
  \label{fig:orion}
\end{figure}\newpage

\appendix

\section{An $H^1(\Gamma_I)$-stable quasi-interpolant}\label{sec:appendix_A}\vspace{-0.5mm}
\hspace{5mm}Let $\mathcal{I}_{h} \colon H^1(\Omega) \to \mathcal{S}^1(\mathcal{T}_h)$ be a quasi-interpolant of Scott--Zhang or Cl\'ement type (see, \textit{e.g.}, \cite[Sec.\ 1.6]{ErnGuermond2004}) and $\mathcal{I}_h^I \colon H^1(\Gamma_I) \to \mathcal{S}^1(\mathcal{S}_h^{I})$ its  boundary analogue defined~on~the~\mbox{boundary}~mesh~$\mathcal{S}_h^{I}$. Then, we
define a quasi-interpolant $\mathcal{J}_h\colon E_\Delta(\Omega)\hspace{-0.1em}\coloneqq\hspace{-0.1em}\{ v \hspace{-0.1em}\in\hspace{-0.1em} H^{\frac{3}{2}}(\Omega) \mid \Delta v \in L^2(\Omega)\} \hspace{-0.1em}\to \hspace{-0.1em}\mathcal{S}^1(\mathcal{T}_h)$~by~setting\vspace{-0.5mm}
\begin{align}
    (\mathcal{J}_h v)(\nu) \coloneqq  \begin{cases}
        (\mathcal{I}_h^I v)(\nu) &\text{ if } \nu \in \mathcal{N}_h^{I}\,, \\
        (\mathcal{I}_h v)(\nu) &\text{ if } \nu \in \mathcal{N}_h\setminus \mathcal{N}_h^{I}\,,
    \end{cases}\label{eq:quasi_interpolant_def}
\end{align} 
where we used the trace embedding $E_\Delta(\Omega)\hookrightarrow \smash{H^1(\Gamma_I)}$ (\textit{cf}.\ \cite[Cor.\ 3.7]{BehrndtGesztesyMitrea2025}), 
which has local stability and approximation properties with respect to both  $\mathcal{T}_h$ and $\smash{\mathcal{S}_h^{I}}$:\enlargethispage{7.5mm}\vspace{-0.5mm}

\begin{lemma}\label{prop:properties_Jh}  The following statements apply:
\begin{itemize}[noitemsep,topsep=2pt,leftmargin=!,labelwidth=\widthof{(iii)}]
    \item[(i)]\hypertarget{prop:properties_Jh.i}{} For every $v_h\in \mathcal{S}^1(\mathcal{T}_h)$, there holds $\mathcal{J}_h v_h=v_h$ in $\Omega$;
    \item[(ii)]\hypertarget{prop:properties_Jh.ii}{} For every $S\in \mathcal{S}_h^{I}$, setting $\smash{\omega_S \coloneqq  \operatorname{int}\bigcup\{S' \in \mathcal{S}_h^{I} \mid \overline{S'} \cap \overline{S} \ne \emptyset \}}$, there holds
\begin{align}
  \smash{\|v - \mathcal{J}_h v\|_{S} +  h_S\|\nabla_S \mathcal{J}_h v\|_{S} \lesssim  h_S \| \nabla_{\Gamma} v \|_{\omega_S}}\,;\label{eq:quasi_interpolant_volume_error_S}
\end{align} 
    \item[(iii)]\hypertarget{prop:properties_Jh.iii}{} For every $T\in \mathcal{T}_h$, setting $\smash{\omega_T \coloneqq  \operatorname{int}\bigcup\{T' \in \mathcal{T}_h \mid \overline{T'} \cap \overline{T} \ne \emptyset \}}$, there holds
\begin{align}
  \smash{\|v - \mathcal{J}_h v\|_{T} +  h\|\nabla \mathcal{J}_h v\|_{T} \lesssim  h \| \nabla v \|_{\omega_T}}\,;\label{eq:quasi_interpolant_volume_error}
\end{align} 
    \item[(iv)]\hypertarget{prop:properties_Jh.iv}{} For every  $T\in \mathcal{T}_h$ and $v\in H^s(\omega_T)$, $s \in (\frac{1}{2},2]$, there holds
    \begin{align}
        \smash{\|\nabla v - \nabla\mathcal{J}_h v\|_{T} \lesssim h^{s-1}|v|_{s,\omega_T}}\,. \label{eq:Jh_estimate_H1}
    \end{align} 
\end{itemize}
The implicit constants in $\lesssim$ in \eqref{eq:quasi_interpolant_volume_error_S}--\eqref{eq:Jh_estimate_H1} depend only on the shape-regularity~of~$\{\mathcal{T}_h\}_{h>0}$.
\end{lemma}
\begin{proof}
\textit{ad (\hyperlink{prop:properties_Jh.i}{i}).} For every $v_h\in \mathcal{S}^1(\mathcal{T}_h)$ and $\nu \in \mathcal{N}_h$, there holds $(\mathcal{J}_h v_h)(\nu)=(\mathcal{I}_h^{I} v_h)(\nu) = v_h(\nu)$ if  $\nu \in \mathcal{N}_h^{I}$ and $(\mathcal{J}_h v_h)(\nu)=(\mathcal{I}_h v_h)(\nu) = v_h(\nu)$, so that $\mathcal{J}_h v_h=v_h$ in $\Omega$.

\textit{ad (\hyperlink{prop:properties_Jh.ii}{ii}).} Since $\mathcal{J}_hv\!\!\restriction_{\Gamma_I}=\mathcal{I}_h^{I}v$, using the local stability and interpolation error estimate of $\mathcal{I}_h^{I}$ (\textit{cf}.\ \cite[Lems.\ 1.127, 1.130]{ErnGuermond2004}), we conclude that \eqref{eq:quasi_interpolant_volume_error_S}  applies.

\textit{ad (\hyperlink{prop:properties_Jh.iii}{iii}).}
 \hspace{-0.1mm}Let \hspace{-0.1mm}$v\hspace{-0.15em}\in\hspace{-0.15em} H^1(\Omega)$ \hspace{-0.1mm}be \hspace{-0.1mm}fixed, \hspace{-0.1mm}but \hspace{-0.1mm}arbitrary. \hspace{-0.1mm}We \hspace{-0.1mm}distinguish \hspace{-0.1mm}the \hspace{-0.1mm}cases  \hspace{-0.1mm}${\overline{T} \hspace{-0.15em}\cap\hspace{-0.15em} \overline{\Gamma}_I\hspace{-0.15em} = \hspace{-0.15em}\emptyset}$~and~${\overline{T} \hspace{-0.15em}\cap\hspace{-0.15em} \overline{\Gamma}_I \hspace{-0.15em}\neq\hspace{-0.15em} \emptyset}$:

 \textit{$\bullet$ Case 1: ($\overline{T} \cap \overline{\Gamma}_I = \emptyset$).} In this case,
 we have that $\mathcal{J}_hv\!\!\restriction_T=\mathcal{I}_hv\!\!\restriction_T$, so that \eqref{eq:quasi_interpolant_volume_error} follows from a local stability and interpolation error estimate for $\mathcal{I}_h$ (\textit{cf}.\ \cite[Lems.\ 1.127, 1.130]{ErnGuermond2004}).

 \textit{$\bullet$ Case 2: ($\overline{T} \cap \overline{\Gamma}_I \neq \emptyset$).} In this case, by the node-based norm equivalence (\textit{cf}.\ \cite[Prop.~12.5]{EG21I}),\vspace{-0.5mm}
 \begin{align}\label{prop:properties_Jh.1}
     \|\mathcal{I}_h v - \mathcal{J}_h v\|_{T}^2\sim \max_{\nu^{i}\in \mathcal{N}_h^{I}\cap \overline{T}}{\bigl\{h_{\nu^{i}}^d\vert (\mathcal{I}_h v - \mathcal{I}_h^{I} v)(\nu^{i})\vert^2\bigr\}}\,,\\[-7mm]\notag
 \end{align}
 where, for every $\smash{\nu^{i}}\hspace{-0.1em}\in\hspace{-0.1em} \smash{\mathcal{N}_h^{I}\cap \overline{T}}$, we define $\smash{h_{\smash{\omega_{\smash{\nu^{i}}}}}\hspace{-0.1em}\coloneqq \hspace{-0.1em}\textup{diam}(\omega_{\nu^{i}})}$ and $\smash{\omega_{\nu^{i}} \hspace{-0.1em}\coloneqq   \hspace{-0.1em}\operatorname{int}\bigcup\{\overline{T}\hspace{-0.1em} \in\hspace{-0.1em} \mathcal{T}_h \mid \nu^{i} \hspace{-0.1em}\in\hspace{-0.1em} \overline{T} \}}$,~due~to $\mathcal{I}_hc=c$ in $\Omega$ and $\mathcal{I}_h^{I} c=c$ on $\Gamma_I$ for all $c\in \mathbb{R}$, 
 local inverse estimates (\textit{cf}.\ \cite[Lem.\ 12.1]{EG21I}), a discrete trace estimate (\textit{cf}.\ \cite[Lem.\ 12.8]{EG21I}), a local interpolation error estimate for $\mathcal{I}_h$ (\textit{cf}.\ \cite[Lems.\ 1.127, 1.130]{EG21I}), and the fractional Poincar\'{e}--Steklov inequality (\textit{cf}.\ \cite[Lem.~3.26]{EG21I}),~there~holds\vspace{-0.5mm} 
  \begin{align}\label{prop:properties_Jh.2}
  \begin{aligned} 
    |(\mathcal{I}_h v - \mathcal{I}_h^{I} v)(\nu^{i})|^2 
   & \lesssim  \smash{h^{-d}_{\nu^{i}}} \| \mathcal{I}_h v - (1,v)_{\smash{\omega_{\smash{\nu^{i}}}}}\|_{\smash{\omega_{\smash{\nu^{i}}}}}^2 + \smash{h^{1-d}_{\nu^{i}}} \| \mathcal{I}_h^{I} v - (1,v)_{\smash{\omega_{\smash{\nu^{i}}}}}\|_{\partial\smash{\omega_{\smash{\nu^{i}}}}}^2 \\
     & \lesssim  \smash{h^{-d}_{\nu^{i}}} \| \mathcal{I}_h v \pm v - (1,v)_{\smash{\omega_{\smash{\nu^{i}}}}}\|_{\smash{\omega_{\smash{\nu^{i}}}}}^2 \\
          & \lesssim  
          \smash{h^{2-d}_{\nu^{i}}} \|\nabla v\|_{\omega_{T}}^2\,.
    \end{aligned}
    \end{align}
 Therefore, from \eqref{prop:properties_Jh.2} in \eqref{prop:properties_Jh.1}, we obtain\vspace{-0.5mm}
 \begin{align}
  \|v - \mathcal{J}_h v\|_{T} \lesssim  h \| \nabla v \|_{\omega_T}\,.\label{prop:properties_Jh.3}
\end{align}
 Eventually, from \eqref{prop:properties_Jh.3}, using a stability and interpolation error estimate for $\mathcal{I}_h$ (\textit{cf}.\ \cite[Lems.\ 1.127, 1.130]{ErnGuermond2004}) and a local inverse estimate (\textit{cf}.\ \cite[Lem.\ 12.1]{EG21I}), we infer that\vspace{-0.5mm}
 \begin{align}\label{prop:properties_Jh.4}
    \begin{aligned} 
     \|\nabla\mathcal{J}_h v\|_{T}&\lesssim \|\nabla\mathcal{I}_h v\|_{T}+\|\nabla (\mathcal{I}_h v-\mathcal{J}_h v)\|_{T}
     \\&\lesssim \|\nabla v\|_{\omega_T} + \smash{h^{-1}}\| \mathcal{I}_h v\pm v-\mathcal{J}_h v\|_{T}
     \\&\lesssim \|\nabla v\|_{\omega_T}\,.
     \end{aligned}
 \end{align} 
 In summary, from \eqref{prop:properties_Jh.3} and \eqref{prop:properties_Jh.4}, we conclude that the claimed local stability and interpolation error estimate \eqref{eq:quasi_interpolant_volume_error} applies.\newpage

 \textit{ad (\hyperlink{prop:properties_Jh.iv}{iv}).} \hspace{-0.1mm}Due \hspace{-0.1mm}to \hspace{-0.1mm}$\mathcal{J}_h(1,v)_{\omega_T}\hspace*{-0.15em}=\hspace*{-0.15em}(1,v)_{\omega_T}$, \hspace{-0.1mm}using \hspace{-0.1mm}\eqref{eq:quasi_interpolant_volume_error} \hspace{-0.1mm}and \hspace{-0.1mm}the \hspace{-0.1mm}fractional \hspace{-0.1mm}Poincar\'{e}--Steklov~\hspace{-0.1mm}\mbox{inequality} (\textit{cf}.\ \cite[Lem.\ 3.26]{EG21I}), we conclude that
    \begin{align*}
        \|\nabla (v - \mathcal{J}_h v)\|_{T} &\le  \|\nabla (v - (1,v)_{\omega_T}) \|_{T} + \| \nabla \mathcal{J}_h(v - (1,v)_{\omega_T}) \|_{T} \\&\lesssim \|\nabla (v - (1,v)_{\omega_T})\|_{\omega_T}\\&\lesssim h^{s-1}|v|_{s,\omega_T}\,,
    \end{align*} 
    \textit{i.e.}, the claimed local interpolation error estimate \eqref{eq:Jh_estimate_H1} applies.
\end{proof}  

\section{A positivity- and mean-preserving projection}

\hspace{5mm}In this section, we show that the \emph{mass-lumped $L^2$-projection} (onto $\mathcal{S}^1(\mathcal{S}_h^I)$) $\Pi_h^{\mathrm{m}} \colon L^1(\Gamma_I) \to \mathcal{S}^1(\mathcal{S}_h^I)$, for every $\eta\in L^1(\Gamma_I)$ and $\mu_h\in \mathcal{S}^1(\mathcal{S}_h^I)$ defined by 
\begin{align} \label{eq:lumped_l2_projection}
    (\Pi_h^{\mathrm{m}} \eta, \mu_h)_{\Gamma_I,h} \coloneqq (\eta, \mu_h)_{\Gamma_I}\,,
\end{align}
is positivity- and  mean-preserving, $L^p(\Gamma_I)$- and $H^s(\Gamma_I)$-stable, and satisfies interpolation error estimates in $L^2(\Gamma_I)$ and $H^{-t}(\Gamma_I)$. 


\begin{lemma}[Positivity and  mean preservation] \label{prop:mass_positivity_preserving}
    \hspace{-0.15em}For every $\eta\hspace{-0.15em}\in\hspace{-0.15em} L^1(\Gamma_I)$ with $\eta\hspace{-0.15em}\ge\hspace{-0.15em} 0$~a.e.~on~$\Gamma_I$, there holds $\Pi_h^{\mathrm{m}} \eta \hspace{-0.15em}\ge\hspace{-0.15em} 0$ a.e.\ on~$\Gamma_I$ and $\|\Pi_h^{\mathrm{m}} \eta\|_{1,\Gamma_I}\hspace{-0.15em}=\hspace{-0.15em}\|\eta\|_{1,\Gamma_I}$. In particular,  
    there~holds~$\Pi_h^{\mathrm{m}}(\mathcal{H}_m)\hspace{-0.15em}\subseteq \hspace{-0.15em}\mathcal{H}_m^h$. 
\end{lemma} 

In order to prove Lemma \ref{prop:mass_positivity_preserving}, we first characterize the nodal values of $\Pi_h^{\mathrm{m}}$.

\begin{lemma}\label{lem:lumped_l2_representation}
    For every $\eta \in L^1(\Gamma_I)$ and  $\nu \in \mathcal{N}_h^I$, there holds 
    \begin{align} \label{eq:lumped_l2_representation}
        (\Pi_h^{\mathrm{m}} \eta)(\nu) = 
        (1, \varphi_\nu)_{\Gamma_I}^{-1} (\eta, \varphi_\nu )_{\Gamma_I}\,.
    \end{align}
\end{lemma}

    \begin{proof}
    Let $\eta \in L^1(\Gamma_I)$ and  $\nu \in \mathcal{N}_h^I$ be fixed, but arbitrary. Choosing $\mu_h = \varphi_{\nu}\!\!\restriction_{\Gamma_I}\in \mathcal{S}^1(\mathcal{S}_h^{I})$~in~\eqref{eq:lumped_l2_projection} as well as
        using  $\Pi_h^{\mathrm{m}} \eta \hspace{-0.1em}= \hspace{-0.1em}\sum_{\nu' \in \mathcal{N}_h^I} (\Pi_h^{\mathrm{m}} \eta)(\nu') \varphi_{v'}$ in $\Gamma_I$ and $(\varphi_\nu, \varphi_{\nu'})_{\Gamma_I,h}\hspace{-0.1em}=\hspace{-0.1em}\delta_{\nu\nu'}( 1,\varphi_\nu)_{\Gamma_I,h}$,~we~find~that $(\eta,\varphi_{\nu})_{\Gamma_I}= (\Pi_h^{\mathrm{m}} \eta,\varphi_{\nu})_{\Gamma_I,h}= \sum_{\nu' \in \mathcal{N}_h^I} (\Pi_h^{\mathrm{m}} \eta)(\nu') (\varphi_{\nu'}, \varphi_{\nu})_{\Gamma_I,h} = 
           (\Pi_h^{\mathrm{m}} \eta)(\nu)(1,\varphi_{\nu})_{\Gamma_I}$.
        %
    \end{proof}
    
\begin{proof}[Proof (of Lemma \ref{prop:mass_positivity_preserving}).] Let $\eta \in L^1(\Gamma_I)$ with $\eta\ge 0$ a.e.\ on $\Gamma_I$ be fixed, but arbitrary.

    \textit{ad Positivity preservation.} From $\eta\hspace{-0.15em}\ge\hspace{-0.15em} 0$ a.e.\ on $\Gamma_I$, by Lemma \ref{lem:lumped_l2_representation} and $\varphi_\nu\hspace{-0.15em}\ge\hspace{-0.15em} 0$~on~$\Gamma_I$~for~all~${\nu\hspace{-0.15em} \in\hspace{-0.15em} \mathcal{N}_h^I}$, it follows that 
    $(\Pi_h^{\mathrm{m}} \eta)(\nu) \ge 0$ for all $\nu \in \mathcal{N}_h^I$ and, thus, $\Pi_h^{\mathrm{m}} \eta \ge 0$ on $\Gamma_I$.

    \textit{ad Mean preservation.} From $\eta\ge 0$ a.e.\ on $\Gamma_I$ and $\Pi_h^{\mathrm{m}} \eta \ge 0$ on $\Gamma_I$, choosing $\mu_h=1 \in \mathcal{S}^1(\mathcal{S}_h^I)$ in \eqref{eq:lumped_l2_projection}, it follows that $\| \Pi_h^{\mathrm{m}} \eta \|_{1,\Gamma_I}=(1,\Pi_h^{\mathrm{m}} \eta)_{\Gamma_I}=(1,\Pi_h^{\mathrm{m}} \eta)_{\Gamma_I,h}=(1, \eta)_{\Gamma_I}=\|\eta\|_{1,\Gamma_I}$.
\end{proof}
Lemma \ref{prop:mass_positivity_preserving}, in turn, puts us in the position to show the $L^p(\Gamma_I)$-stability of $\Pi_h^{\mathrm{m}}$~with~\mbox{constant}~1.

\begin{proposition}\label{prop:Lp-stab}
    For every $p\in [1,+\infty]$ and $\eta\in L^p(\Gamma_I)$, there holds
   \begin{align}\label{eq:lumped_l2_nonexpansive}
   \| \Pi_h^{\mathrm{m}} \eta \|_{p,\Gamma_I} \le \| \eta \|_{p,\Gamma_I}\,.
   \end{align}
\end{proposition}
\begin{proof}
    Let $\eta\in L^p(\Gamma_I)$,  $p\in [1,\infty]$, be fixed, but arbitrary. By the Schur--Riesz--Thorin interpolation theorem (\textit{cf}.\ \cite[Thm.\ 4.32]{Brezis2011}), it is sufficient to establish the assertion for $p\in \{1,\infty\}$:

    $\bullet$ \textit{Case $p=1$.} In this case, if we set ${\eta^+\hspace{-0.15em}\coloneqq\hspace{-0.15em} (\eta)_+,\eta^-\hspace{-0.15em}\coloneqq\hspace{-0.15em} (-\eta)_+\hspace{-0.15em}\in\hspace{-0.15em} L^1(\Gamma_I)}$, so that $\eta^+,\eta^- \ge 0$, $\eta = \eta^+ - \eta^-$, and $\vert \eta \vert= \vert\eta^+\vert + \vert\eta^-\vert$ a.e.\ on $\Gamma_I$, using Lemma \ref{prop:mass_positivity_preserving},
    we~find~that $\|\Pi_h^{\mathrm{m}}\eta\|_{1,\Gamma_I} \le \|\Pi_h^{\mathrm{m}}\eta^+\|_{1,\Gamma_I} + \|\Pi_h^{\mathrm{m}}\eta^-\|_{1,\Gamma_I} = \|\eta^+\|_{1,\Gamma_I} + \|\eta^-\|_{1,\Gamma_I} = \| \eta \|_{1,\Gamma_I}$.

    $\bullet$ \textit{Case $p=\infty$.}  In this case, using Lemma \ref{lem:lumped_l2_representation} and that~$ (1,\varphi_\nu)_{\Gamma_I}=\|\varphi_{\nu}\|_{1,\Gamma_I}$ for all $\nu \hspace{-0.15em}\in \hspace{-0.15em}\mathcal{N}_h^{I}$, we obtain $ {\max_{\nu\in \mathcal{N}_h^{I}}{\hspace{-0.15em}\{(\Pi_h^{\mathrm{m}} \eta)(\nu)\}}\hspace{-0.15em}=\hspace{-0.15em}\max_{\nu\in \mathcal{N}_h^{I}}{\hspace{-0.15em}\{\|\varphi_{\nu}\|_{1,\Gamma_I}^{-1} \hspace{-0.1em}(\eta, \varphi_\nu )_{\Gamma_I}\}}\hspace{-0.15em}\leq\hspace{-0.15em} \|\eta\|_{\infty,\Gamma_I} }$. 
\end{proof}

\begin{lemma}\label{lem:lumped_fractional_bound}
For every $\eta\in H^s(\Gamma_I)$, $s\in (0,1]$, there holds
\begin{align} \label{eq:lumped_fractional_bound}
    \| \eta - \Pi_h^{\mathrm{m}} \eta \|_{\Gamma_I}+h^s|\Pi_h^{\mathrm{m}} \eta|_{s,\Gamma_I} \lesssim h^s|\eta|_{s,\Gamma_I}\,,
\end{align}
where the implicit constant in $\lesssim$ depends only on the shape-regularity of $\{\mathcal{T}_h\}_{h>0}$.
\end{lemma}

\begin{proof}
Let $\eta\in H^s(\Gamma_I)$, $s\in (0,1]$, be fixed, but arbitrary. Then, the proof proceeds in two steps: 

\textit{$\bullet$ Step 1: ($s=1$).} For every $\mu_h \in \mathcal{S}^1(\mathcal{S}_h^I)$, 
due to \eqref{eq:lumped_l2_projection} and the $L^2(\Gamma_I)$-self-adjointness~of~$\Pi_h$, we have that $(\Pi_h^{\mathrm{m}} \eta, \mu_h)_{\Gamma_I,h} = (\Pi_h\eta, \mu_h)_{\Gamma_I}$ and, thus, using \cite[Lem.\ 3.9]{Bartels15}, we obtain\vspace{-0.5mm}
\begin{align}\label{lem:lumped_fractional_bound.1} 
   \smash{\vert(\Pi_h^{\mathrm{m}} \eta - \Pi_h\eta,\mu_h)_{\Gamma_I,h}\vert = \vert(\Pi_h\eta,\mu_h)_{\Gamma_I} - (\Pi_h \eta,\mu_h)_{\Gamma_I,h}\vert  \lesssim h \| \nabla_{\Gamma}\Pi_h \eta\|_{\Gamma_I} \|\mu_h\|_{\Gamma_I}\,.} \\[-6mm]\notag
\end{align}
Then, choosing  $\mu_h = \Pi_h^{\mathrm{m}} \eta - \Pi_h\eta\in \smash{\mathcal{S}^1(\mathcal{S}_h^I)}$ in \eqref{lem:lumped_fractional_bound.1} as well as using the $\smash{H^1(\Gamma_I)}$-stability of $\Pi_h$ and \cite[Lem.\ 3.9]{Bartels15}, we find that\vspace{-0.5mm}
\begin{align}\label{lem:lumped_fractional_bound.1.0} 
    \smash{\|\Pi_h^{\mathrm{m}} \eta - \Pi_h\eta\|_{\Gamma_I}\leq \|\Pi_h^{\mathrm{m}} \eta - \Pi_h\eta\|_{\Gamma_I,h}\lesssim h \| \nabla_{\Gamma}\Pi_h \eta\|_{\Gamma_I} \lesssim h \| \nabla_{\Gamma} \eta\|_{\Gamma_I} \,,}\\[-6mm]\notag 
\end{align}
which, by the interpolation error estimate for $\Pi_h$ (\textit{cf}.\ \cite[Prop.\ 22.19]{EG21I}), implies that\vspace{-0.5mm}
\begin{align}\label{lem:lumped_fractional_bound.1.2}
     \smash{\| \eta - \Pi_h^{\mathrm{m}}\eta\|_{\Gamma_I}\leq \|\eta-\Pi_h \eta\|_{\Gamma_I}+\|\Pi_h\eta-\Pi_h^{\mathrm{m}} \eta\|_{\Gamma_I}\lesssim h \| \nabla_{\Gamma} \eta\|_{\Gamma_I} \,.}\\[-6mm]\notag
\end{align}
From \hspace{-0.125mm}\eqref{lem:lumped_fractional_bound.1.0}, \hspace{-0.125mm}in \hspace{-0.125mm}turn, \hspace{-0.125mm}using \hspace{-0.125mm}the \hspace{-0.125mm}$\smash{H^1(\Gamma_I)}$-stability \hspace{-0.125mm}of \hspace{-0.125mm}$\Pi_h$ \hspace{-0.125mm}and \hspace{-0.125mm}a \hspace{-0.125mm}global \hspace{-0.125mm}inverse \hspace{-0.125mm}estimate \hspace{-0.125mm}(\textit{cf}.~\hspace{-0.125mm}\mbox{\cite[\hspace{-0.5mm}Lem.~\hspace{-0.5mm}3.5]{Bartels15}}), we infer that\vspace{-1mm}
\begin{align}\label{lem:lumped_fractional_bound.2}
    \begin{aligned}
    \|\nabla_{\Gamma}\Pi_h^{\mathrm{m}} \eta \|_{\Gamma_I} &\le \|\nabla_{\Gamma}\Pi_h \eta \|_{\Gamma_I} + \|\nabla_{\Gamma}\Pi_h^{\mathrm{m}} \eta - \nabla_{\Gamma}\Pi_h \eta \|_{\Gamma_I} \\
    & \lesssim \|\nabla_{\Gamma}\eta \|_{\Gamma_I} + h^{-1} \|\Pi_h^{\mathrm{m}} \eta - \Pi_h \eta \|_{\Gamma_I} \\
    & \lesssim \|\nabla_{\Gamma}\eta \|_{\Gamma_I}\,.
    \end{aligned}\\[-6mm]\notag
\end{align}
In summary, from \eqref{lem:lumped_fractional_bound.1.2} and \eqref{lem:lumped_fractional_bound.2}, we conclude that the claimed local stability and interpolation error estimate \eqref{eq:lumped_fractional_bound} applies in the case $s=1$.

\textit{$\bullet$ Step 2: ($s<1$).}
First, interpolating between $L^2(\Gamma_I)$ and $H^1(\Gamma_I)$ (\textit{cf}.\ \cite[Prop.\ 14.1.5]{brenner2008mathematical}), from Proposition \ref{prop:Lp-stab}  and \eqref{lem:lumped_fractional_bound.2}, for every $\mu\in H^s(\Gamma_I)$, it follows that\vspace{-0.5mm}
\begin{align} \label{eq:Hs_norm_stability}
    \smash{\|\Pi_h^{\mathrm{m}} \mu \|_{s,\Gamma_I} \lesssim  \|\mu \|_{s,\Gamma_I}\,.}\\[-6mm]\notag
\end{align} 
Then, since $\Pi_h^{\mathrm{m}}(1,\eta )_{\Gamma_I}=(1,\eta )_{\Gamma_I}$ on $\Gamma_I$, from \eqref{eq:Hs_norm_stability} and the fractional Poincar\'{e}--Steklov inequality (\textit{cf}.\ \cite[Lem.\ 3.26]{EG21I}), we infer that\vspace{-0.5mm}
\begin{align}\label{eq:Hs_norm_stability.2}
  \smash{|\Pi_h^{\mathrm{m}} \eta|_{s,\Gamma_I} = |\Pi_h^{\mathrm{m}} (\eta - (1,\eta )_{\Gamma_I})|_{s,\Gamma_I} \lesssim \|\eta - (1,\eta )_{\Gamma_I}\|_{s,\Gamma_I}  \lesssim | \eta|_{s,\Gamma_I}\,.}\\[-6mm]\notag
\end{align}
Since $\|\Pi_h\|_{H^1 \to H^1} \lesssim 1$ and $\|\Pi_h\|_{L^2 \to H^1}\lesssim \smash{h^{-1}}$ (\textit{cf}.\ \cite[Lem.\ 3.5]{Bartels15}), interpolating~between~$L^2(\Gamma_I)$ and $H^1(\Gamma_I)$ (\textit{cf}.\ \cite[Prop.\ 14.1.5]{brenner2008mathematical}), we find that $\| \Pi_h\|_{H^s \to H^1}\leq \|\Pi_h\|_{H^1 \to H^1}^s\|\Pi_h\|_{L^2 \to H^1}^{1-s}\lesssim h^{s-1}$. In other words, for every $\mu\in H^s(\Gamma_I)$, there holds\vspace{-0.5mm}
\begin{align}\label{eq:Hs_norm_stability.3}
    \smash{\| \nabla_{\Gamma}\Pi_h \mu\|_{\Gamma_I} \lesssim h^{s-1}\|\mu\|_{s,\Gamma_I}\,.}\\[-6mm]\notag
\end{align}
Arguing as for \eqref{eq:Hs_norm_stability.2}, but using \eqref{eq:Hs_norm_stability.3} instead of \eqref{eq:Hs_norm_stability}, since $\Pi_h$ preserves constants,~we~infer~that\vspace{-0.5mm}
\begin{align} \label{eq:h1_to_hs_seminorm}
    \smash{\|\nabla_{\Gamma} \Pi_h \eta \|_{\Gamma_I}  
    \le \| \nabla_{\Gamma}\Pi_h (\eta - (1,\eta )_{\Gamma_I}) \|_{\Gamma_I} \lesssim h^{s-1} \vert \eta - (1,\eta )_{\Gamma_I}\vert_{s,\Gamma_I} \lesssim h^{s-1} |\eta|_{s,\Gamma_I}\,.}\\[-6mm]\notag
\end{align}
Using the interpolation error estimate for $\Pi_h$ (\textit{cf}.\ \cite[Prop.\ 22.19]{EG21I}), \eqref{lem:lumped_fractional_bound.1}, and  \eqref{eq:h1_to_hs_seminorm},~we~obtain\vspace{-0.5mm}
\begin{align}\label{eq:h1_to_hs_seminorm.2}
\begin{aligned}
    \|\eta - \Pi_h^{\mathrm{m}} \eta\|_{\Gamma_I} &\le  \|\eta - \Pi_h \eta\|_{\Gamma_I} +  \|\Pi_h\eta - \Pi_h^{\mathrm{m}} \eta\|_{\Gamma_I} \\
    & \lesssim h^s|\eta|_{s,\Gamma_I} + h\|\nabla_{\Gamma}\Pi_h \eta\|_{\Gamma_I} \\
    & \lesssim h^s|\eta|_{s,\Gamma_I}\,.
    \end{aligned}\\[-6mm]\notag
\end{align}
In summary, from \eqref{eq:Hs_norm_stability.2} and \eqref{eq:h1_to_hs_seminorm.2}, we conclude that the claimed local stability and interpolation error estimate \eqref{eq:lumped_fractional_bound} applies in the case $s\in (0,1)$.\enlargethispage{7.5mm}
\end{proof}

\begin{lemma} For every  $\eta \in H^s(\Gamma_I)$, $s \in (0,1]$, and $t \in (0,1]$, there holds\vspace{-0.5mm}
    \begin{align}\label{eq:lumped_negative_norm}
        \smash{\|\eta - \Pi_h^{\mathrm{m}} \eta \|_{-t,\Gamma_I} \lesssim h^{s+t} |\eta|_{s,\Gamma_I}\,,}\\[-6mm]\notag
    \end{align}
    where the implicit constant in $\lesssim$ depends only on the shape-regularity of $\{\mathcal{T}_h\}_{h>0}$.
\end{lemma}

\begin{proof}
    For every   $\psi \in \smash{H^t(\Gamma_I)}$, using \eqref{eq:lumped_l2_projection}, \eqref{lem:std_prop_lumping.2}, \eqref{eq:h1_to_hs_seminorm}, and Lemma \ref{lem:lumped_fractional_bound}, we find that
   \begin{align*}
      |(\eta - \Pi_h^{\mathrm{m}} \eta, \psi)_{\Gamma_I}| &\leq \vert (\eta - \Pi_h^{\mathrm{m}} \eta, \psi - \Pi_h \psi)_{\Gamma_I}\vert + \vert(\Pi_h^{\mathrm{m}} \eta , \Pi_h \psi )_{\Gamma_I,h}-(\Pi_h^{\mathrm{m}} \eta, \Pi_h \psi )_{\Gamma_I}\vert
      \\&\lesssim \|\eta - \Pi_h^{\mathrm{m}} \eta\|_{\Gamma_I}\|\psi - \Pi_h \psi\|_{\Gamma_I} + h^2\|\nabla_{\Gamma}\Pi_h^{\mathrm{m}} \eta\|_{\Gamma_I} \|\nabla_{\Gamma}\Pi_h \psi\|_{\Gamma_I} \\
      &\lesssim h^{s+t} |\eta|_{s,\Gamma_I} |\psi|_{t,\Gamma_I} + h^{1+s-1}h^{1+t-1} | \eta|_{s,\Gamma_I} | \psi|_{t,\Gamma_I} \,.
   \end{align*}
   so that, taking the supremum with respect to 
    $\psi \in \smash{H^t(\Gamma_I)}$ with $| \psi|_{t,\Gamma_I}=1$,~we~conclude~the~claimed negative-norm interpolation error estimate \eqref{eq:lumped_negative_norm}. 
\end{proof}

\section{Block coordinate descent algorithm}

\hspace{5mm}In order to compute a minimizer $(u_h,\smash{\widetilde{\mathtt{d}}_h})\hspace{-0.15em} \in\hspace{-0.15em} \mathcal{S}^{1}(\mathcal{T}_h)\times \mathcal{H}_m^h$ of the discrete heat loss~\mbox{functional}~\eqref{def:Eh}, we resort to a block coordinate descent algorithm that seeks to iteratively solve the discrete optimality conditions  \eqref{prop:discrete_optimality.1}--\eqref{prop:discrete_optimality.3}.

\begin{algorithm}[Block coordinate descent algorithm]\label{algorithm}
    Let $u_h^{(0)}\in \mathcal{S}^{1}(\mathcal{T}_h)$ be an   initial iterate. 
    Then, for every $i\in \mathbb{N}_0$, perform the following iteration loop:
     \begin{enumerate}[noitemsep,topsep=2pt,leftmargin=!,labelwidth=\widthof{(iii)}]
	\item[(i)]\hypertarget{step.i}{}  
    Given $\smash{u_h^{(i)}}\in \mathcal{S}^{1}(\mathcal{T}_h)$ from the previous iteration, 
    compute $\smash{\mathtt{C}_{\smash{u_h^{(i)}}}}> 0$, implicitly defined by
    \begin{align} \label{algorithm.1}
    \mathtt{C}_{\smash{u_h^{(i)}}} = \tfrac{1}{m\beta}\|I_h\{(\vert  u_h^{(i)}-u_{\infty}^h\vert -\mathtt{C}_{\smash{u_h^{(i)}}})_+\}\|_{1,\Gamma_{I}}\,;
    \end{align} 
           \item[(ii)]\hypertarget{step.ii}{} Given $\smash{\mathtt{C}_{\smash{u_h^{(i)}}}}> 0$ from step (\hyperlink{step.i}{i}), 
           compute 
           $\smash{\smash{\widetilde{\mathtt{d}}_h}^{(i)}} \in \mathcal{H}_m^h$ given via 
           \begin{align}\label{algorithm.3}
               \smash{\widetilde{\mathtt{d}}_h}^{(i)}  \coloneqq \smash{\tfrac{1}{\mathtt{C}_{\smash{u_h^{(i)}}}\beta}}I_h\{(|u_h^{(i)}-u_{\infty}^h|-\mathtt{C}_{\smash{u_h^{(i)}}})_+\}\quad \text{ on }\Gamma_{I}\,;
           \end{align}
           \item[(iii)]\hypertarget{step.iii}{} Given $\smash{\smash{\widetilde{\mathtt{d}}_h}^{(i)}} \in \mathcal{H}_m^h$ from step (\hyperlink{step.ii}{ii}), compute $\smash{u_h^{(i+1)}} \in \mathcal{S}^{1}(\mathcal{T}_h)$ such that for every $v_h\in \mathcal{S}^{1}(\mathcal{T}_h)$, there holds
           \begin{align}\label{algorithm.4}
            \hspace*{-2mm}\kappa(\nabla u_h^{(i+1)},\nabla v_h)_{\Omega} + \beta((1+\beta \smash{\widetilde{\mathtt{d}}_h}^{(i)})^{-1}(u_h^{(i+1)}-u_{\infty}^h), v_h)_{\smash{\Gamma_{I}},h}  = (f_h, v_h)_{\Omega}+(g_h,v_h)_{\Gamma_N}\,.
           \end{align}
    \end{enumerate}   
\end{algorithm}

The following theorem establishes the well-posedness and linear convergence of 
Algorithm~\ref{algorithm}.

\begin{theorem}[Well-posedness and  linear convergence]\label{thm:algorithm}
    The following statements apply: 
    \begin{itemize}[noitemsep,topsep=2pt,leftmargin=!,labelwidth=\widthof{(ii)}]
        \item[(i)]\hypertarget{thm:algorithm.i}{}  Algorithm \ref{algorithm} is well-posed, \textit{i.e.}, for every  $i\in \mathbb{N}_0$, the iterate $(u_h^{(i+1)},\mathtt{C}_{\smash{u_h^{(i)}}},\smash{\widetilde{\mathtt{d}}_h}^{(i)})\in \mathcal{S}^1(\mathcal{T}_h)\times \mathbb{R}_{>0}\times \mathcal{H}_m^h$ is computable according to steps (\hyperlink{step.i}{i})--(\hyperlink{step.iii}{iii});
        \item[(ii)]\hypertarget{thm:algorithm.ii}{} If, in addition, the discrete non-degeneracy condition
        \begin{align}\label{thm:algorithm.0.1}
            \delta_{\star,h}
    \coloneqq
    \operatorname{dist}_{\mathcal{S}^1(\mathcal{S}_h^{I})}\bigl(1, \tfrac{\beta (u_h-u_\infty^h)}{1+\beta\smash{\widetilde{\mathtt{d}}_h}}\smash{\widetilde{\mathcal{S}}}^1(\mathcal{S}_h^{I})\bigr)>0\,,
        \end{align}
        where $\smash{\widetilde{\mathcal{S}}}^1(\mathcal{S}_h^{I})\coloneqq \{\eta_h\in \mathcal{S}^1(\mathcal{S}_h^{I})\mid (\eta_h,1)_{\Gamma_I,h}=0\}$, is satisfied, 
        Algorithm \ref{algorithm}~is~\mbox{linearly}~convergent, \textit{i.e.}, for every $i\in \mathbb{N}_0$, there holds
    \begin{align}\label{thm:algorithm.0.2}
        \begin{aligned} 
         \sigma_h\|u_h^{(i)}-u_h\|_{H^1(\Omega)}^2&\leq E_h(u_h^{(i)},\smash{\widetilde{\mathtt{d}}_h}^{(i)}) - E_h(u_h,\smash{\widetilde{\mathtt{d}}_h})\\& \le (1 - \tfrac{\sigma_h}{8L})^i \bigl(E_h(u_h^0,\smash{\widetilde{\mathtt{d}}_h}^0) - E_h(u_h,\smash{\widetilde{\mathtt{d}}_h})\bigr)\,,
        \end{aligned}
    \end{align}
    where $\sigma_h>0$ is a constant that may deteriorate as $h\to 0^+$ and $L\coloneqq \smash{\kappa+c_{\mathrm{Tr}}^2\beta(d+1)}>0$.
    \end{itemize}
\end{theorem}
\begin{proof} \emph{ad (\hyperlink{thm:algorithm.i}{i}).} 
First, the well-posedness of step (\hyperlink{step.i}{i}), by analogy with \cite[Lem.\ 4.1]{PietraNitschScalaTrombetti2021}, follows from the intermediate value theorem. Second,  the well-posedness of step (\hyperlink{step.ii}{ii}) follows from $\mathtt{C}_{\smash{u_h^{(i)}}}>0$, which itself follows from step (\hyperlink{step.iii}{iii}) of the previous iteration arguing as in the proof of Lemma \ref{lem:positivity_Cu_h}. Third, the well-posedness of step (\hyperlink{step.iii}{iii}) follows from the Lax--Milgram lemma (\textit{cf}.\ \cite[Lem.\ 2.2]{ErnGuermond2004}).

\emph{ad (\hyperlink{thm:algorithm.ii}{ii}).} To begin with, we introduce the extended energy functional $\widehat{E}_h\colon \mathcal{S}^1(\mathcal{T}_h)\times \mathcal{S}^1(\mathcal{S}_h^{I})\to \mathbb{R}$, for every $(v_h,\eta_h)\in \mathcal{S}^1(\mathcal{T}_h)\times \mathcal{S}^1(\mathcal{S}_h^{I})$ defined by
\begin{align}
    \widehat{E}_h(v_h,\eta_h)\coloneqq\tfrac{\kappa}{2}\|\nabla v_h\|^2_{\Omega}  + \tfrac{\beta}{2}\|(1 + \beta \eta_h)^{-\smash{\frac{1}{2}}}(v_h-u_{\infty}^h)\|_{\Gamma_{I},h}^2 - (f_h, v_h)_{\Omega}-(g_h,v_h)_{\Gamma_N} \,,\label{thm:algorithm.0}
\end{align}
which satisfies the following quadratic growth and Lipschitz gradient properties:

\textit{$\bullet$ Quadratic growth.}~Analogously~to~the proof of Corollary \ref{cor:quadratic_growth_new}, for every $(v_h,\eta_h)\in \mathcal{S}^1(\mathcal{T}_h)\times \mathcal{H}_m^h$, we find that
\begin{align}\label{eq:quadratic_growth_Eh}
    \smash{\widehat{E}_h(v_h,\eta_h)-\widehat{E}_h(u_h,\smash{\widetilde{\mathtt{d}}_h})\ge \sigma_h(\eta_h) \|v_h-u_h\|_{H^1(\Omega)}^2\,,}
\end{align}
where $\sigma_h(\eta_h)>0$ depends on $\mathcal{T}_h$, $\Omega$, $\Gamma_I$, $\kappa$, $\beta$,
$\|\smash{\widetilde{\mathtt{d}}_h}\|_{\infty,\Gamma_I}$,
$\|\eta_h\|_{\infty,\Gamma_I}$, and $\delta_{\star,h}^{-1}$.~By~a~global~inverse estimate (\textit{cf}.\ \cite[Lem.\ 3.5]{Bartels15}), we have that
$\|\eta_h\|_{\infty,\Gamma_I}\lesssim h^{-d+1}\|\eta_h\|_{1,\Gamma_I}=h^{-d+1}m$~for~all~$\eta_h\in  \mathcal{H}_m^h$, so that \eqref{eq:quadratic_growth_Eh} applies with $\sigma_h(\eta_h)$~replaced~by~${\sigma_h\coloneqq\inf_{\eta_h\in \mathcal{H}_m^h}{\{\sigma_h(\eta_h)\}}>0}$.

\textit{$\bullet$ Lipschitz gradient.}
For every $(v_h,\eta_h)\in \mathcal{S}^1(\mathcal{T}_h)\times \mathcal{S}^1(\mathcal{S}_h^{I})$ and $w_h\in \mathcal{S}^1(\mathcal{T}_h)$, using~\eqref{lem:std_prop_lumping.1}, we find that
\begin{align}\label{eq:Lipschitz_gradient_Eh}
\begin{aligned} 
&\|\mathrm{D}_{v_h}\{\widehat{E}_h(v_h+w_h,\eta_h)-\widehat{E}_h(v_h,\eta_h)\}\|_{(\mathcal{S}^1(\mathcal{T}_h))^*}\\&=\sup_{\widetilde{w}_h\in \mathcal{S}^1(\mathcal{T}_h)\,:\,\|\widetilde{w}_h\|_{H^1(\Omega)}\leq 1}{\bigl\{\kappa (\nabla w_h,\nabla \widetilde{w}_h)_{\Omega}+\beta((1+\beta \eta_h)^{-1}w_h,\widetilde{w}_h)_{\Gamma_I,h}\bigr\}}
\\[-1mm]&\leq \sup_{\widetilde{w}_h\in \mathcal{S}^1(\mathcal{T}_h)\,:\,\|\widetilde{w}_h\|_{H^1(\Omega)}\leq 1}{\bigl\{\kappa \|\nabla w_h\|_{\Omega}\|\nabla\widetilde{w}_h\|_{\Omega}+\beta(d+1)\|w_h\|_{\Gamma_I}\|\widetilde{w}_h\|_{\Gamma_I}\bigr\}}
\\[-1mm]&\leq L\| w_h\|_{H^1(\Omega)}\,.
\end{aligned}
\end{align}

Based on the quadratic growth and Lipschitz gradient properties, following~\mbox{arguments}~of~\cite{Both:2022}, we are now in the position to establish the claimed linear convergence of Algorithm \ref{algorithm}:

Then, from the Fr\'echet differentiability and convexity of $\widehat{E}_h\colon \hspace{-0.125em}\mathcal{S}^1(\mathcal{T}_h)\times \mathcal{S}^1(\mathcal{S}_h^{I})\hspace{-0.125em}\to\hspace{-0.125em} \mathbb{R}$~as~well~as~that $\langle \mathrm{D}_{\eta_h} \widehat{E}_h(u_h^{(i)},\smash{\widetilde{\mathtt{d}}_h}^{(i)}), \smash{\widetilde{\mathtt{d}}_h}^{(i)} - \smash{\widetilde{\mathtt{d}}_h} \rangle_{\mathcal{S}^1(\mathcal{S}_h^{I})} \le 0$, since  $-\mathrm{D}_{\eta_h} \widehat{E}_h(u_h^{(i)},\smash{\widetilde{\mathtt{d}}_h}^{(i)}) \in \partial_{\eta_h} I_{\mathcal{H}_m^h}(\smash{\widetilde{\mathtt{d}}_h}^{(i)})$,  for all $i\in \mathbb{N}_0$, for every $i\in \mathbb{N}_0$, it follows that
\begin{align}\label{thm:algorithm.1}
\begin{aligned} 
 \widehat{E}_h(u_h^{(i)},\smash{\widetilde{\mathtt{d}}_h}^{(i)}) - \widehat{E}_h(u_h,\smash{\widetilde{\mathtt{d}}_h}) 
    & \le \langle \mathrm{D}_{v_h} \widehat{E}_h(u_h^{(i)},\smash{\widetilde{\mathtt{d}}_h}^{(i)}), u_h^{(i)} - u_h\rangle_{\mathcal{S}^1(\mathcal{T}_h)}\\[-0.5mm]&\quad + \langle \mathrm{D}_{\eta_h} \widehat{E}_h(u_h^{(i)},\smash{\widetilde{\mathtt{d}}_h}^{(i)}),\smash{\widetilde{\mathtt{d}}_h}^{(i)}-\smash{\widetilde{\mathtt{d}}_h} \rangle_{\mathcal{S}^1(\mathcal{S}_h^{I})}
    \\[-0.5mm]&\leq \langle \mathrm{D}_{v_h} \widehat{E}_h(u_h^{(i)},\smash{\widetilde{\mathtt{d}}_h}^{(i)}), u_h^{(i)} - u_h\rangle_{\mathcal{S}^1(\mathcal{T}_h)}\,,
\end{aligned}
\end{align}
where, for every $\gamma \in (0,1]$, due to \eqref{eq:quadratic_growth_Eh}, \eqref{eq:Lipschitz_gradient_Eh}, the Fr\'echet differentiability and convexity of $\widehat{E}_h(\cdot,\smash{\widetilde{\mathtt{d}}_h}^{(i)})\colon \mathcal{S}^1(\mathcal{T}_h)\to \mathbb{R}$, and the minimality of $u_h^{(i+1)}\in \mathcal{S}^1(\mathcal{T}_h)$~for~the~latter,~we~have~that
\begin{align}\label{thm:algorithm.2}
\begin{aligned}
    &\langle \mathrm{D}_{v_h} \widehat{E}_h(u_h^{(i)},\smash{\widetilde{\mathtt{d}}_h}^{(i)}), u_h^{(i)} - u_h\rangle_{\mathcal{S}^1(\mathcal{T}_h)}  \\[-0.5mm]&=  \langle \mathrm{D}_{v_h} \widehat{E}_h(u_h^{(i)},\smash{\widetilde{\mathtt{d}}_h}^{(i)}) - \mathrm{D}_{v_h} \widehat{E}_h(u_h^{(i)} + \gamma(u_h - u_h^{(i)}),\smash{\widetilde{\mathtt{d}}_h}^{(i)}), u_h^{(i)} - u_h\rangle_{\mathcal{S}^1(\mathcal{T}_h)} \\[-0.5mm]&\quad+ \tfrac{1}{\gamma}\langle  \mathrm{D}_{v_h} \widehat{E}_h(u_h^{(i)} + \gamma(u_h - u_h^{(i)}),\smash{\widetilde{\mathtt{d}}_h}^{(i)}), \gamma(u_h^{(i)} - u_h)\rangle_{\mathcal{S}^1(\mathcal{T}_h)} \\[-0.5mm]
& \le L\gamma \| u_h^{(i)} - u_h\|_{H^1(\Omega)}^2 + \tfrac{1}{\gamma}\bigl(\widehat{E}_h(u_h^{(i)},\smash{\widetilde{\mathtt{d}}_h}^{(i)}) - \widehat{E}_h(u_h^{(i)} + \gamma(u_h - u_h^{(i)}), \smash{\widetilde{\mathtt{d}}_h}^{(i)})\bigr) \\[-0.5mm]
& \le \tfrac{2L\gamma}{\sigma_h} \bigl(\widehat{E}_h(u_h^{(i)},\smash{\widetilde{\mathtt{d}}_h}^{(i)}) -\widehat{E}_h(u_h,\smash{\widetilde{\mathtt{d}}_h})\bigr) + \tfrac{1}{\gamma}\bigl(\widehat{E}_h(u_h^{(i)},\smash{\widetilde{\mathtt{d}}_h}^{(i)}) - \widehat{E}_h(u_h^{(i+1)}, \smash{\widetilde{\mathtt{d}}_h}^{(i)})\bigr)\,.
\end{aligned}
\end{align}
Using \eqref{thm:algorithm.2} with $\gamma = \smash{\frac{\sigma_h}{4L}}\in (0, \frac{1}{4}]$ in \eqref{thm:algorithm.1}  and  rearranging the resulting estimate,~for~every~$i\in \mathbb{N}_0$, we find that
\begin{align*}
   \widehat{E}_h(u_h^{(i+1)},\smash{\widetilde{\mathtt{d}}_h}^{(i)})-\widehat{E}_h( u_h,\smash{\widetilde{\mathtt{d}}_h}) \le (1 - \tfrac{\sigma_h}{8L}) (\widehat{E}_h(u_h^{(i)},\smash{\widetilde{\mathtt{d}}_h}^{(i)}) -\widehat{E}_h( u_h,\smash{\widetilde{\mathtt{d}}_h}))\,.
\end{align*}
%
Eventually, noting that $\smash{\widehat{E}_h(u_h^{(i)},\smash{\widetilde{\mathtt{d}}_h}^{(i)})\leq \widehat{E}_h(u_h^{(i)},\smash{\widetilde{\mathtt{d}}_h}^{(i-1)}) \leq
    \widehat{E}_h(u_h^{(i-1)},\smash{\widetilde{\mathtt{d}}_h}^{(i-1)})}$ for all $i\in \mathbb{N}$, 
%
for every $i\in \mathbb{N}_0$, we conclude that
\begin{align*}
    \widehat{E}_h(u_h^{(i+1)},\smash{\widetilde{\mathtt{d}}_h}^{(i+1)}) - \widehat{E}_h(u_h,\smash{\widetilde{\mathtt{d}}_h}) &\le \widehat{E}_h(u_h^{(i+1)},\smash{\widetilde{\mathtt{d}}_h}^{(i)}) - \widehat{E}_h(u_h,\smash{\widetilde{\mathtt{d}}_h}) \\[-0.5mm] &\le (1 - \tfrac{\sigma_h}{8L}) \bigl(\widehat{E}_h(u_h^{(i)},\smash{\widetilde{\mathtt{d}}_h}^{(i)}) - \widehat{E}_h(u_h,\smash{\widetilde{\mathtt{d}}_h})\bigr) \\[-2.5mm]
    & \,\,\vdots \\[-2mm]
    & \le (1 - \tfrac{\sigma_h}{8L})^{i+1} \bigl(\widehat{E}_h(u_h^{(0)},\smash{\widetilde{\mathtt{d}}_h}^{(0)}) - \widehat{E}_h(u_h,\smash{\widetilde{\mathtt{d}}_h})\bigr)\,,
\end{align*}
which, together with \eqref{eq:quadratic_growth_Eh}, is the claimed linear convergence estimate \eqref{thm:algorithm.0.2} for Algorithm \ref{algorithm}.
\end{proof}

\begin{remark}[further comments on Algorithm \ref{algorithm}]
    \begin{itemize}[noitemsep,topsep=2pt,leftmargin=!,labelwidth=\widthof{(ii)}]
        \item[(i)] \hspace{-0.1mm}\emph{Computation \hspace{-0.1mm}of \hspace{-0.1mm}$\mathtt{C}_{\smash{u_h^{(i)}}}$.} \hspace{-0.1mm}Although~\hspace{-0.1mm}in~\hspace{-0.1mm}the proof of Theorem \ref{thm:algorithm}(i), the existence of a constant $\mathtt{C}_{\smash{u_h^{(i)}}}$, solving \eqref{algorithm.1},~was~shown~based~on the intermediate value theorem, there are several ways to obtain this constant constructively:
        \begin{itemize}[noitemsep,topsep=2pt,leftmargin=!,labelwidth=\widthof{(i.c)}]
            \item[(i.a)] In general,  since the intermediate value theorem is applicable, one may resort to the bisection method, which converges globally with linear rate and reduces
the error by a factor of $\frac{1}{2}$ in each iteration.
            \item[(i.b)] If a suitable initial guess is available (for example, obtained via the
            bisection method), one may employ a semi-smooth Newton method, which
converges locally superlinearly.

            \item[(i.c)]
If $|\Gamma_I|<m\beta$, then the associated fixed-point mapping is a
contraction; consequently, by the Banach fixed-point theorem, one may resort to
a Picard iteration, which converges globally with linear rate and reduces
the error by a factor of $\smash{\frac{|\Gamma_I|}{m\beta}}$ in each iteration.
        \end{itemize} 

        \item[(ii)] \emph{Convergence of $\smash{\{(\mathtt{C}_{\smash{u_h^{(i)}}},\smash{\widetilde{\mathtt{d}}_h}^{(i)})\}_{i\in \mathbb{N}_0}}$.} If $|\Gamma_I|<m\beta$, arguing similarly to the proof of Lemma~\ref{lem:C_u_d_u_estimate}, for every $i\in \mathbb{N}_0$, we find that\vspace{-1mm}
        \begin{subequations} 
        \begin{align}
            |\mathtt{C}_{\smash{u_h^{(i)}}} - \mathtt{C}_{\smash{u_h}}|  &\leq \smash{\tfrac{1}{m\beta-\vert \Gamma_I\vert}}\|u_h^{(i)}-u_h\|_{1,\Gamma_I}\,,\\
            \|\smash{\widetilde{\mathtt{d}}_h}^{(i)}-\smash{\widetilde{\mathtt{d}}_h}\|_{1,\Gamma_I}&\leq \smash{\tfrac{2m}{\mathtt{C}_{\smash{u_h^{(i)}}}}}\smash{\tfrac{1}{m\beta-\vert \Gamma_I\vert }}\|u_h^{(i)}-u_h\|_{1,\Gamma_I}\,,
        \end{align} 
        so that the linear convergence estimate in Theorem \ref{thm:algorithm}(\hyperlink{thm:algorithm.ii}{ii}) carries over.
        \end{subequations}
    \end{itemize}
\end{remark}

\begin{remark}[Sufficient conditions for the discrete non-degeneracy condition \eqref{thm:algorithm.0.1}]
\label{rem:nondegeneracy_conditions_discrete}
The discrete non-degeneracy condition
\eqref{thm:algorithm.0.1}
excludes the possibility that constant modes of $v_h-u_h$ can be
compensated by admissible variations of $\smash{\widetilde{\mathtt{d}}_h}$. It is satisfied,
for instance, under the following assumptions:

\begin{itemize}[noitemsep,topsep=2pt,leftmargin=!,labelwidth=\widthof{(iii)}]

\item[(i)]\hypertarget{rem:nondegeneracy_conditions_discrete.i}{}
If there exists a vertex
$\nu\in \mathcal{N}_h^I$ such that $u_h(\nu)=u_\infty^h(\nu)$, then $\smash{\frac{\beta(u_h(\nu)-u_\infty^h(\nu))}
    {1+\beta\smash{\widetilde{\mathtt{d}}_h}(\nu)}
    \xi_h(\nu)}
    =
    0$ for all 
$\xi_h\in \smash{\widetilde{\mathcal{S}}}^1(\mathcal{S}_h^{I})$, so that 
\begin{align}\label{rem:nondegeneracy_conditions_discrete.1}
    1
    \notin
    \tfrac{\beta(u_h-u_\infty^h)}
    {1+\beta\smash{\widetilde{\mathtt{d}}_h}}
    \smash{\widetilde{\mathcal{S}}}^1(\mathcal{S}_h^{I})\,,
\end{align}
which is equivalent to
$\delta_{\star,h}>0$;

\item[(ii)]\hypertarget{rem:nondegeneracy_conditions_discrete.ii}{}
If $u_h(\nu)-u_\infty^h(\nu)\neq0$ for all $\nu\in \mathcal{N}_h^I$,  
then $\frac{\beta(u_h-u_\infty^h)}
    {1+\beta\smash{\widetilde{\mathtt{d}}_h}}
    \smash{\widetilde{\mathcal{S}}}^1(\mathcal{S}_h^{I})$ 
is a closed hyperplane~in~the~finite-dimensional  trace space with
normal vector $\frac{1+\beta\smash{\widetilde{\mathtt{d}}_h}}
    {\beta(u_h-u_\infty^h)}$.  
If, in addition, $(
        \frac{1+\beta\smash{\widetilde{\mathtt{d}}_h}}
        {u_h-u_\infty^h},
        1 )_{\Gamma_I,h}
    \neq0$, 
then \eqref{rem:nondegeneracy_conditions_discrete.1}
applies, which is equivalent to
$\delta_{\star,h}>0$;

A sufficient condition for
(\hyperlink{rem:nondegeneracy_conditions_discrete.ii}{ii})
is that $u_h-u_\infty^h
    \ge c_0$  or $u_h-u_\infty^h
    \le -c_0$ on $\Gamma_I$~for~some~$c_0>0$;

\item[(iii)]\hypertarget{rem:nondegeneracy_conditions_discrete.iii}{}
If
$u_\infty^h=\mathrm{const}$ on $\Gamma_I$,
$f_h\ge0$ in $\Omega$, 
$g_h\ge0$ on $\Gamma_N$, and
 Assumption~\ref{assum:nonobtuse_mesh}
is satisfied,  then the discrete maximum principle implies $u_h-u_\infty^h\ge0$ on $\Gamma_I$, which follows
by testing the discrete optimality condition \eqref{prop:discrete_optimality.3}
with $v_h= -I_h\{(u_h-u_\infty^h)_-\}\in \mathcal{S}^1(\mathcal{T}_h)$ 
and~using~Lemma~\ref{lem:discrete_chainrule}, resulting in
\begin{align}
    \kappa
    \|
        \nabla
        I_h\{(u_h-u_\infty^h)_-\}
    \|_\Omega^2
    +
    \beta
    \bigl\|
         I_h\{(1+\beta\smash{\widetilde{\mathtt{d}}_h})^{-\smash{\frac12}}
       (u_h-u_\infty^h)_-\}
    \bigr\|_{\Gamma_I,h}^2
    \le0 \,.
\end{align}
Consequently, if $ u_h-u_\infty^h>0$
    on $\Gamma_I$,  
then $\smash{(\frac{1+\beta\smash{\widetilde{\mathtt{d}}_h}}
    {u_h-u_\infty^h},1)_{\Gamma_I,h}
  }  >0$, so that (\hyperlink{rem:nondegeneracy_conditions.ii}{ii}) applies. Otherwise, (\hyperlink{rem:nondegeneracy_conditions.i}{i}) applies and, therefore, $\delta_{\star,h}>0$.
\end{itemize}
\end{remark}

{\setlength{\bibsep}{0pt plus 0.0ex}\small
		
		\bibliographystyle{aomplain}
		\bibliography{references}
		
	}
    
\end{document}